\documentclass[11pt]{article}
\usepackage{smile}

\usepackage{fullpage}
\usepackage{lscape}
\usepackage{bigints}
\usepackage{framed}
\usepackage{mdframed}
\usepackage{enumerate}
\usepackage[inline]{enumitem}
\usepackage[T1]{fontenc}
\usepackage{moresize}
\usepackage{bm}
\usepackage{bbm}
\usepackage{dsfont}
\usepackage{amsmath}
\usepackage{amssymb}
\usepackage{amsthm}
\usepackage{amsfonts}
\usepackage{stmaryrd}
\usepackage{array}
\usepackage{mathrsfs}
\usepackage{mathtools} 
\usepackage{extarrows}
\usepackage{stackrel}
\usepackage{relsize,exscale}
\usepackage{scalerel}
\usepackage[nodisplayskipstretch]{setspace}
\usepackage{color}
\usepackage[usenames,dvipsnames]{xcolor}
\usepackage{cancel}
\usepackage{soul}
\usepackage{undertilde}
\usepackage{xfrac}
\usepackage{siunitx}
\usepackage{graphicx}
\usepackage{float}
\usepackage{rotating}
\usepackage{subcaption}
\usepackage{overpic}
\usepackage[all]{xy}
\usepackage{tikz}
\usetikzlibrary{arrows,matrix,positioning,calc,automata,patterns}
\usepackage{booktabs}
\usepackage{dcolumn}
\usepackage{multirow}
\usepackage{diagbox}
\usepackage{tabularx}
\usepackage{verbatim}
\usepackage{listings}
\usepackage[ruled,vlined]{algorithm2e}
\usepackage{fancyvrb}
\usepackage{hyperref}
\usepackage[round]{natbib}
\usepackage{sectsty}

\hypersetup{
    bookmarks=true,         
    unicode=false,          
    pdftoolbar=true,        
    pdfmenubar=true,        
    pdffitwindow=false,     
    pdfstartview={FitH},    
    pdftitle={My title},    
    pdfauthor={Author},     
    pdfsubject={Subject},   
    pdfcreator={Creator},   
    pdfproducer={Producer}, 
    pdfkeywords={key1, key2}, 
    pdfnewwindow=true,      
    colorlinks=true,        
    linkcolor=blue,         
    citecolor=blue,         
    filecolor=blue,         
    urlcolor=cyan           
}

\usepackage{stackengine}
\stackMath
\newcommand\tenq[2][1]{%
\def\useanchorwidth{T}%
\ifnum#1>1%
\stackunder[0pt]{\tenq[\numexpr#1-1\relax]{#2}}{\!\scriptscriptstyle\thicksim}%
\else%
\stackunder[1pt]{#2}{\!\scriptstyle\thicksim}%
\fi%
}

\makeatletter
\DeclareRobustCommand\widecheck[1]{{\mathpalette\@widecheck{#1}}}
\def\@widecheck#1#2{%
    \setbox\z@\hbox{\m@th$#1#2$}%
    \setbox\tw@\hbox{\m@th$#1%
       \widehat{%
          \vrule\@width\z@\@height\ht\z@
          \vrule\@height\z@\@width\wd\z@}$}%
    \dp\tw@-\ht\z@
    \@tempdima\ht\z@ \advance\@tempdima2\ht\tw@ \divide\@tempdima\thr@@
    \setbox\tw@\hbox{%
       \raise\@tempdima\hbox{\scalebox{1}[-1]{\lower\@tempdima\box
\tw@}}}%
    {\ooalign{\box\tw@ \cr \box\z@}}}
\makeatother

\def\given{\,|\,}
\def\biggiven{\,\big{|}\,}
\def\Biggiven{\,\Big{|}\,}
\def\tr{\mathop{\text{tr}}\kern.2ex}
\def\tZ{{\tilde Z}}
\def\tX{{\tilde X}}

\def\P{{\mathrm P}}
\def\Q{{\mathrm Q}}

\def\E{{\mathrm E}}
\def\R{{\mathbbm R}}
\def\Z{{\mathbbm Z}}
\def\N{{\mathbbm N}}

\def\d{{\mathrm d}}
\newcommand{\Zeta}{\Delta_0}

\newcommand{\dCov}{\mathrm{dCov}}

\renewcommand{\Pr}{\mathrm{P}}

\newcommand{\TV}{\mathrm{TV}}

\newcommand{\HL}{\mathrm{HL}}

\newcommand{\indep}{\perp \!\!\!\perp}

\newcommand{\zahl}[1]{\llbracket #1\rrbracket}
\newcommand\yestag{\addtocounter{equation}{1}\tag{\theequation}}
\newcolumntype{L}[1]{>{\raggedright\let\newline\\\arraybackslash\hspace{0pt}}m{#1}}
\newcolumntype{C}[1]{>{  \centering\let\newline\\\arraybackslash\hspace{0pt}}m{#1}}
\newcolumntype{R}[1]{>{ \raggedleft\let\newline\\\arraybackslash\hspace{0pt}}m{#1}}
\newcolumntype{d}[1]{D{.}{.}{#1}}
\newcolumntype{H}{>{\setbox0=\hbox\bgroup}c<{\egroup}@{}}
\newcolumntype{Z}{>{\setbox0=\hbox\bgroup}c<{\egroup}@{\hspace*{-\tabcolsep}}}
\newcolumntype{b}{X}
\newcolumntype{s}{>{\hsize=.5\hsize}X}

\numberwithin{equation}{section}

\newtheorem{theorem}{Theorem}[section]
\newtheorem{lemma}{Lemma}[section]
\newtheorem{proposition}{Proposition}[section]
\newtheorem{assumption}{Assumption}[section]
\newtheorem{corollary}{Corollary}[section]
\newtheorem{innercustomgeneric}{\customgenericname}
\providecommand{\customgenericname}{}
\newcommand{\newcustomtheorem}[2]{%
  \newenvironment{#1}[1]
  {%
   \renewcommand\customgenericname{#2}%
   \renewcommand\theinnercustomgeneric{##1}%
   \innercustomgeneric
  }
  {\endinnercustomgeneric}
}
\newcustomtheorem{customdefinition}{Definition}
\newcustomtheorem{customdefinitions}{Definitions}
\newcustomtheorem{customtheorem}{Theorem}
\newcustomtheorem{customassumption}{Assumption}
\newcustomtheorem{customlemma}{Lemma}
\newcustomtheorem{customexample}{Example}
\theoremstyle{definition}
\newtheorem{definition}{Definition}[section]
\newtheorem{example}{Example}[section]
\newtheorem{remark}{Remark}[section]

\graphicspath{{./fig3/}}

\allowdisplaybreaks

\begin{document}

\setlength{\abovedisplayskip}{5pt}
\setlength{\belowdisplayskip}{5pt}
\setlength{\abovedisplayshortskip}{5pt}
\setlength{\belowdisplayshortskip}{5pt}
\hypersetup{colorlinks,breaklinks,urlcolor=blue,linkcolor=blue}

\title{\LARGE On Azadkia--Chatterjee's conditional dependence coefficient}

\author{
Hongjian Shi\thanks{Department of Mathematics, Technical University of Munich, 85748 Garching bei M\"unchen, Germany; e-mail: {\tt hongjian.shi@tum.de}},~~
Mathias Drton\thanks{Department of Mathematics, Technical University
  of Munich, 85748 Garching bei M\"unchen, Germany; e-mail: {\tt mathias.drton@tum.de}},~~and~
Fang Han\thanks{Department of Statistics, University of Washington, Seattle, WA 98195, USA; e-mail: {\tt fanghan@uw.edu}}.
}

\date{}

\maketitle


\begin{abstract}
  In recent work, \citet{MR4352523} laid out an ingenious
  approach to defining consistent measures of conditional dependence.
  Their fully nonparametric approach forms statistics based on ranks
  and nearest neighbor graphs.  The appealing nonparametric
  consistency of the resulting conditional dependence
  measure and the associated empirical conditional dependence coefficient
  has quickly prompted follow-up work that seeks to study its statistical
  efficiency.
  In this paper, we take up
  the framework of conditional randomization tests (CRT) for
  conditional independence and conduct a power analysis that considers
  two types of local alternatives, namely, parametric quadratic mean
  differentiable alternatives and nonparametric H\"older smooth
  alternatives. Our local power analysis shows that conditional independence tests
  using the Azadkia--Chatterjee coefficient remain inefficient even
  when aided with the CRT framework, and serves as motivation to
  develop variants of the approach; cf. \citet{lin2021boosting}.
  As a byproduct, we resolve a conjecture of Azadkia and Chatterjee by
  proving central limit
  theorems for the considered conditional dependence
  coefficients, with explicit formulas for the asymptotic
  variances.
  
\end{abstract}

{\bf Keywords:} conditional independence, graph-based test, rank-based test, nearest neighbor graphs, local power analysis.

\section{Introduction}\label{sec:intro}

Conditional (in)dependence is a fundamental statistical concept that
plays a central role in statistical inference and theory
\citep{MR535541,MR568723}. Testing conditional independence is
nowadays a routine task in graphical modeling \citep{MR3889064},
 causal discovery \citep{MR3822088}, feature selection
 \citep{koller1996toward}, and many other statistical
 applications. Formally, the problem of interest is to test for three
 random vectors $\mX,\mY,\mZ$ the hypothesis
\begin{align}\label{eq:H0}
  H_0: \mY\text{ and }\mZ\text{ are conditionally independent given }\mX,
\end{align}
based on a finite sample of size $n$ from the joint distribution of $(\mX,\mY,\mZ)$. 
It is customary to denote the conditional independence by  $\mY \indep
\mZ \given \mX$.

In contrast to the discrete/categorical case or favorable parametric
settings such as multivariate normality, the general problem of testing
\eqref{eq:H0} when $\mX$ is continuous is a remarkably challenging
task 
\citep{bergsma2004testing, MR4124333, MR4319245}.
A number of attempts have been made to provide
nonparametric solutions,
and notable examples include \citet{linton1997conditional} (on
conditional cumulative distribution functions); \citet{MR2413488,MR2428851,MR3212759} (on
conditional characteristic functions,
conditional probability density functions, and
smoothed empirical likelihood ratios, respectively); \citet{MR2676883} (on maximal
nonlinear conditional correlation); \citet{NIPS2007_3340},
\citet{10.5555/3020548.3020641}, \citet{10.5555/3020751.3020766}, and
\citet{StroblZhangVisweswaran+2019} (on kernel-based conditional
dependence); \citet{pmlr-v22-poczos12} and \citet{pmlr-v84-runge18a} (on
conditional mutual information); \citet{MR3269983} and \citet{MR3449068}
(on conditional distance correlation); \citet{MR2572451} and
\citet{JMLR:v23:20-682} (based on Rosenblatt transformation);
\citet{bergsma2004testing,bergsma2011nonparametric} and
\citet{MR2859749} (copula-based); \citet{NIPS2008_f7664060},
\citet{5740928}, \citet{MR4124333}, and \citet{MR4253763} (regression-based);
\citet{MR3826290} and \citet{MR4319245} (binning-based).


For the important special case where $\mY$ is a random scalar (and hence
denoted in regular font by $Y$), \citet{MR4352523} introduced
a novel and rather different
conditional dependence measure whose estimate ingenuously combines
ideas from rank statistics, nearest neighbor graphs and associated
minimum spanning trees for data sets.
The dependence measure and estimate
were shown to possess the following four
appealing properties:
\begin{enumerate}[itemsep=-.5ex,label=(\arabic*)]
\item the conditional dependence measure takes values in $[0,1]$, is 0
  if and only if $Y\indep \mZ\given \mX$, and is 1 if and only if $Y$
  is almost surely (a.s.) equal to a measurable function of $\mZ$ given $\mX$;
\item the estimate has a simple expression and can be computed in $O(n\log n)$ time; 
\item the estimate is fully nonparametric and has no tuning parameter;
\item the estimate is consistent as long as $Y$ is not a.s.~equal to a measurable function of $\mX$.
\end{enumerate}
The new approach has quickly caught attention. First follow-up work
studies extensions to topological spaces and multidimensional
$\mY$ and explores connections to general random graphs; see
\citet{deb2020kernel} and \citet{huang2020kernel}.  Moreover, for the
case of unconditional dependence, analyses were conducted to better understand the
statistical power of the approach. These analyses treat the very closely
related coefficient presented by \citet{MR4353729}; see
\citet{cao2020correlations}, \citet{MR4430960}, and
\citet{auddy2021exact}.

In this paper, we study the statistical efficiency of
Azadkia--Chatterjee's conditional dependence coefficient in testing the
hypothesis of \emph{conditional} independence from
\eqref{eq:H0}. \citet{MR4352523} themselves did not pursue
using their coefficient for inferential problems such as testing.  To
implement a test, we employ the conditional randomization test (CRT)
framework developed in \citet{MR3798878}; see \citet{MR4060981} for a
related proposal. The CRT framework, which was also adopted in \citet[Section
6.1.3]{huang2020kernel}, assumes that the
conditional distribution of $Y$ given $\mX$ is known, and thus the
null distribution of any conditional dependence coefficient can be
approximated by simulation.


Local power analyses for tests rely on a choice of local alternatives.  In the context of this paper, an
important subtlety lies in the
fact that in order to be relevant for a
CRT-based Azadkia--Chatterjee-type test, the conditional distribution
of $Y$ given $\mX$ should be identical between the null and local
alternatives. Two such families of local alternatives are considered
in this manuscript:
\begin{enumerate}[label=(\alph*)]
\item the joint density of $(\mX, Y, \mZ)$ in the alternative is
  assumed to be ``smoothly'' changing to the null in the sense of quadratic mean differentiability \citep[Definition~12.2.1]{MR2135927}. This is akin to parametric settings, and 
such families of local alternatives have been explored in studies of rank- and graph-based tests in related statistical problems \citep{MR3961499,cao2020correlations,MR4474478}. The critical detection boundary in such cases is known to be at root-$n$;
\item the conditional distribution of $(Y,\mZ)$ given $\mX$ is assumed
  to be H\"older smoothly changing with regard to $\mX$. This is akin
  to nonparametric settings, and the case we will consider is an
  extension of the one that has been examined by
  \citet{MR4319245}. There, as
  $\mX,Y,\mZ$ are all random scalars, the critical detection boundary
  is $n^{-2s/(4s+3)}$, where $s$ denotes the H\"older smoothness exponent.
\end{enumerate}

The local power analyses we report on in this manuscript provide, in
both of the above scenarios, examples that show that the CRT-based
Azadkia--Chatterjee-type test is unfortunately unable to achieve the
critical detection boundary, i.e., the sum of type-I and type-II
errors will not decrease to zero along the
boundary. 
We emphasize here that the power of CRT-based Azadkia--Chatterjee-type
test cannot simply be boosted to achieve the detection boundary by
using the additional information coming from the CRT framework, in view of the
H\'ajek representation theorem; compare Equation (S4) in
\citet{MR4430960-supp}. Our theoretical analysis thus echoes the
empirical observations made in \citet{{huang2020kernel}} and calls for
developing new variants of the tests that use an increasing number of
nearest neighbors when constructing the nearest neighbor graphs; see
\citet[Proposition~1]{MR3961499}, \citet[Remark
4.3]{deb2020kernel}, and in particular, a recent preprint by \citet{lin2021boosting} on boosting the power of Chatterjee's original proposal which however cannot be directly applied to multidimensional cases. These conclusions are also connected to related claims
made by \citet[Corollary~3]{MR443204}, \citet{MR3445317}, and
\citet[Theorems~1 and 2]{MR3909934} in other settings
of nonparametric statistics.

Our local power analysis in Case (a) is built on the innovative new
work of \citet{{deb2020kernel}}, who developed a general framework to
study normalized graph-based dependence measures (combined with rank-
and kernel-based ones) that invokes a Berry--Esseen theorem for
dependency graphs \citep{MR2073183}.  In order to complete our
analysis of the local power, however, an additional ingredient is
needed, namely, we have to prove the existence as well as calculate
the asymptotic variance of the (unnormalized) Azadkia--Chatterjee
conditional dependence coefficient.  To obtain this crucial result
we use
asymptotic techniques devised for 1-nearest neighbor graphs in
\citet{MR914597} and \citet{MR937563}.  This part of our derivations
shall occupy the main body of the proofs of our local power results.
Our analysis covers as a special instance the case of full
independence and, thus, resolves a conjecture of
\citet{MR4352523} about a central limit theorem (CLT) for their
statistic under full independence; see Section~\ref{sec:asynormal}
ahead for details.

Our local power analysis in Case (b), on the other hand, is based on a
brute-force calculation of the mean and variance of the
Azadkia--Chatterjee conditional dependence coefficient along a special
non-standard local alternative sequence that serves as the ``worst
case'' in the minimax lower bound construction of
\citet{MR4319245}. This involves handling permutation
statistics for which permutation randomness is not (though close to)
uniform over all possible rearrangements, a notoriously difficult
task. Interestingly, in a very recent preprint, \citet{auddy2021exact}
did related calculations in analyzing the local power of Chatterjee's
rank correlation coefficient \citep{MR4353729} against a
different family of non-standard local alternative sequences; see the
proof of their Theorem 2.1. It appears, though, that the techniques
used are substantially different from the ones present here, which of
course also differs through the focus on conditional (in-)dependence.

The rest of the paper is organized as
follows. Section~\ref{sec:prelim} reviews the conditional dependence
coefficient proposed by \citet{MR4352523}, denoted $\xi_n$, as
well as the conditional randomization test framework proposed by
\citet{MR3798878}. Section~\ref{sec:asynormal} presents the asymptotic
normality of $\xi_n$ under independence. Local power analyses of CRTs
based on $\xi_n$, in the two cases of alternatives are
presented in Section \ref{sec:subopt1} and Section~\ref{sec:subopt2},
respectively. A brief conclusion is provided in Section~\ref{sec:conclu}. 
The proof of Theorem \ref{thm:powerless} is given in Section~\ref{sec:main-proof},
with auxiliary results and remaining proofs deferred to the supplement.

\smallskip\noindent{\bf Notation.}
For an integer $n\ge 1$, let $\zahl{n}:= \{1,2,\ldots,n\}$.
A set consisting of distinct elements $x_1,\dots,x_n$ is written as either $\{x_1,\dots,x_n\}$ or $\{x_i\}_{i=1}^{n}$. 
The corresponding sequence is denoted $[x_1,\dots,x_n]$ or $[x_i]_{i=1}^{n}$. 
%
For a sequence of vectors $\mv_1,\ldots,\mv_k$, we use $(\mv_1,\ldots,\mv_k)$ as a shorthand for $(\mv_1^\top,\ldots,\mv_k^\top)^\top$. 
For a vector $\mv\in\R^d$, $\lVert\mv\rVert_{}$ stands for the Euclidean norm.
The symbols $\lfloor \cdot\rfloor$ and $\lceil \cdot\rceil$ denote the floor and ceiling functions, respectively. 
The notation $\ind(\cdot)$ is used for the indicator function.  
%
%
For any real-valued random vectors $\mU$ and $\mV$,  the (induced) probability measure, cumulative distribution function, and the probability density function of $\mU$ (if existing) are denoted as $\P_{\mU}$, $F_{\mU}$, and $q_{\mU}$, respectively; the conditional probability density function of $\mU$ given $\mV$ (if existing) is written as $q_{\mU\given\mV}$. In the following, the terms ``absolutely continuous'' and ``almost everywhere'' (shorthanded as ``a.e.'') are with respect to Lebesgue measure. 

\section{Conditional dependence measures and tests}\label{sec:prelim}

In the sequel, let $Y\in\R$ be a random scalar, and let $\mX\in \R^p$
and $\mZ\in \R^q$ be two random vectors, all defined on the same
probability space.  The
goal is to test \eqref{eq:H0} based on observations $(\mX_1, Y_1, \mZ_1), \ldots, (\mX_n, Y_n,
\mZ_n)$ that consist of $n$ independent copies of the triple $(\mX, Y, \mZ)$.   Note that the joint distribution of $(\mX, Y, \mZ)$ need
not be continuous. 

\subsection{Conditional dependence measures and coefficients}

\citet{MR4352523} proposed the following measure of conditional dependence between $Y$ and $\mZ$ given $\mX$:
\begin{align}\label{eq:xi}
\xi=\xi(Y,\mZ\given \mX):=\;&  \frac
{\displaystyle\int\E\big[\Var\big\{\P\big(Y\ge y\biggiven \mX,\mZ\big)\biggiven \mX\big\}\big]\d \P_{Y}(y)}
{\displaystyle\int\E\big[\Var\big\{\ind\big(Y\geq y\big)\biggiven \mX\big\}\big]\d \P_{Y}(y)}.
\end{align}
The following proposition describes the appealing properties we
pointed out in the introduction.

\begin{proposition}\citep[Theorem 2.1]{MR4352523}\label{prop:AC1}
Suppose that $Y$ is not a.s.~equal to a measurable function of $\mX$. 
Then $\xi$ is well-defined and belongs to the interval $[0,1]$. 
Moreover, $\xi$ is a consistent measure of conditional dependence with tailored extremal properties in the sense that $\xi = 0$ if and only if $Y$ and $\mZ$ are conditionally independent given $\mX$, 
and $\xi = 1$ if and only if $Y$ is a.s.\ equal to a measurable function of $\mZ$ given $\mX$. 
\end{proposition}
The dependence measure $\xi$  clearly extends an earlier
introduced measure of marginal dependence between $Y$ and (a random
scalar) $Z$, namely, 
\begin{align}\label{eq:DSS}
\xi^{\rm DSS}=\xi^{\rm DSS}(Y,Z):=\frac
{\displaystyle\int\Var\big\{\P\big(Y\ge y\biggiven Z\big)\big\}\d \P_{Y}(y)}
{\displaystyle\int\Var\big\{\ind\big(Y\geq y\big)\big\}\d \P_{Y}(y)},
\end{align}
which \citet{MR3024030} introduced for continuous
distributions and \citet{MR4353729} considered in general.
The quantities $\xi$ and $\xi^{\rm DSS}$ share similar properties:
(i) the consistency in measuring dependence is natural as the numerator (a
nonnegative scalar) is zero if and only if either $Y$ is independent
of $Z$ (for $\xi^{\rm DSS}$), or $Y$ is independent of $\mZ$ given
$\mX$ (for $\xi$); (ii) the self-normalization structure yields
tailored extremal properties as the numerator is always upper bounded
by the denominator; (iii) both the numerator and the denominator
involve the indicator $\ind(Y\geq y)$, which motivates estimation using the ranks
of the $Y_i$'s and their regression on the
$Z_i$'s or $(\mX_i,\mZ_i)$' to account for the conditioning in each term.

Both \citet{MR4353729} and
\citet{MR4352523} advocate a $1$-nearest neighbor (1-NN) approach
to performing the aforementioned regression; note that in one
dimension 1-NN is
obviously corresponding to working with ranks.
In detail, let
\begin{equation}\label{eq:Ri}
R_{i}:=\sum_{j=1}^{n}\ind\big(Y_{j}\le Y_{i}\big)
\end{equation}
be the rank of $Y_i$ among $Y_1,\ldots,Y_n$, and define
\begin{align*}
N(i)&:=\big\{j\ne i: \text{$\mX_j$ is the nearest neighbor of $\mX_i$}\big\},\\
M(i)&:=\big\{j\ne i: \text{$(\mX_j,\mZ_j)$ is the nearest neighbor of $(\mX_i,\mZ_i)$}\big\},
\yestag\label{eq:NNindex}
\end{align*}
to be the indices of the nearest neighbors of $\mX_i$ and
$(\mX_i,\mZ_i)$, respectively.  Here, nearest neighbors are determined
by Euclidean distance and possible ties in distance are broken at random. Azadkia and Chatterjee's conditional dependence coefficient is then defined as 
\begin{align}\label{eq:xin}
\xi_n =\xi_n\big(\big[(\mX_i,Y_i,\mZ_i)\big]_{i=1}^{n}\big):=\frac{\sum_{i=1}^{n}\{\min(R_i,R_{M(i)})-\min(R_i,R_{N(i)})\}}{\sum_{i=1}^{n}\{R_i-\min(R_i,R_{N(i)})\}}. 
\end{align}

Although not at all immediate at first sight, \citet{MR4352523}
showed that $\xi_n$ is a strongly consistent estimator of $\xi$ as
long as the latter is well-defined. We summarize the fact in the following proposition.

\begin{proposition}\citep[Theorem 2.2]{MR4352523}\label{prop:AC2}
As long as $Y$ is not a.s.\ equal to a measurable function of $\mX$,
it holds that $\xi_n$ converges to $\xi$ a.s.\ as $n\to\infty$. 
\end{proposition}

\begin{remark}
The intuition behind the convergence is by no means transparent. We
refer the readers of interest to \citet[Section 8]{MR4353729}
and the following heuristic argument:
\begin{align*}
\E\Big[n^{-1}\{R_1-\min(R_1,R_{N(1)})\}\Big] \approx&~ \E\Big[F_Y(Y_1)-\min\{F_Y(Y_1),F_Y(Y_{N(1)})\}\Big]\\
=&~ \E\Big[\int \Big\{ \ind(Y_1\geq t)-\ind(Y_1\geq t)\ind(Y_{N(1)}\geq t)\Big\}\d \P_Y(t)\Big]\\
\approx&~  \E\Big[\frac{1}{2}\int\Big\{\ind(Y_1\geq t)-\ind(Y_{N(1)}\geq t)\Big\}^2 \d \P_Y(t)\Big],
\end{align*}
with
\[
\E\{\ind(Y_1\geq t)-\ind(Y_{N(1)}\geq t)\}^2 \approx 2\E[\Var\{\ind(Y_1\geq t)\given \mX_1\}].
\]
Here the last step is intrinsically performed using an 1-NN regression. 
\end{remark}

\subsection{Conditional randomization tests}

Next we introduce the conditional randomization test (CRT)
framework of \citet{MR3798878}.  This framework is designed
for settings where the conditional distribution of $Y$ given $\mX$ is
known or can be accurately inferred from a large
out-of-sample data set. See also \citet[Section 2.2]{MR4060981} for an
illustration of application scenarios of this framework. In the sequel, we use $\Q\equiv\Q(\cdot\given \mx)$ to denote the
Markov kernel used in the algorithm implementing the test of conditional independence. Ideally, $\Q$ is (very
close to) the conditional distribution of $Y$ given $\mX=\mx$.

The CRT framework leverages that under $H_0$ the conditional
distribution of $Y$ given $(\mX,\mZ)$ is the same as that of
$Y$ given $\mX$. This yields the following observation: if $\Q$ equals
the conditional distribution of $Y$ given $\mX=\mx$ and $Y^{(1)}$ is
drawn independently from $\Q(\cdot\given \mX)$, then the two triples $(\mX,Y,\mZ)$ and
$(\mX,Y^{(1)},\mZ)$ are equal in distribution under $H_0$. In
contrast, any difference between the distributions of $(\mX,Y,\mZ)$
and $(\mX,Y^{(1)},\mZ)$ will manifest itself as evidence against $H_0$.
 
To make the idea practical, consider a real-valued test
statistic $\psi_n$ defined on the range of
$\big[(\mX_i, Y_i, \mZ_i)\big]_{i=1}^n$.  Let $B$ be a chosen number
of Monte Carlo simulations.  Then in each round $b\in \zahl{B}$,  one independently draws a new copy $Y_{i}^{(b)}$
from $Q(\cdot\given\mX_i)$ for $i\in\zahl{n}$, and  calculates
$\psi_n\big(\big[(\mX_i, Y_i^{(b)}, \mZ_i)\big]_{i=1}^n\big)$.  The CRT
then examines the difference between the distributions of
$(\mX,Y,\mZ)$ and $(\mX,Y^{(1)},\mZ)$ by comparing the observed test statistic
$\psi_n\big(\big[(\mX_i, Y_i, \mZ_i)\big]_{i=1}^n\big)$ to the
simulated values
$\psi_n\big(\big[(\mX_i, Y_i^{(b)}, \mZ_i)\big]_{i=1}^n\big)$.
The procedure is detailed in Algorithm~\ref{algo:1}.

\begin{algorithm}[t]
\SetAlgoLined
\KwIn{Data $\big[(\mX_i,Y_i,\mZ_i)\big]_{i=1}^{n}$, the chosen conditional distribution $\Q$, 
test statistic $\psi_n$, 
number of simulations $B$, 
and significance level $\alpha\in(0,1)$.}
\For{$b=1,\dots,B$}{Draw a sample $\big[Y_{i}^{(b)}\big]_{i=1}^{n}$ from the product distribution $\bigotimes_{i=1}^{n}\Q(\cdot\given\mX_{i})$, 
 independent of the observed $\big[(Y_i,\mZ_i)\big]_{i=1}^{n}$ and conditionally on $\big[\mX_{i}\big]_{i=1}^{n}$.}
\KwOut{CRT $p$-value defined as 
\[
~~~~p_{\rm CRT}=(1+B)^{-1}\Big[1+\sum_{b=1}^{B}\ind\Big\{
\psi_n\Big(\big[(\mX_i,Y_i^{(b)},\mZ_i)\big]_{i=1}^{n}\Big)\ge 
\psi_n\Big(\big[(\mX_i,Y_i,\mZ_i)\big]_{i=1}^{n}\Big)\Big\}\Big].\!\!\!\!
\]
The CRT is then 
\[
\sT^{\Q,\psi_n}_\alpha\big(\big[(\mX_i,Y_i,\mZ_i)\big]_{i=1}^{n}\big)=\ind(p_{\rm CRT}\le\alpha).
\]}
\caption{\label{algo:1}Conditional randomization test (CRT)}
\end{algorithm}

%

\subsection{CRT using the Azadkia--Chatterjee coefficient}

We will be concerned with the CRT that is obtained by taking  the test
statistic $\psi_n$ in Algorithm \ref{algo:1} to be $\xi_n$, the
Azadkia--Chatterjee conditional dependence coefficient. The resulting
test for significance level $\alpha$ is denoted by
$\sT^{\Q,\xi_n}_\alpha$; here $\Q$ is added to highlight the
dependence of the implementation on the chosen conditional
distribution $\Q$.

\begin{remark}
  \citet[Section 2.2]{MR4060981} argued that in many cases the
  unlabeled data, i.e., data on $(\mX,Y)$ but without the $\mZ$
  component, are plentiful, but labeled data on $(\mX, Y, \mZ)$
  jointly are scarce.  In such cases, it is natural to assume that one
  not only knows (or may very accurately estimate) the needed
  conditional distribution but also the joint distribution of $(\mX,
  Y)$.  When the distribution of $(\mX,Y)$ is known, the only term to
  be estimated from data is the numerator in \eqref{eq:xi};  the denominator in \eqref{eq:xi} only depends on the distribution of $(\mX,Y)$. 

  However, we would like to emphasize that, in view of the H\'ajek
  representation theorem as given in Equation (S4) in
  \citet{MR4430960-supp}, replacing each $R_i$ by $nF_Y(Y_i)$ in
  $\xi_n$ will not result in an (asymptotic) improvement of
  $\xi_n$. In detail, although it is tempting to define an ``oracle
  version'' of $\xi_n$ that uses more information as
\begin{align*}
\widecheck\xi_n=\widecheck\xi_n\Big(\big[(\mX_i,Y_i,\mZ_i)\big]_{i=1}^{n}\Big):=\frac{n^{-1}\sum_{i=1}^{n}[\min\{F_{Y}(Y_i),F_Y(Y_{M(i)})\}-\min\{F_{Y}(Y_i),F_{Y}(Y_{N(i)})\}]}{\int\E[\Var\{\ind(Y\ge y)\given \mX\}]\d \P_{Y}(y)},
\end{align*}
this change will not increase the CRT's power (asymptotically), 
at least in all settings considered in this paper.
\end{remark}

Later, in Sections \ref{sec:subopt1}--\ref{sec:subopt2}, we shall
study in detail the test based on $\xi_n$. However, some preliminary results deserve to be documented first. 
In the following, let $\mathcal{P}_{\Q}$ be the family of all joint
distributions for $(\mX,Y,\mZ)$ such that $Y$ is not a.s.~equal to a
measurable function of $\mX$ and the conditional distribution of $Y$
given $\mX=\mx$ coincides with a given (non-trivial) Markov kernel
$\Q$.  

\begin{proposition}[Control of size and consistency]\label{prop:vacon}
Fix a Markov kernel $\Q$ for $Y$ given $\mX$.
\begin{enumerate}[label=(\roman*)]
\item\label{prop:vacon-1} The test $\sT^{\Q,\xi_n}_\alpha$ is valid in
  the sense that for any $\P_{(\mX,Y,\mZ)}\in\mathcal{P}_{\Q}$
  satisfying $H_0$, denoting $\P_{H_0}:=\P_{(\mX,Y,\mZ)}^{\otimes n}$
  as the corresponding product measure, it holds for any $n\geq 1$ that
\[
\P_{H_0}(\sT^{\Q,\xi_n}_\alpha=1)\le\alpha;
\] 
notice that no assumption 
concerning the number of simulations $B$ is required at all.
\item\label{prop:vacon-2} In addition, $\sT^{\Q,\xi_n}_\alpha$ is consistent in the sense that for any $\P_{(\mX,Y,\mZ)}\in\mathcal{P}_{\Q}$ violating $H_0$, denoting $\P_{H_1}$ as the corresponding product measure, we have
\[
\lim_{n\to\infty} \P_{H_1}(\sT^{\Q,\xi_n}_\alpha=1)=1
\]
as long as the number of simulations $B$ tends to infinity as $n\to\infty$.
\end{enumerate}
\end{proposition}

{\begin{remark}
Of note, \citet[Corollary 23]{MR4253763} and \citet[Theorem 4]{lundborg2021conditional}, among many others, proved uniform consistency of their conditional independence tests against particular subsets of alternative hypotheses. Their results are established via some careful non-asymptotic analysis of the test statistic along such local alternatives. It will be a statistically and also technically very interesting question to examine whether $\sT^{\Q,\xi_n}_\alpha$ also enjoys similar properties. This is still an open problem.
\end{remark}}

\section{Asymptotic normality under independence}\label{sec:asynormal}

In this section we consider the asymptotic behavior of $\xi_n$
under independence of $Y$ and $(\mX,\mZ)$, which constitutes a
special subfamily of the conditional independence hypothesis $H_0$ that our subsequent
theoretic analysis shall be built on.  In this
(unconditional) independence scenario we then consider the coefficient
$\xi_n$ from Section \ref{sec:prelim} as well as a variant introduced in \citet{MR4352523}.  

In detail, \citet{MR4352523} also examined the case when
$p=0$, i.e., $\mX$ has no component. In this case, 
the conditional dependence measure $\xi$ from
\eqref{eq:xi} reduces to the unconditional dependence measure
defined analogous to $\xi^{\rm DSS}$ from \eqref{eq:DSS};
here the dimension of $\mZ$ is not necessarily one. 
They then introduced the following coefficient $\xi^\#_n$, which
extends the original proposal of \citet[Eqn.~(1)]{MR4353729} to higher dimension $q\geq 1$:
\[
\xi^\#_n
:=\frac{\sum_{i=1}^{n}\{n\min(R_i,R_{M(i)})-L_i^2\}}
{\sum_{i=1}^{n}L_i(n-L_i)}.
\yestag
\]
Here $R_{i}$ and $M(i)$ are defined in \eqref{eq:Ri} and \eqref{eq:NNindex}, respectively, with the understanding that $\mX$'s part in \eqref{eq:NNindex} is removed since it is of no component, and $L_{i}:=\sum_{j=1}^{n}\ind\big(Y_{j}\ge Y_{i}\big)$. 

\citet{MR4352523} conjectured that under
independence between $Y$ and $\mZ$, $\sqrt{n}\xi^\#_n$ obeys a
CLT. Building on results of \citet{deb2020kernel}, we are able to
derive the following theorem that, in particular, gives an affirmative answer to this
conjecture under (absolute) continuity.

\begin{theorem}[Asymptotic normality]
\label{thm:byprod}\mbox{ }
\begin{enumerate}[label=(\roman*)]
\item \label{thm:byprod1} Assume that $Y\in\R$ is continuous and independent of $(\mX, \mZ)\in\R^{p+q}$. In addition, assume $(\mX,\mZ)$ is absolutely continuous admitting a density continuous over its support. We then have as, $n\to\infty$,  
\begin{align*}
\sqrt{n}\xi_n&\stackrel{\sf d}{\longrightarrow}N\Big(0,\frac45+\frac25\Big\{\kq_{p+q}+\kq_{p}\Big\}+\frac45\Big\{\ko_{p+q}+\ko_{p}\Big\}\Big),
\end{align*}
where for any integer $d\geq 1$, $\kq_{d}$ and $\ko_{d}$ are positive
constants depending only on $d$.  Their values are
\begin{align*}
&\kq_d:=\Big\{2-I_{3/4}\Big(\frac{d+1}{2},\frac{1}{2}\Big)\Big\}^{-1},
~~~~
I_{x}(a,b):=\frac{\int_{0}^{x}t^{a-1}(1-t)^{b-1} \d t}
                 {\int_{0}^{1}t^{a-1}(1-t)^{b-1} \d t},
\yestag\label{eq:defn_kqd}\\
&\ko_{d}:=\int_{\Gamma_{d;2}}\exp\Big[-\lambda\Big\{B(\mw_1,\lVert\mw_1\rVert_{})\cup B(\mw_2,\lVert\mw_2\rVert_{})\Big\}\Big]\d(\mw_1,\mw_2),
\yestag\label{eq:defn_kod}\\
&\Gamma_{d;2}:=\Big\{(\mw_1,\mw_2)\in(\R^d)^2: \max(\lVert\mw_1\rVert_{},\lVert\mw_2\rVert_{})<\lVert\mw_1-\mw_2\rVert_{}\Big\},
\end{align*} 
where $B(\mw_1,r)$ is the ball of radius $r$ centered at $\mw_1$, and $\lambda(\cdot)$ is the Lebesgue measure. 
\item \label{thm:byprod2} Assume $Y\in\R$ is continuous and independent of $\mZ\in\R^q$. In addition, assume $\mZ$ is absolutely continuous. 
We then have, as $n\to\infty$, 
\begin{align*}
\sqrt{n}\xi^\#_n&\stackrel{\sf d}{\longrightarrow}N\Big(0,\frac25+\frac25\kq_{q}+\frac45\ko_{q}\Big).
\end{align*}
\end{enumerate}
\end{theorem}

\begin{remark}
  The asymptotic variance of $\sqrt{n}\xi_n$ (or $\sqrt{n}\xi_n^{\#}$)
  under independence between $Y$ and $(\mX,\mZ)$ (or $\mZ$) is seen to
  be distribution-free, i.e., its value will not change with the
  particular distribution of $\P_{(\mX,Y,\mZ)}$ as long as the
  (absolute) continuity conditions in Theorem \ref{thm:byprod}
  hold. This (asymptotic) distribution-freeness is in line with
  similar observations made earlier for related problems such as two-sample goodness-of-fit tests, where \citet{MR532236} extended \citet{MR2083}'s rank-based run test to multivariate spaces via minimum spanning trees; see, also, \citet{MR947577}, \citet{MR1212489}, \citet{MR1701112}, and \citet{MR3961499} for other notable results along that track, {and  \citet{MR3811758, MR4474546, lin2021boosting, lin2022limit} for more related work.}
\end{remark}

\begin{remark}
  It may be interesting to note that the asymptotic variance of
  $\sqrt{n}\xi_n^{\#}$ is strictly larger than $2/5$, the asymptotic
  variance of the rank correlation from Equation~(1) in
  \citet{MR4353729}.  However, this observation
  should not be interpreted as an advantage of Chatterjee's rank
  correlation over $\sqrt{n}\xi_n^{\#}$ in terms of statistical efficiency. 
  As a matter of fact, both are powerless when used for testing independence; 
  cf.~\citet[Theorem~1]{MR4430960} and Theorem \ref{thm:powerless} ahead.
\end{remark}

\begin{remark}
  In order to derive the above CLTs for $\xi_n$ and $\xi_n^{\#}$, we
  adopt techniques devised in \citet{deb2020kernel}.  We highlight here
  some of these technical ingredients of our proof. 
  \citet{deb2020kernel} were focused on establishing
  general asymptotic results for graph-based statistics with an
  additional self-normalization step.  In the present
  context of our Theorem \ref{thm:byprod}, following \citet[Theorem
  4.1]{deb2020kernel}, it is readily shown that
\[
\sqrt{n}\xi_n\Big/\Big\{\hat\Var(\xi_n)\Big\}^{1/2} \stackrel{\sf d}{\longrightarrow} N(0,1),
\]
for some data-based normalization statistic $\hat\Var(\xi_n)$. Our
main focus, accordingly, can be understood as proving the existence as
well as deriving the value of the limit of $\hat\Var(\xi_n)$ as
$n\to\infty$. This problem was not touched upon in
\citet{deb2020kernel} for a good reason, but is crucial for our
analysis of the power ahead. To fill the gap, our proof draws on the remarkable techniques developed in \citet{MR914597} and  \citet{MR937563}, which will be detailed blow.
\end{remark}

The asymptotic variances of $\xi_n$ and $\xi_n^{^\#}$ may look
mysterious but they are in fact connected to the behavior of nearest
neighbor graphs.  We present here a series of results that illustrate
this connection.  The first is a well-known result by \citet[Corollary
S1]{MR682809} on maximum degrees in 1-NN graphs.

\begin{lemma}[Maximum degree in nearest neighbor graphs]\label{lem:MD}
  Let $\mw_1,\dots,\mw_n$ be any collection of $n$ distinct points in
  $\R^d$.  Then there exists a constant $\kC_d$ depending only on the dimension $d$
  such that $\mw_1$ is the nearest neighbor of at most $\kC_d$
  points from $\{\mw_2,\dots,\mw_n\}$.
\end{lemma}

The notation $\kC_d$, representing a constant upper bound of the
maximum degree, will be used throughout the manuscript.  For
convenience, we take $\kC_d$ as the smallest constant for which
the property in Lemma~\ref{lem:MD} holds.


In the following, consider a sample $[\mW_i]_{i=1}^n$ comprised of $n$
independent copies of a random vector $\mW\in\R^d$.  Let $\cG_n$ be
the associated directed nearest neighbor graph (NNG), i.e.,  $\cG_n$
has vertex set $\zahl{n}$ and contains a directed edge from $i$ to
$j$ whenever $\mW_j$ is a nearest neighbor of $\mW_i$.  We write
$\cE(\cG_n)$ for the edge set of $\cG_n$. 

The parameter $\kq_d$ in Theorem \ref{thm:byprod} comes from the following crucial result of \citet[Theorem~2]{MR937563}.

\begin{lemma}[Expected number of nearest-neighbor pairs]\label{lem:devr} 
As long as $\mW$ is absolutely continuous, we have
\[
  \E\bigg(\frac{1}{n} \#\Big  \{(i,j)~\text{distinct}: \, i\to j, j\to i\in \cE(\cG_n)\Big\}
\bigg)
\;\longrightarrow\; \frac{V_d}{U_d}=\kq_d,
\]
where $V_d$ is the volume of the unit ball in $\R^d$, and $U_d$ is the volume of the union of two unit balls in $\R^d$ whose centers are a unit distance apart. The explicit value of $\kq_d$ shown in \eqref{eq:defn_kqd} is given by \citet[Equation (3)]{MR2813331}.
\end{lemma}

The parameter $\ko_d$ in Theorem \ref{thm:byprod}, on the other hand,
comes from the following new lemma, which is developed in this manuscript.
The lemma builds on an earlier result of \citet{MR914597}. 

\begin{lemma}\label{lem:tech3}
As long as $\mW$ is absolutely continuous, we have
\[
 \E\bigg(\frac{1}{n} \#\Big  \{(i,j,k)~\text{distinct}: \, i\to k, j\to k\in \cE(\cG_n)\Big\}
\bigg)
\;\longrightarrow\;\ko_d,
\]
where $\ko_d=\kd_{d;2}$, a quantity defined in Lemma \ref{lem:henze} below.
\end{lemma}

The next lemma is due to \citet[Theorem~1.4,  Corollaries~1.5 and 1.6]{MR914597}.

\begin{lemma}[Expected number of vertices of specified degree]\label{lem:henze}
Let $d^{-}_j$ be the in-degree of vertex $\mW_j$ in $\cG_n$, i.e.,
$d^{-}_j:=\#\{i: i\to j\in \cE(\cG_n)\}$. 
If $\mW$ is absolutely continuous with a density continuous a.e., then for any integer $k \in [0, \kC_{d}]$, we have
\[
\E\Big(\frac{1}{n}\#\{j: d^{-}_j=k\}\Big)\;\longrightarrow\; \kp_{d;k}
~~~~\text{and}~~~~
\Var(d^{-}_1)\;\longrightarrow\;\kd_{d;2},
\]
where
\[
\kp_{d;k}=\frac{1}{k!}\sum_{u=0}^{\kC_d-k}\frac{1}{u!}(-1)^{u}\kd_{d;k+u},
~~~~0\le k\le \kC_d,
\]
and
\begin{align*}
&\kd_{d;0}=\kd_{d;1}=1,~~~~\kd_{d;r}=\int_{\Gamma_{d;r}}\exp\Big[-\lambda\Big\{\bigcup_{i=1}^{r}B\big(\mw_i,\lVert\mw_i\rVert_{}\big)\Big\}\Big]\d(\mw_1,\ldots,\mw_r),\\
&\Gamma_{d;r}=\Big\{(\mw_1,\dots,\mw_r)\in(\R^d)^r: \lVert\mw_i\rVert_{}<\min_{1\le j\le r: j\ne i}\lVert\mw_i-\mw_j\rVert_{}, 1\le i\le r\Big\},~~~~2\le r\le \kC_d.
\end{align*}
Notice that $\kp_{d;k}\in[0,1]$ is a constant only depending on $d$ and $k$. 
\end{lemma}

\begin{remark}
  We note that in Theorem \ref{thm:byprod}\ref{thm:byprod1}, a slightly stronger
  condition (continuity over its support) is required for the density
  function in order to establish CLTs.  This additional requirement is
  made for handling the ``cross terms'' of 1-NN graphs built on
  $[(\mX_i,\mZ_i)]_{i=1}^n$ and $[\mX_i]_{i=1}^n$ separately
  (cf.~Lemma \ref{lem:tech6} ahead as an analogue of Lemmas~\ref{lem:devr}--\ref{lem:henze} for the cross terms). Such cross terms are not present
  in \citet{MR937563} and \citet{MR914597}. Roughly speaking, we will
  prove that 
  the two 1-NN graphs built on $[(\mX_i,\mZ_i)]_{i=1}^n$ and $[\mX_i]_{i=1}^n$ 
  are nearly independent from each other. The proof of Lemma \ref{lem:tech6}
  adopts Devroye's and Henze's ideas but involves further analysis. 
\end{remark}

\section{Power analysis: Parametric case}\label{sec:subopt1}

This section investigates the local power of the proposed
tests for quadratic mean differentiable classes of alternatives
\citep[Definition~12.2.1]{MR2135927}, 
for which  we show that the CRT based on Azadkia and Chatterjee
$\xi_n$ possesses only trivial power in $n^{-1/2}$ neighborhoods.

We begin with a set of local alternatives
\begin{align}\label{eq:qmd}
\Big\{q_\Delta(\mx,y,\mz): |\Delta|<\Delta^*\Big\}, ~~~\Delta^*>0, 
\end{align}
where for each $|\Delta|<\Delta^*$, $q_\Delta(\mx,y,\mz)$ is a joint density with respect to the Lebesgue measure. 
We then make assumptions on the set in \eqref{eq:qmd}. In the
following, $\E_0(\cdot)$ is understood to be the expectation operator
with regard to the density function $q_0(\mx,y,\mz)$ obtained for $\Delta=0$.

\begin{assumption}\label{asp:only}
It is assumed that
\begin{enumerate}[label=(\roman*)]
\item\label{asp:0} $q_0(\mx,y,\mz)$ is such that $Y$ and $\mZ$ are conditionally independent given $\mX$;
\item\label{asp:1} for all $|\Delta|<\Delta^*$, 
\[
\int q_\Delta(\mx,y,\mz)\d\mz=q_{\mX,Y}(\mx,y),
\]
where $q_{\mX,Y}(\cdot,\cdot)$, the density of $\P_{(\mX,Y)}$, is fixed and equal to the product of densities of $\P_{\mX}$ and $\P_{Y}$, and invariant with regard to $\Delta$; 
\item\label{asp:2} the score function
\[
\dot{\ell}_\Delta(\mx,y,\mz)
:=\frac{\partial}{\partial\Delta}\log q_\Delta(\mx,y,\mz)
\]
exists at $\Delta=0$, and the family $\{q_\Delta(\mx,y,\mz)\}_{|\Delta|<\Delta^*}$ is quadratic mean differentiable (QMD) at $\Delta=0$ with score function $\dot{\ell}_0$, that is,
\[
\int\Big(\sqrt{q_\Delta(\mx,y,\mz)}-\sqrt{q_0(\mx,y,\mz)}
-\frac12\Delta\dot{\ell}_0(\mx,y,\mz)\sqrt{q_0(\mx,y,\mz)}\Big)^2\d(\mx, y, \mz) 
=o(\Delta^2)
\]
as $\Delta\to0$; 
\item\label{asp:3} $\E_{0}\{\dot{\ell}_0(\mX,Y,\mZ)^2\}>0$
(Assumption~\ref{asp:2} implies 
$\E_{0}\{\dot{\ell}_0(\mX,Y,\mZ)^2\}<\infty$ and 
$\E_{0}\{\dot{\ell}_0(\mX,Y,\mZ)\}=0$);
\item\label{asp:4} $\E_{0}\{\dot{\ell}_0(\mX,Y,\mZ)\given \mX,\mZ\}=0$ almost surely;
\item\label{asp:5} $\E_{0}\{\lvert\dot{\ell}_0(\mX,Y,\mZ)\rvert^{4+\epsilon}\}<\infty$ for some fixed constant $\epsilon>0$; 
\item\label{asp:6} $\dot{\ell}_0(\mx,y,\mz)$ cannot be written as $h_1(y)+h_2(\mx,\mz)$.
\end{enumerate}
\end{assumption}

\begin{example}[Rotation alternatives]\label{ex:ellip}
Suppose that $\mX^*\in\R^p$, $Y^*\in\R$, and $\mZ^*\in\R^{q}$  
are centered and jointly normally distributed random variables 
such that $Y^*$ is independent of $(\mX^*,\mZ^*)$. 
Then Assumption~\ref{asp:only} holds for rotation alternatives given as
\[
\big(\mX,Y,\mZ\big) = 
\Big(\mX^*,Y^*,\mZ^*+\Delta\big(\fA\mX^*+\fB Y^*\big)\Big),
\]
where $\fA\in\R^{q\times p},\fB\in\R^{q\times 1}$ are deterministic matrices,
and $\fB$ is nonzero. 
\end{example}


\begin{example}[Farlie alternatives]\label{ex:far}
Suppose that $\mX^*\in\R^p$, $Y^*\in\R$, and $\mZ^*\in\R^q$ are 
absolutely continuous random variables 
such that $Y^*$ is independent of $(\mX^*,\mZ^*)$. 
Then Assumption~\ref{asp:only} holds for the (generalized) Farlie alternatives 
(see \citet[Sec.~1.1.5]{MR2288045} for the one-dimensional case) that
are defined as
\[
q_{\Delta}\big(\mx,y,\mz\big):=
q_{Y^*}\big(y\big)q_{(\mX^*,\mZ^*)}\big(\mx,\mz\big)
\Big[1+\Delta\big\{1-2F_{Y^*}\big(y\big)\big\}
             \big\{1-2F_{(\mX^*,\mZ^*)}\big(\mx,\mz\big)\big\}\Big].
\]
\end{example}

For a local power analysis for an alternative set under the listed assumptions,
we examine the asymptotic power along a respective sequence of
alternatives obtained as
\begin{equation}\label{eq:local-alter}
H_{1,n}(\Delta_0):\Delta=\Delta_n,  \text{ where } \Delta_n:= n^{-1/2}\Delta_0
\end{equation}
with some constant $\Delta_0\ne 0$. 
In this local model, testing the null hypothesis of independence reduces
to testing 
\[
H_0:\Delta_0=0~~~ {\rm versus}~~~ H_1:\Delta_0\ne 0.
\] 
We obtain the following theorem on the local power of the discussed
tests.  The result demonstrates the trivial power claimed for $\xi_n$
in the beginning of
this section.

\begin{theorem}[Power analysis for $\xi_n$]\label{thm:powerless}
Suppose that the considered set of local alternatives in
\eqref{eq:qmd} satisfies Assumption~\ref{asp:only} and constitutes a
subset of $\cP_{\Q}$. Then for any sequence of alternatives given in \eqref{eq:local-alter}, for any fixed constant $\Delta_0>0$,  
\begin{enumerate}[label=(\roman*)]
\item \label{thm:pl1}
assuming the number of simulations
$B$ for the CRT tends to infinity as $n\to\infty$, 
it holds that
$$
\lim_{n\to\infty} \P_{H_{1,n}(\Delta_0)}(\sT^{\Q,\xi_n}_\alpha=1)\le\alpha;
$$
\item \label{thm:pl3}
in contrast, there exists a test $\sT_{\alpha}^{\rm opt}$ such that for any $\alpha,\beta\in(0,1)$, as long as $\Delta_0$ is sufficiently large, it holds that
\[
\lim_{n\to\infty} \P_{H_0}(\sT_{\alpha}^{\rm opt}=1)\le\alpha
~~~~\text{and}~~~~
\lim_{n\to\infty} \P_{H_{1,n}(\Delta_0)}(\sT_{\alpha}^{\rm opt}=1)\ge 1-\beta,
\]
while for small $\Zeta$ the total variation distance vanishes as
\[
\lim_{\Delta_0\to 0}\lim_{n\to\infty}\TV(H_{1,n}(\Delta_0), H_0)=0,
\]
{and hence
\[
\lim_{\Delta_0\to 0}\lim_{n\to\infty}\inf_{\overline\sT_{\alpha}\in\cT_{\alpha}}\P_{H_{1,n}(\Delta_0)}(\overline\sT_{\alpha}=0)\geq 1-\alpha.
\]
Here the infimum is taken over all size-$\alpha$ tests.}
\end{enumerate}
\end{theorem}

\begin{remark}
We give a rigorous proof of Theorem~\ref{thm:powerless}\ref{thm:pl1} 
 in Section~\ref{subsebsec:711}. 
The main idea is to first derive the joint limiting null distribution of 
$\sqrt{n}\xi_n$ and the log likelihood ratio between two hypotheses,
and then to use Le~Cam's third lemma. 
In addition to Theorem~\ref{thm:byprod}\ref{thm:byprod1},
combining results from \citet{MR4352523}, 
we are able to prove joint asymptotic normality 
with deterministic variance 
of $\sqrt{n}\xi_n$ and the log likelihood ratio; 
in particular, 
zero asymptotic covariance between $\sqrt{n}\xi_n$ and the log likelihood ratio
is the technical reason why the CRT-based Azadkia--Chatterjee-type test
is inefficient in the quadratic mean differentiable class.  
\end{remark}

\begin{remark}
The phenomenon that a (1-NN) graph-based test has zero asymptotic
efficiency has been encountered also in other situations.
For example, the lack of power of the Wald--Wolfowitz runs test is a
classic result in the literature \citep[p.~102]{MR1680991}. A systematic analysis of this phenomenon in the two-sample test context was done recently in \citet{MR3961499} and similar analyses for Chatterjee's 1-NN tests of unconditional independence were performed in \citet{cao2020correlations}, \citet{MR4430960}, and \citet{auddy2021exact}.
\end{remark}


\section{Power analysis: Nonparametric case}\label{sec:subopt2}

In this section, we conduct local power analyses of the proposed tests within
the H\"older smooth class inspired by the work of
\citet{MR4319245}.  In this class, the conditional distribution of $(Y, \mZ)$ is allowed to change more dramatically (beyond the limit of QMD classes established in Section \ref{sec:subopt1}) as $\mX$ changes. 
In the sequel, following the settings treated in \citet{MR4319245}, we consider $\mX\in[0,1]^p$, $Y\in[0,1]$, and $Z\in[0,1]$ to be continuous random vector/variables.  

\subsection{Rate of convergence}

Let $\cE_{[0,1]^{p+2}}$ be the set of all absolutely continuous
distributions $(\mX, Y, Z)\in [0,1]^{p+2}$ such that the randomness of
the triplet can be understood as first sampling $\mX$ from a density
$q_{\mX}$ with support $[0, 1]^p$, and then sampling $Y$ and $Z$ from
a conditional distribution $q_{(Y,Z) \given \mX}$ of support $[0,1]\times[0,1]$ for (almost) all $\mX$.
Let $\cP_{0} \subseteq \cE_{[0,1]^{p+2}}$ be the subset for which $Y \indep Z \given \mX$,
and let $\cP_{1} = \cE_{[0,1]^{p+2}} \backslash \cP_{0}$.

Next we separately define the two H\"older classes of density
functions, belonging to $\cP_0$ (the null class) and $\cP_1$. Our main interest is on exponents $s$ that are close
to 0, representing those conditional distributions of $(Y,Z)$ that
change possibly very roughly with the values of $\mX$.

\begin{definition}[Null H\"older class]\label{def:hs-null}
Let $\cP_{0}(L,s) \subseteq \cP_{0}$ with $L>1$ and $s\in(0,1]$ be
the collection of joint distributions of $(\mX,Y,Z)$ such that, for all $\mx,\mx'\in[0, 1]^p$, $y,y',z,z'\in[0, 1]$, we have
\begin{align*}
\Big\lvert q_{Y \given \mX}(y \given \mx) - q_{Y \given \mX}(y \given \mx')\Big\rvert &\le L\lVert \mx - \mx'\rVert_{}^s\\
\text{and}~~~~
\Big\lvert q_{Z \given \mX}(z \given \mx) - q_{Z \given \mX}(z \given \mx')\Big\rvert &\le L\lVert \mx - \mx'\rVert_{}^s.
\end{align*}
\end{definition}


\begin{definition}[Alternative H\"older class]\label{def:hs-alter}
Let $\cP_{1}(L,s) \subseteq \cP_{1}$ with $L>1$ and $s\in(0,1]$ be the collection of joint distributions of  $(\mX,Y,Z)$ such that, for all $\mx, \mx' \in [0, 1]^p$, $y,y',z,z'\in[0, 1]$, we have
\begin{align*}
\Big\lvert q_{(Y,Z) \given \mX}(y,z \given \mx) - q_{(Y,Z) \given \mX}(y,z \given \mx')\Big\rvert 
&\le L\lVert \mx - \mx'\rVert_{}^s,\\
\Big\lvert q_{(Y,Z) \given \mX}(y,z \given \mx) - q_{(Y,Z) \given \mX}(y',z' \given \mx)
\Big\rvert 
&\le L\Big(\lvert y-y'\rvert^s+\lvert z-z'\rvert^s\Big),\\[1mm]
\text{and}~~~~
L^{-1}\le q_{Y,Z \given \mX}(y,z\given\mx) &\le L.
\end{align*}
\end{definition}

To obtain the rate of convergence for $\xi_n$ under $\cP_0(L,s)$ as
well as $\cP_1(L,s)$, we establish the following two results. The
first is a proposition that extends Theorem~4.1 in \citet{MR4352523}, which
focused on the Lipschitz class with $s=1$.  The second result is a
lemma that shows that the distributions in $\cP_0(L,s)$ and $\cP_1(L,s)$ satisfy the conditions in the proposition.

\begin{proposition}\label{prop:ACHolder}
Restricted to this proposition, $\mZ\in\R^q$ is allowed to be
multidimensional. Suppose then that $Y$ is not a.s.~equal to a
measurable function of $\mX$ and that
\begin{enumerate}[label=(\roman*)]
\item there are universal constants $C_1>0$ and $s\in(0,1]$ such that for any $t \in\R$, $\mx,\mx'\in \R^p$, and $\mz,\mz'\in \R^q$, 
\begin{align*}
\Big\lvert\P(Y\ge t \given \mX=\mx,\mZ=\mz) -\P(Y\ge t \given \mX=\mx',\mZ=\mz')\Big\rvert 
&\le C_1 \Big(\lVert\mx-\mx'\rVert_{}^s+\lVert\mz-\mz'\rVert_{}^s\Big),\\
\text{and}~~~~
\Big\lvert\P(Y\ge t \given \mX=\mx)-\P(Y\ge t \given \mX=\mx')\Big\rvert 
&\le C_1 \Big(\lVert\mx-\mx'\rVert_{}^s\Big);
\end{align*}
\item there exists a universal constant $C_2>0$ such that 
$\P(\lVert\mX\rVert_{} \ge C_2)=0$ and $\P(\lVert\mZ\rVert_{} \ge C_2)=0$.
\end{enumerate}
Then, as $n\to\infty$,
\[
\xi_n -\xi =O_{\P}\Big(\frac{(\log n)^{p+q+1}}{n^{s/(p+q)}}\Big).
\]
\end{proposition}

\begin{lemma}\label{lem:condmore}
\begin{enumerate}[label=(\roman*)]
\item \label{lem:condmore1} If $\P_{(\mX,Y,Z)}\in\cP_{0}(L,s)$ for a fixed $L>1$ and $s\in(0,1]$, then for all $\mx,\mx'\in[0, 1]^p$, $y,y',z,z'\in[0, 1]$, we have
\[
\Big\lvert q_{Y \given (\mX, Z)}(y \given \mx,z) - q_{Y \given (\mX, Z)}(y \given \mx',z')\Big\rvert \le L\lVert \mx - \mx'\rVert_{}^s. 
\]

\item \label{lem:condmore2} If $\P_{(\mX,Y,Z)}\in\cP_{1}(L,s)$ 
for a fixed $L>1$ and $s\in(0,1]$, 
then for all $\mx,\mx'\in[0, 1]^p$, $y,y',z,z'\in[0, 1],$ we have 
\[
\Big\lvert q_{Y \given (\mX, Z)}(y \given \mx,z) - q_{Y \given (\mX, Z)}(y \given \mx',z')\Big\rvert \le L'\Big(\lVert \mx-\mx'\rVert_{}^s+\lvert z-z'\rvert^s\Big)
\]
for some $L'\le 2L^4$. 
\end{enumerate}
\end{lemma}

%

Combining Lemma \ref{lem:condmore} with Proposition \ref{prop:ACHolder} gives the following corollary.

\begin{corollary}
Suppose $\P_{(\mX,Y,Z)}\in\cP_{0}(L,s)$, or
$\P_{(\mX,Y,Z)}\in\cP_{1}(L,s)$ with $Y$ not a.s.~equal to a
measurable function of $\mX$.  Then as $n\to\infty$, 
\[
\xi_n -\xi =O_{\P}\Big(\frac{(\log n)^{p+q+1}}{n^{s/(p+q)}}\Big).
\] 
\end{corollary}

\begin{remark}
\citet{huang2020kernel} proposed a measure of conditional dependence and a coefficient estimating the measure in general spaces by generalizing \citet{MR4352523}'s idea. In particular, they also explored the rate of convergence of such a coefficient in their Theorem~3.3, and mentioned that the rate of convergence may be arbitrarily slow without a smoothness assumption on the conditional distribution \citep[Remark~3.1]{huang2020kernel}. While their Assumptions 4--8 are made for general spaces and the analysis techniques are not substantially different, ours are specifically designed to facilitate the local power analysis to be presented in the next section. We thus decide to still document these results for easy reference. 
\end{remark}

\subsection{Power analysis}

Fix $L>1$ and $s\in(0,1]$.  We consider now the problem of  testing
\[
H_0:
\P_{(\mX,Y,Z)}=\P_{0}\in\cP_{0}(L,s)
\]
against a sequence of local alternatives, 
\[
H_{1,n}:
\P_{(\mX,Y,Z)}=\P_{1,n}\in\cP_{1}(L,s).
\]

\begin{corollary}\label{crl:fruit}
Assume both $\P_0$ and $\{\P_{1,n}, n=1,2,\ldots\}$ belong to
$\cP_{\Q}$, and $\xi(\P_{1,n})$, the conditional dependence measure $\xi$ under the local alternative $\P_{1,n}$, satisfies that
\[
\xi(\P_{1,n})\gtrsim n^{-s/(p+1)+\delta}
\]
for some (arbitrarily small) constant $\delta>0$. Further assume that the number of simulations $B$ 
tends to infinity as $n\to\infty$.   Then 
\[
\lim_{n\to\infty} \P_{H_{1,n}}(\sT^{\Q,\xi_n}_\alpha=1)=1.
\]
\end{corollary}

We observe that unfortunately, even as the H\"older exponent $s$ is
small, the threshold $n^{-s/(p+1)}$ is (from a worst case perspective)
not the critical boundary in the studied nonparametric class. The
following theorem shall confirm it rigorously. To this end, we
consider a simplified setting when $p=1$, so $X,Y,Z\in\R$. Define the class
\[
\cP_1(\epsilon;L,s):=\Big\{q\in\cP_{1}(L,s): 
      \inf\limits_{q^{0}\in\cP_{0}}\lVert q - q^{0}\rVert_1\ge\epsilon\Big\},
\]
where $\lVert q - q^{0}\rVert_1:=\int \lvert q(x,y,z) - q^{0}(x,y,z)\rvert \d (x,y,z)$. Consider testing 
\[
H_0:\P_{(X,Y,Z)}=\P_{0}\in \cP_0(L,s)
\]
against the following particular sequence of local alternatives: 
\[
H_{1,n}(\Zeta):\P_{(X,Y,Z)}=\P_{1,n}(\Zeta)\in\cP_1(\Zeta n^{-2s/(4s+3)};L,s)\big\}.
\]

\begin{theorem}\label{thm:neybov}
For any $s\in(0,1]$, there exist 
\[
\P_0\in\cP_{0}(L,s)~~~~\text{and}~~~~\P_{1,n}(\Zeta)\in\cP_1(\Zeta n^{-2s/(4s+3)};L,s)
\] 
such that $\P_{(X,Y)}\in\cP_\Q$ does not vary under both the null and local alternatives, and
\begin{enumerate}[label=(\roman*)]
\item \label{thm:neybov1} assuming the number of simulations $B$ tends to infinity as $n\to\infty$, for any $\Zeta>0$ and $\alpha<0.1$, it holds that
\[
\limsup_{n\to\infty} \P_{H_{1,n}(\Zeta)}(\sT^{\Q,\xi_n}_\alpha=1)\le \beta_{\alpha},
\]
where $\beta_{\alpha}<1$ is a constant only depending on $\alpha$;

\item \label{thm:neybov2} if further $s\in[1/4,1]$, then there exists a test $\sT_{\alpha}^{\rm bin}$ such that, for any $\alpha,\beta\in(0,1)$, as long as $\Zeta$ is sufficiently large, 
\[
\P_{H_0}(\sT_{\alpha}^{\rm bin}=1)\le\alpha
~~~~\text{and}~~~~
\lim_{n\to\infty} \P_{H_{1,n}(\Zeta)}(\sT_{\alpha}^{\rm bin}=1)\ge 1-\beta;
\]
in contrast, as $\Zeta$ becomes small,
\[
\lim_{\Zeta\to 0}\lim_{n\to\infty}\TV(H_{1,n}(\Zeta), H_0)=0.
\]
\end{enumerate}
\end{theorem}

\begin{remark}
Our proof of Theorem \ref{thm:neybov}\ref{thm:neybov1} is
different from the approach we used to prove
Theorem~\ref{thm:powerless}\ref{thm:pl1}.  It
depends on the fact that $\sqrt{n}\xi_n$ has the same asymptotic mean and variance under a null hypothesis and 
a special non-standard local alternative sequence constructed in
\citet{MR4319245}. We show this in a direct calculation. 
\end{remark}


\section{Conclusion}\label{sec:conclu}

In this manuscript, we explore the use of Azadkia--Chatterjee's
conditional dependence coefficient in inferential tasks.
Specifically, we adopt the framework of conditional randomization
tests in order to study the power of Azadkia--Chatterjee-type tests of
conditional independence.  Our analyses take up two types of local
alternatives: First, a rather general quadratic mean differentiable
class and second, a rougher H\"older class.  In these settings, we
prove that the CRT-based Azadkia--Chatterjee test is unfortunately
statistically inefficient.

The current analyses are focused on the situation when $\mX$ and $Y$
are independent, 
which makes the required analysis of permutation statistics mathematically
tractable.  Indeed, while it would be natural and interesting to also study
cases where $\mX$ and $Y$ are dependent, entirely new technical tools
would need to be developed to attack this problem.  This said, we
conjecture that the inefficiency of Azadkia--Chatterjee-type test
persists for more general local alternatives, with $\mX$ and $Y$ are
dependent.

Finally, the inefficiency we demonstrate motivates further efforts to
develop variants of the considered approach.  One possible avenue
would be to develop tests that use modified versions of
Azadkia--Chatterjee's conditional dependence coefficient, in which one
uses $k$-nearest neighbor graphs with $k$ allowed to tend to infinity
as the sample size $n$ increases; in the unconditional setting recent progress in this direction was made by \citet{lin2021boosting}.  Another interesting topic for
future research would be a generalization of the coefficient to the
setting where all of $\mX,\mY,\mZ$ are multivariate.



\section{Proof of Theorem \ref{thm:powerless}}\label{sec:main-proof}

\subsection{Proof of Theorem~\ref{thm:powerless}\ref{thm:pl1}}\label{subsebsec:711}

\begin{proof}[Proof of Theorem~\ref{thm:powerless}\ref{thm:pl1}]

The proof is divided into three steps. 
The first step reviews Le~Cam's third lemma and 
introduces graph theoretic notions. 
The second step derives the distribution of $\xi_n$ under the local alternative. 
The third step computes the local power. 

{\bf Step I-1.} To derive the local alternative distribution of $\xi_n$, 
we will use Le~Cam's third lemma \citep[Theorem~7.2 and
Example~6.7]{MR1652247}. The lemma states that if under the null hypothesis, 
\[
\Big(\sqrt{n}\xi_n,\frac{1}{\sqrt{n}}\sum_{i=1}^n\dot{\ell}_0(\mX_i,Y_i,\mZ_i)\Big)
\stackrel{\sf d}{\longrightarrow}
N\bigg(\bigg(\begin{matrix}0\\0\end{matrix}\bigg),
\bigg(\begin{matrix}\sigma^2 & \tau\\
\tau & I_0\end{matrix}\bigg)\bigg)
\yestag\label{eq:lecam3-no}
\]
where $\sigma^2,\tau$ are fixed constants and $I_0:=\E\{\dot{\ell}_0(\mX,Y,\mZ)^2\}$ equals the Fisher information for $\Delta$ at $0$, then under the local alternative hypothesis, we have
\[
\sqrt{n}\xi_n
\stackrel{\sf d}{\longrightarrow}
N(\Delta_0\tau,\sigma^2).
\]
In order to employ the Cram\'er--Wold device to prove
\eqref{eq:lecam3-no} for some $\sigma^2$ and $\tau$, we need to show that under the null, for any real numbers $a$ and $b$, 
\[
a\sqrt{n}\xi_{n}+ bn^{-1/2}\sum_{i=1}^{n}\dot{\ell}_0(\mX_i,Y_i,\mZ_i)
\stackrel{\sf d}{\longrightarrow}
N\Big(0,a^2\sigma^2+2ab\tau+b^2I_0\Big).
\yestag\label{eq:cramerwold1-no}
\]
To this end, first notice that \citet[Theorem~9.1]{MR4352523} show
\[
\frac{1}{n^2}\sum_{i=1}^{n}\{R_i-\min(R_i,R_{N(i)})\}
\stackrel{\sf a.s.}{\longrightarrow}
\int\E[\Var\{\ind(Y\ge t)\given \mX\}]\d\P_{Y}(t).
\]
Therefore, by Slutsky's theorem, it suffices to establish \eqref{eq:lecam3-no} for $\sqrt{n}\widehat\xi_n$ instead of $\sqrt{n}\xi_n$, where 
\begin{align*}
\widehat\xi_n
&:=\frac{n^{-2}\sum_{i=1}^{n}\{\min(R_i,R_{M(i)})-\min(R_i,R_{N(i)})\}}
{\int\E[\Var\{\ind(Y\ge t)\given \mX\}]\d\P_{Y}(t)}.
\yestag\label{eq:xinh}
\end{align*}
Moreover, consider the ``oracle'' version of $\widehat\xi_n$ defined as
\begin{align*}
\widecheck\xi_n
&:=\frac{n^{-1}\sum_{i=1}^{n}[\min\{F_Y(Y_i),F_Y(Y_{M(i)})\}
-\min\{F_Y(Y_i),F_Y(Y_{N(i)})\}]}
{\int\E[\Var\{\ind(Y\ge t)\given \mX\}]\d\P_{Y}(t)}.
\yestag\label{eq:xint}
\end{align*}
We have the following lemma for $\widehat\xi_n$ and
$\widecheck\xi_n$.  This result and lemmas given later in this section
are derived in the supplement.
\begin{lemma}\label{lem:tech4}
Under the null hypothesis,  $\sqrt{n}\widehat\xi_n-\sqrt{n}\widecheck\xi_n=o_{\P}(1)$. 
\end{lemma}
Thus we only need to show
\[
a\sqrt{n}\widecheck\xi_{n}+ bn^{-1/2}\sum_{i=1}^{n}\dot{\ell}_0(\mX_i,Y_i,\mZ_i)
\stackrel{\sf d}{\longrightarrow}
N\Big(0,a^2\sigma^2+2ab\tau+b^2I_0\Big).
\yestag\label{eq:cramerwold1-check}
\]
The idea of proving \eqref{eq:cramerwold1-check} is to first show a conditional central limit result, 
\begin{align*}
a\sqrt{n}\widecheck\xi_n+ bn^{-1/2}\sum_{i=1}^{n}\dot{\ell}_0(\mX_i,Y_i,\mZ_i)
\Biggiven \cF_{n}
\stackrel{\sf d}{\longrightarrow}
N\Big(0,a^2\sigma^2+2ab\tau+b^2I_0\Big)\\
~~~\text{for almost every sequence $[(\mX_n,\mZ_n)]_{n\ge1}$},
\yestag\label{eq:condCLT-no}
\end{align*}
where $\cF_{n}$ denotes the $\sigma$-field generated by
$(\mX_1,\mZ_1),\dots,(\mX_n,\mZ_n)$, 
{i.e., for almost every $\omega$ of the probability space supporting the $(\mX_i,\mZ_i)$'s, 
\[
a\sqrt{n}\widecheck\xi_n\Big(\big[\big(\mX_i(\omega),Y_i,\mZ_i(\omega)\big)\big]_{i=1}^{n}\Big)+ bn^{-1/2}\sum_{i=1}^{n}\dot{\ell}_0\big(\mX_i(\omega),Y_i,\mZ_i(\omega)\big)\stackrel{\sf d}{\longrightarrow}N(0,a^2\sigma^2+2ab\tau+b^2I_0)
\]
\citep[Theorem~10.14]{MR1102015}
,}
and then deduce the desired unconditional central limit result \eqref{eq:cramerwold1-check}, and thus \eqref{eq:cramerwold1-no}.

{\bf Step I-2.} To show \eqref{eq:condCLT-no}, we 
introduce the language of graph theory. 
We write 
\begin{align*}
S_n&=a\sqrt{n}\widecheck\xi_n+ bn^{-1/2}\sum_{i=1}^{n}\dot{\ell}_0(\mX_i,Y_i,\mZ_i)\\
&=a\gamma^{-1} n^{-1/2}\sum_{i=1}^{n}\min\{F_{Y}(Y_i),F_{Y}(Y_{M(i)})\}
-a\gamma^{-1} n^{-1/2}\sum_{i=1}^{n}\min\{F_{Y}(Y_i),F_{Y}(Y_{N(i)})\}\\
&\mkern450mu +bn^{-1/2}\sum_{i=1}^{n}\dot{\ell}_0(\mX_i,Y_i,\mZ_i)\\
&=a\gamma^{-1} n^{-1/2}\sum_{i=1}^{n}\sum_{j:i\to j\in\cE(\cG_n)}K_{\wedge}(Y_i,Y_{j})
-a\gamma^{-1} n^{-1/2}\sum_{i=1}^{n}\sum_{k:i\to k\in\cE(\cG^{\mX}_n)}K_{\wedge}(Y_i,Y_{k})\\
&\mkern450mu +bn^{-1/2}\sum_{i=1}^{n}\dot{\ell}_0(\mX_i,Y_i,\mZ_i),
\end{align*}
where for all $Y$ independent of $\mX$, 
\[
\gamma:=\int\E[\Var\{\ind(Y\ge t)\given \mX\}]\d\P_{Y}(t)=\int\E[\Var\{\ind(Y\ge t)\}]\d\P_{Y}(t)=\frac16,
\]
$\cG_n$ is the directed nearest neighbor graph (NNG) of the vertices $[(\mX_i,\mZ_i)]_{i=1}^n$, 
$\cG^{\mX}_n$ is the directed nearest neighbor graph (NNG) of the vertices $[\mX_i]_{i=1}^n$, 
and $K_{\wedge}(y_1,y_2):=\min\{F_{Y}(y_1),F_{Y}(y_2)\}$. 
Next we define
\begin{align*}
&V_{i;1}:=n^{-1/2}\Big\{6a\sum_{j:i\to j\in\cE(\cG_n)}K_{\wedge}(Y_i,Y_{j})
-6a\sum_{k:i\to k\in\cE(\cG^{\mX}_n)}K_{\wedge}(Y_i,Y_{k})\Big\},~~~\\
&V_{i;2}:=n^{-1/2}b\dot{\ell}_0(\mX_i,Y_i,\mZ_i),
~~~~\text{and}~~~~
V_i:=V_{i;1}+V_{i;2},
\yestag\label{eq:Diff1-no}
\end{align*}
such that $S_n$ can be written as $\sum_{i=1}^{n}V_i$. 
Observe that, since $\big[Y_i\big]_{i=1}^{n}$ is independent of $\big[(\mX_i,\mZ_i)\big]_{i=1}^{n}$ under the null, 
\begin{align*}
&\E(V_{i;1}\given\cF_n)=n^{-1/2}\Big[6a\E\Big\{K_{\wedge}(Y_i,Y_{M(i)})\Biggiven \cF_n\Big\}
-6a\E\Big\{K_{\wedge}(Y_i,Y_{N(i)})\Biggiven \cF_n\Big\}\Big]\\
&\mkern86mu=n^{-1/2}\Big\{6a(1/3)-6a(1/3)\Big\}=0,\\
&\E(V_{i;2}\given\cF_n)=\E\Big\{\dot{\ell}_0(\mX_i,Y_i,\mZ_i)\Biggiven \cF_n\Big\}=0,
~~~~\text{by Assumption~\ref{asp:only}\ref{asp:4},}\\
\text{and}~~~~
&\E(V_{i}\given\cF_n)=a\E(V_{i;1}\given\cF_n)+b\E(V_{i;2}\given\cF_n)=0.
\end{align*}

To establish a conditional central limit theorem for $S_n$, we make use of the following lemma.

\begin{lemma}\label{lem:tech1}
  It holds that
\[
\sup_{z\in\R}\Big\lvert\Pr\Big(\frac{S_n}{\sqrt{\Var(S_n\given \cF_n)}}\le z\Biggiven\cF_n\Big)-\Phi(z)\Big\rvert 
\le 75 C_{p+q}^{5(1+\epsilon)}\frac{\E(\sum_{i=1}^{n}\lvert V_i\rvert^{2+\epsilon}\given \cF_n)}{\{\Var(S_n\given \cF_n)\}^{(2+\epsilon)/2}}
~~~~\text{a.s.,}
\yestag\label{eq:BE-no}
\]
where $C_{p+q}$ is a constant depending only on $p+q$. 
\end{lemma}

To control the right-hand side of \eqref{eq:BE-no}, we get by the
``$c_r$-inequality'' that
\[
\E\Big(\sum_{i=1}^{n}\lvert V_i\rvert^{2+\epsilon}\Biggiven \cF_n\Big)
\le 2^{1+\epsilon}
\Big\{\E\Big(\sum_{i=1}^{n}\lvert V_{i;1}\rvert^{2+\epsilon}\Biggiven \cF_n\Big)
     +\E\Big(\sum_{i=1}^{n}\lvert V_{i;2}\rvert^{2+\epsilon}\Biggiven \cF_n\Big)\Big\}.
\]
Here
\[
n^{\epsilon/2}\E\Big(\sum_{i=1}^{n}\lvert V_{i;1}\rvert^{2+\epsilon}\Biggiven \cF_n\Big)\le \lvert6a\rvert^{2+\epsilon}
~~~~\text{and}~~~~
n^{\epsilon/2}\E\Big(\sum_{i=1}^{n}\lvert V_{i;2}\rvert^{2+\epsilon}\Biggiven \cF_n\Big)
\stackrel{\sf a.s.}{\longrightarrow}
\E\Big\{\lvert b\dot{\ell}_0(\mX,Y,\mZ)\rvert^{2+\epsilon}\Big\},
\]
where the former follows from $\lvert V_{i;1}\rvert\le\lvert6a\rvert
n^{-1/2}$ and the latter from the strong law of large numbers and Assumption~\ref{asp:only}\ref{asp:5}. 

{\bf Step II.} In order to show \eqref{eq:condCLT-no}, in view of \eqref{eq:BE-no}, it suffices to show 
\[
\Var(S_n\given \cF_n)
\stackrel{\sf a.s.}{\longrightarrow}
a^2\sigma^2+2ab\tau+b^2I_0. 
\yestag\label{eq:ES-no}
\]
for some fixed $\sigma^2>0$ and $\tau$, and recall $I_0:=\E\{\dot{\ell}_0(\mX,Y,\mZ)^2\}$.
We proceed in two sub-steps. We will first compute $\Var(S_n\given \cF_n)$, 
then claim $\Var(S_n\given \cF_n)-\Var(S_n)
\stackrel{\sf a.s.}{\longrightarrow}0$ and determine the limit value of $\Var(S_n)$ accordingly. 

{\bf Step II-1.} 
Set
\begin{flalign*}
 & \gamma_{1;a}:=\E\Big[\Big\{6aK_{\wedge}(Y,Y')-2a\Big\}^2\Big],
 & \gamma_{2;a}:=\E\Big[\Big\{6aK_{\wedge}(Y,Y')-2a\Big\}\Big\{6aK_{\wedge}(Y,Y'')-2a\Big\}\Big],
 &                                                                                             \\
 & \gamma_{4;a,b}^*(\mx,\mz):=\E\Big[\Big\{6aK_{\wedge}(Y,Y')-2a\Big\}\Big\{b\dot{\ell}_0(\mx,Y,\mz)\Big\}\Big],\mkern-250mu
 & \gamma_{4;a,b}:=\E\Big\{\gamma_{4;a,b}^*(\mX,\mZ)\Big\},
 &
\\
 & \gamma_{5;b}^*(\mx,\mz):=\E\Big[\Big\{b\dot{\ell}_0(\mx,Y,\mz)\Big\}^2\Big],
 & \gamma_{5;b}:=\E\Big\{\gamma_{5;b}^*(\mX,\mZ)\Big\},
\yestag\label{eq:Diff2-no}
\end{flalign*}
where $Y'$ and $Y''$ are independent copies of $Y$.
We obtain
\[
\Var(S_n\given \cF_n)
 =\E(S_n^2\given \cF_n)
 =\sum_{i=1}^{n}\E(V_i^2\given \cF_n)+\sum_{i\ne j}\E(V_iV_j\given \cF_n),
\]
where
\begin{align*}
\sum_{i=1}^{n}\E(V_i^2\given \cF_n)&=
n^{-1}\sum_{i=1}^{n}\Big\{2\gamma_{1;a}-2\sum_{j:i\to j\in\cE(\cG_n)\cap \cE(\cG^{\mX}_n)}\gamma_{1;a}\\
&\qquad-2\sum_{(j,k):i\to j\in\cE(\cG_n), i\to k\in\cE(\cG^{\mX}_n), j\ne k}\gamma_{2;a}
+\gamma_{5;a,b}^*(\mX_i,\mZ_i)\Big\},
\yestag\label{eq:Gamma1-no}
\end{align*}
and
\begin{align*}
&\sum_{i\ne j}\E(V_iV_j\given \cF_n)\\
&=n^{-1}\Big\{
\sum_{\substack{(i,j)~\text{distinct} \\ i\to j, j\to i\in \cE(\cG_n)}}
\gamma_{1;a}
+\sum_{\substack{(i,j,k)~\text{distinct} \\ i\to k, j\to k\in \cE(\cG_n)\\
\text{or}~i\to j, j\to k\in \cE(\cG_n)\\
\text{or}~i\to k, j\to i\in \cE(\cG_n)}}
\gamma_{2;a}
+\sum_{\substack{(i,j)~\text{distinct} \\ i\to j, j\to i\in \cE(\cG^{\mX}_n)}}
\gamma_{1;a}
+\sum_{\substack{(i,j,k)~\text{distinct} \\ i\to k, j\to k\in \cE(\cG^{\mX}_n)\\
\text{or}~i\to j, j\to k\in \cE(\cG^{\mX}_n)\\
\text{or}~i\to k, j\to i\in \cE(\cG^{\mX}_n)}}
\gamma_{2;a}\\
&\qquad-2\sum_{\substack{(i,j)~\text{distinct} \\ i\to j\in \cE(\cG_n), j\to i\in \cE(\cG^{\mX}_n)}}
\gamma_{1;a}
-2\sum_{\substack{(i,j,k)~\text{distinct} \\ i\to k\in \cE(\cG_n), j\to k\in \cE(\cG^{\mX}_n)\\
\text{or}~i\to j\in \cE(\cG_n), j\to k\in \cE(\cG^{\mX}_n)\\
\text{or}~i\to k\in \cE(\cG_n), j\to i\in \cE(\cG^{\mX}_n)}}
\gamma_{2;a}\\
&\qquad+2\sum_{\substack{(i,j)~\text{distinct} \\ j\to i\in \cE(\cG_n)}}
\gamma_{4;a,b}^*(\mX_i,\mZ_i)
-2\sum_{\substack{(i,j)~\text{distinct} \\ j\to i\in \cE(\cG^{\mX}_n)}}
\gamma_{4;a,b}^*(\mX_i,\mZ_i)\Big\}.
\yestag\label{eq:Gamma2-no}
\end{align*}

{\bf Step II-2.} 
We employ the following result.
\begin{lemma}\label{lem:tech5+2}
\[
\Var(S_n\given \cF_n)-\Var(S_n)
\stackrel{\sf a.s.}{\longrightarrow}0.
\yestag\label{eq:ES1-no}
\]
\end{lemma}
Then it remains to prove 
\[
\Var(S_n)\to a^2\sigma^2+2ab\tau+b^2I_0
\yestag\label{eq:ES2-no}
\]
for some fixed $\sigma^2>0$ and $\tau$; notice \eqref{eq:ES1-no} and \eqref{eq:ES2-no} will imply \eqref{eq:ES-no}. 
To this end, in addition to Lemmas~\ref{lem:devr} and \ref{lem:henze}, 
we also need the following lemma, which is a ``covariance'' version of Lemmas~\ref{lem:devr} and \ref{lem:henze}. 

\begin{lemma}\label{lem:tech6}
Let $[\mW_i]_{i=1}^n=[(\mX_i,\mZ_i)]_{i=1}^n$ be a sample comprised of $n$ independent copies of $\mW=(\mX,\mZ)$, with $\mX\in\R^p$ and $\mZ\in\R^q$. 
Let $\cG_n$ be the directed nearest neighbor graph (NNG) of the vertices $[\mW_i]_{i=1}^n$,
and let $\cG^{\mX}_n$ be the directed nearest neighbor graph (NNG) of the vertices $[\mX_i]_{i=1}^n$. 
If random vector $\mW$ is absolutely continuous with a Lebesgue density $f$ that is continuous, 
then \\
\begin{align}
&\E\Big(n^{-1}
\sum_{\substack{(i,j)~\text{distinct} \\ i\to j\in\cE(\cG_n)\cap \cE(\cG^{\mX}_n)}}1\Big)
\to 0,\label{eq:kr}\\
&\E\Big(n^{-1}
\sum_{\substack{(i,j)~\text{distinct} \\ i\to j\in \cE(\cG_n), j\to i\in \cE(\cG^{\mX}_n)}}1\Big)
\to 0,\label{eq:ks}\\
\text{and}~~~~
&\E\Big(n^{-1}
\sum_{\substack{(i,j,k)~\text{distinct} \\ i\to k\in \cE(\cG_n), j\to k\in \cE(\cG^{\mX}_n)}}1\Big)
\to 1.\label{eq:kt}
\end{align}
\end{lemma}

Adding \eqref{eq:Gamma1-no} and \eqref{eq:Gamma2-no} together, we obtain
\begin{align*}
\Var(S_n)
&=\E\Big(2n^{-1}\sum_{i=1}^{n}\gamma_{1;a}\Big)
-\E\Big(2\sum_{\substack{(i,j)~\text{distinct} \\ i\to j\in\cE(\cG_n)\cap \cE(\cG^{\mX}_n)}}\gamma_{1;a}\Big)\\
&\qquad-\E\Big(2n^{-1}\sum_{\substack{(i,j,k)~\text{distinct} \\ i\to j\in\cE(\cG_n), i\to k\in\cE(\cG^{\mX}_n)}}\gamma_{2;a}\Big)
+\E\Big\{n^{-1}\sum_{i=1}^{n}\gamma_{5;a,b}^*(\mX_i,\mZ_i)\Big\}\\
&\qquad+\E\Big(n^{-1}
\sum_{\substack{(i,j)~\text{distinct} \\ i\to j, j\to i\in \cE(\cG_n)}}
\gamma_{1;a}\Big)
+\E\Big(n^{-1}\sum_{\substack{(i,j,k)~\text{distinct} \\ i\to k, j\to k\in \cE(\cG_n)\\
\text{or}~i\to j, j\to k\in \cE(\cG_n)\\
\text{or}~i\to k, j\to i\in \cE(\cG_n)}}
\gamma_{2;a}\Big)\\
&\qquad+\E\Big(n^{-1}\sum_{\substack{(i,j)~\text{distinct} \\ i\to j, j\to i\in \cE(\cG^{\mX}_n)}}
\gamma_{1;a}\Big)
+\E\Big(n^{-1}\sum_{\substack{(i,j,k)~\text{distinct} \\ i\to k, j\to k\in \cE(\cG^{\mX}_n)\\
\text{or}~i\to j, j\to k\in \cE(\cG^{\mX}_n)\\
\text{or}~i\to k, j\to i\in \cE(\cG^{\mX}_n)}}
\gamma_{2;a}\Big)\\
&\qquad-\E\Big(2n^{-1}\sum_{\substack{(i,j)~\text{distinct} \\ i\to j\in \cE(\cG_n), j\to i\in \cE(\cG^{\mX}_n)}}
\gamma_{1;a}\Big)
-\E\Big(2n^{-1}\sum_{\substack{(i,j,k)~\text{distinct} \\ i\to k\in \cE(\cG_n), j\to k\in \cE(\cG^{\mX}_n)\\
\text{or}~i\to j\in \cE(\cG_n), j\to k\in \cE(\cG^{\mX}_n)\\
\text{or}~i\to k\in \cE(\cG_n), j\to i\in \cE(\cG^{\mX}_n)}}
\gamma_{2;a}\Big)\\
&\qquad+\E\Big\{2n^{-1}\sum_{\substack{(i,j)~\text{distinct} \\ j\to i\in \cE(\cG_n)}}
\gamma_{4;a,b}^*(\mX_i,\mZ_i)\Big\}
-\E\Big\{2n^{-1}\sum_{\substack{(i,j)~\text{distinct} \\ j\to i\in \cE(\cG^{\mX}_n)}}
\gamma_{4;a,b}^*(\mX_i,\mZ_i)\Big\}.
\yestag\label{eq:Gamma-no}
\end{align*}
The first term is $2\gamma_{1;a}$. 
The second term tends to $0$ by Equation \eqref{eq:kr} in Lemma~\ref{lem:tech6}. 
For the third term, we have
\begin{align*}
\E\Big(2n^{-1}\sum_{\substack{(i,j,k)~\text{distinct} \\ i\to j\in\cE(\cG_n), i\to k\in\cE(\cG^{\mX}_n)}}\gamma_{2;a}\Big)
&=2\gamma_{2;a}\E\Big(n^{-1}\sum_{\substack{(i,j)~\text{distinct} \\ i\to j\in\cE(\cG_n), i\to j\not\in\cE(\cG^{\mX}_n)}}1\Big)\\
&=2\gamma_{2;a}\E\Big(1-n^{-1}\sum_{\substack{(i,j)~\text{distinct} \\ i\to j\in\cE(\cG_n), i\to j\in\cE(\cG^{\mX}_n)}}1\Big)
\to2\gamma_{2;a},
\end{align*}
where the last step is by Equation \eqref{eq:kr}.
The fourth term is $\gamma_{5;a}$. 
The fifth term tends to $\gamma_{1;a}\kq_{p+q}$ by Lemma~\ref{lem:devr}. 
The sixth term can be rewritten as
\begin{align*}
&\gamma_{2;a}
\E\Big(n^{-1}\sum_{\substack{(i,j,k)~\text{distinct} \\ i\to k, j\to k\in \cE(\cG_n)}}1
+n^{-1}\sum_{\substack{(i,j,k)~\text{distinct} \\ i\to j, j\to k\in \cE(\cG_n)}}1
+n^{-1}\sum_{\substack{(i,j,k)~\text{distinct} \\ i\to k, j\to i\in \cE(\cG_n)}}1\Big)\\
&=\gamma_{2;a}
\E\Big(n^{-1}\sum_{\substack{(i,j,k)~\text{distinct} \\ i\to k, j\to k\in \cE(\cG_n)}}1
+n^{-1}\sum_{\substack{(i,j)~\text{distinct} \\ i\to j\in \cE(\cG_n), j\to i\not\in \cE(\cG_n)}}1
+n^{-1}\sum_{\substack{(i,j)~\text{distinct} \\ i\to j\not\in \cE(\cG_n), j\to i\in \cE(\cG_n)}}1\Big)\\
&=\gamma_{2;a}
\E\Big\{n^{-1}\sum_{\substack{(i,j,k)~\text{distinct} \\ i\to k, j\to k\in \cE(\cG_n)}}1
+\Big(1-n^{-1}\sum_{\substack{(i,j)~\text{distinct} \\ i\to j, j\to i\in \cE(\cG_n)}}1\Big)
+\Big(1-n^{-1}\sum_{\substack{(i,j)~\text{distinct} \\ i\to j, j\to i\in \cE(\cG_n)}}1\Big)\Big\}\\
&\to\gamma_{2;a}\Big(\ko_{p+q} +2-2\kq_{p+q}\Big),
\end{align*}
where the last step is by Lemmas~\ref{lem:tech3} and \ref{lem:devr}. 
Similarly, the seventh and eighth terms tend to $\gamma_{1;a}\kq_{p}$
and $\gamma_{2;a}(\ko_{p} +2-2\kq_{p})$, respectively. 
The ninth term tends to $0$ by Equation \eqref{eq:ks} in Lemma~\ref{lem:tech6}. 
The tenth term is equal to
\begin{align*}
&2\gamma_{2;a}
\E\Big(n^{-1}\sum_{\substack{(i,j,k)~\text{distinct} \\ i\to k\in \cE(\cG_n), j\to k\in \cE(\cG^{\mX}_n)}}1
+n^{-1}\sum_{\substack{(i,j,k)~\text{distinct} \\ i\to j\in \cE(\cG_n), j\to k\in \cE(\cG^{\mX}_n)}}1
+n^{-1}\sum_{\substack{(i,j)~\text{distinct} \\ i\to k\in \cE(\cG_n), j\to i\in \cE(\cG^{\mX}_n)}}1\Big)\\
&=2\gamma_{2;a}
\E\Big(n^{-1}\sum_{\substack{(i,j,k)~\text{distinct} \\ i\to k\in \cE(\cG_n), j\to k\in \cE(\cG^{\mX}_n)}}1
+n^{-1}\sum_{\substack{(i,j)~\text{distinct} \\ i\to j\in \cE(\cG_n), j\not\to i\in \cE(\cG^{\mX}_n)}}1
+n^{-1}\sum_{\substack{(i,j)~\text{distinct} \\ i\not\to j\in \cE(\cG_n), j\to i\in \cE(\cG^{\mX}_n)}}1\Big)\\
&=2\gamma_{2;a}
\E\Big\{n^{-1}\sum_{\substack{(i,j,k)~\text{distinct} \\ i\to k\in \cE(\cG_n), j\to k\in \cE(\cG^{\mX}_n)}}1\\
&\qquad+\Big(1-n^{-1}\sum_{\substack{(i,j)~\text{distinct} \\ i\to j\in \cE(\cG_n), j\to i\in \cE(\cG^{\mX}_n)}}1\Big)
+\Big(1-n^{-1}\sum_{\substack{(i,j)~\text{distinct} \\ i\to j\in \cE(\cG_n), j\to i\in \cE(\cG^{\mX}_n)}}1\Big)\Big\}\\
&\to 2\gamma_{2;a}\{1+(1-0)+(1-0)\}=6\gamma_{2;a},
\end{align*}
where the second last step is by Lemma~\ref{lem:tech6}. 
For the last two terms, we obtain
\begin{align*}
\E\Big\{n^{-1}\sum_{\substack{(i,j)~\text{distinct} \\ j\to i\in \cE(\cG_n)}}
\gamma_{4;a,b}^*(\mX_i,\mZ_i)\Big\}
&=\E\Big\{n^{-1}\sum_{i=1}^{n}
\gamma_{4;a,b}^*(\mX_{M(i)},\mZ_{M(i)})\Big\}\\
&=\E\Big\{\gamma_{4;a,b}^*(\mX_{M(1)},\mZ_{M(1)})\Big\}
\to\E\Big\{\gamma_{4;a,b}^*(\mX_{1},\mZ_{1})\Big\}=\gamma_{4;a,b},
\end{align*}
where the second last step can be deduced in view of 
Lemmas~11.5 and 11.7 in \citet{MR4352523} and 
uniform integrability by way of \citet[Chap.~3, Exercise~5.4]{MR3701383},
and similarly
\[
\E\Big\{n^{-1}\sum_{\substack{(i,j)~\text{distinct} \\ j\to i\in \cE(\cG^{\mX}_n)}}
\gamma_{4;a,b}^*(\mX_i,\mZ_i)\Big\}
\to\gamma_{4;a,b}.
\]
Thus the last two terms are canceled out.
Plugging all terms back into \eqref{eq:Gamma-no} yields
\begin{align*}
\Var(S_n)
&\to\gamma_{1;a}\Big\{2+\Big(\kq_{p+q}+\kq_{p}\Big)\Big\}
+\gamma_{2;a}\Big\{-4
-2\Big(\kq_{p+q}+\kq_{p}\Big)
+\Big(\ko_{p+q}+\ko_{p}\Big)
\Big\}+\gamma_{5;b}.
\end{align*}
In addition, 
\[
\gamma_{1;a}=2a^2, ~~~~~~~~~~~~
\gamma_{2;a}=4a^2/5, ~~~~~~~~~~~~
\gamma_{5;b}=b^2 I_0.
\]
Therefore, 
\begin{align*}
\Var(S_n)
\to
a^2\Big\{\frac45
+\frac25\Big(\kq_{p+q}+\kq_{p}\Big)
+\frac45\Big(\ko_{p+q}+\ko_{p}\Big)
\Big\}
+b^2I_0.
\end{align*}
This completes the proof of \eqref{eq:condCLT-no} and thus
\eqref{eq:cramerwold1-no} and \eqref{eq:lecam3-no} with
\[
\sigma^2:=\frac45
+\frac25\Big(\kq_{p+q}+\kq_{p}\Big)
+\frac45\Big(\ko_{p+q}+\ko_{p}\Big)
~~~~\text{and}~~~~
\tau:=0.
\]

Finally, by Le~Cam's third lemma, under $\{\P_{n,\Delta_n}\}_{n\ge1}$,
\[
\sqrt{n}\xi_n\Big(\big[(\mX_i,Y_i,\mZ_i)\big]_{i=1}^{n}\Big)
\stackrel{\sf d}{\longrightarrow}
N(0,\sigma^2).
\]
Moreover, we also have
\[
\sqrt{n}\xi_n\Big(\big[(\mX_i,Y_i^{(b)},\mZ_i)\big]_{i=1}^{n}\Big)
\stackrel{\sf d}{\longrightarrow}
N(0,\sigma^2). 
\]

{\bf Step III.} 
Since $B$ tends to infinity as $n\to\infty$, without loss of generality, we assume $B> \alpha^{-1}-1$. 
With the shorthand notation
\[
\xi_n^{(b)} \equiv \xi_n\Big(\big[(\mX_i,Y_i^{(b)},\mZ_i)\big]_{i=1}^{n}\Big)
~~~~\text{and}~~~~
\xi_n^\circ \equiv \xi_n\Big(\big[(\mX_i,Y_i,\mZ_i)\big]_{i=1}^{n}\Big),
\]
the test
\[
\sT^{\Q,\xi_n}_\alpha:=
\ind\Big(\frac{1+\sum_{b=1}^{B}\ind(\xi_n^{(b)}\ge \xi_n^\circ)}{1+B}\le\alpha\Big)
\]
can be restated as
\[
\sT^{\Q,\xi_n}_\alpha=\ind\Big(\sqrt{n}\xi_n^\circ > \sqrt{n}\xi_n^{[1+B-\lfloor\alpha(1+B)\rfloor]}\Big),
\yestag\label{eq:yexu1}
\]
where $\xi_n^{[1]},\xi_n^{[2]},\dots,\xi_n^{[B]}$ is
a rearrangement of $\xi_n^{(1)},\xi_n^{(2)},\dots,\xi_n^{(B)}$ such that
\[
\xi_n^{[1]} \le \xi_n^{[2]} \le \cdots \le \xi_n^{[B]}.
\]
Write $\Phi_{\sigma}(\cdot)$ and $\Phi_{\sigma}^{-1}(\cdot)$ for the cumulative distribution function and quantile function of the normal distribution with mean zero and variance $\sigma^2$. 
We wish to prove 
\[
\sqrt{n}\xi_n^{[1+B-\lfloor\alpha(1+B)\rfloor]}\stackrel{\sf p}{\longrightarrow}\Phi_{\sigma}^{-1}(1-\alpha).
\yestag\label{eq:yexu2}
\] 
Using Theorem~3.1 in \citet{MR57521}, it suffices to prove
\[
B^{-1}\sum_{b=1}^{B}\ind(\sqrt{n}\xi_n^{(b)}\le y)\stackrel{\sf p}{\longrightarrow}\Phi_{\sigma}(y),
\]
which is immediate from
\[
B^{-1}\sum_{b=1}^{B}\ind(\sqrt{n}\xi_n^{(b)}\le y)\Biggiven\cF_n
\stackrel{\sf p}{\longrightarrow}\Phi_{\sigma}(y)
~~~\text{for almost every sequence $[(\mX_n,\mZ_n)]_{n\ge1}$}. 
\]
We obtain that
\begin{align*}
\lim_{n\to\infty} \P_{H_{1,n}(\Delta_0),\Q}\Big(\sT^{\Q,\xi_n}_\alpha=1\Big)
&=\lim_{n\to\infty} \P_{H_{1,n}(\Delta_0),\Q}\Big(\sqrt{n}\xi_n^\circ > \sqrt{n}\xi_n^{[1+B-\lfloor\alpha(1+B)\rfloor]}\Big)\\
&=\lim_{n\to\infty} \P_{H_{1,n}(\Delta_0),\Q}\Big(\sqrt{n}\xi_n^\circ > \Phi_{\sigma}^{-1}(1-\alpha)\Big)=\alpha.
\end{align*}
This completes the proof. 
\end{proof}

\subsection{Proof of Theorem~\ref{thm:powerless}\ref{thm:pl3}}

\begin{proof}[Proof of Theorem~\ref{thm:powerless}\ref{thm:pl3}]
Given that $Y$ is independent of $\mX$, the conditional independence between $Y$ and $\mZ$ given $\mX$ is equivalent to the (unconditional) independence between $Y$ and $\mW=(\mX,\mZ)$. 

To test the independence between $Y\in\R^1$ and $\mW\in\R^{p+q}$, we will adopt the test proposed in \citet[Equation~(13)]{MR4399094}; see \citet{deb2019multivariate} for a similar result. We will briefly illustrate the idea. 

%

Let $(Y_1,\mW_1),\dots,(Y_n,\mW_n)$ be independent copies of $(Y,\mW)$. 
Let $\fF_{Y,\pm}^{(n)}$ and $\fF_{\mW,\pm}^{(n)}$ be the empirical center-outward distribution functions as defined in \citet[Definition 2.3]{MR4255122}
for $\{Y_i\}_{i=1}^{n}$ and $\{\mW_i\}_{i=1}^{n}$, respectively.
We define the test statistic
\[
\widehat M_n:=n\cdot \dCov^2_n\Big([\fF_{\mX,\pm}^{(n)}(\mX_i)]_{i=1}^{n},[\fF_{\mY,\pm}^{(n)}(\mY_i)]_{i=1}^{n}\Big),
\]
where the (sampled) distance covariance $\dCov^2_n(\cdot,\cdot)$ is given in \citet[Definition~1]{MR3055745}, 
and then form the test
\[
\mathsf{T}_{\alpha}^{\rm opt}:= \ind\Big(\widehat M_n>q_{1-\alpha}\Big),
~~~~q_{1-\alpha}:=\inf\Big\{x\in\R:\P\Big(\sum_{k=1}^{\infty}\lambda_k(\xi_k^2-1)\le x\Big)\ge 1-\alpha \Big\}.
\]
Here, $\lambda_k$, $k\in\Z_+$, are the non-zero eigenvalues of the integral equation
given by \citet[Equation~(12)]{MR4399094} and depend only on
$p+q$, and 
$[\xi_k]_{k=1}^{\infty}$ is a sequence of independent standard
Gaussian random variables.  Further details can be found in \citet{MR4399094}.

By Theorem 3.1 in \citet{MR4399094},
\[
\lim_{n\to\infty} \P_{H_0}(\sT_{\alpha}^{\rm opt}=1)\le\alpha,
\]
and by Theorem 5.3 in \citet{MR4474478}, for sufficiently large $\Delta_0$,
\[
\lim_{n\to\infty} \P_{H_{1,n}(\Delta_0)}(\sT_{\alpha}^{\rm opt}=1)\ge 1-\beta.
\]

Finally,  we prove  that
\[
\lim_{\Delta_0\to 0}\lim_{n\to\infty}\TV(H_0, H_{1,n}(\Delta_0))=0.
\]
Equation~(2.20) in \citet{MR2724359} states that total variation and Hellinger distances satisfy 
\[
\TV(H_{1,n}(\Delta_0), H_0)\le\HL(H_{1,n}(\Delta_0), H_0).
\] 
It is also known \citep[p.~83]{MR2724359}  that
\[
1-\frac{\HL^2(H_{1,n}(\Delta_0), H_0)}{2}=\Big(1-\frac{\HL^2(\P_{1,n}(\Delta_0), \P_0)}{2}\Big)^n.
\]
\citet[Example~13.1.1]{MR2135927} show that, under Assumption~\ref{asp:only}, 
\[
n\times \HL^2(\P_{1,n}(\Delta_0), \P_0))
\to \frac{\Delta_0^2\cI_{\mX}(0)}{4};
\]
notice that here the definition of $\HL^2(\Q,\P)$ differs from that in \citet[Definition~13.1.3]{MR2135927} by a factor of $2$. 
Therefore, 
\[
\frac{\HL^2(H_{1,n}(\Delta_0), H_0)}{2} \longrightarrow 1-\exp\Big\{-\frac{\Delta_0^2\cI_{\mX}(0)}{8}\Big\}.
\]
where the right-hand side tends to $0$ as $\Delta_0\to0$. 
The last assertion is a direct corollary of the fact that the sum of probabilities of Type I error and Type II error has the following lower bound:
\[
\inf_{\sT} \Big\{P_{H_0}(\sT=1) + \P_{H_{1,n}(\Delta_0)}(\sT=0)\Big\} = 1-\TV(H_{1,n}(\Delta_0), H_0)
\]
\citep[Theorem~13.1.1]{MR2135927}. 
This completes the proof.
\end{proof}

\section*{Acknowledgments}
The authors would like to thank two anonymous referees, an anonymous 
Associate Editor, and the Editor Mark Podolskij for their stimulating comments, 
which highly improved the quality of this paper.

\section*{Funding}
The authors have received funding from the United States NSF Grants
DMS-1712536 and SES-2019363 and the European Research Council (ERC) under the European Union's Horizon 2020 research and innovation programme (grant agreement No 883818).

\begin{appendix}

\section{Proof}\label{sec:proof}

\noindent{\bf Additional notation.}
For a function $f:\cX\to \R$, we define $\norm{f}_{\infty}:= \max_{x\in\cX}|f(x)|$. 
We use 
$\stackrel{\sf d}{\longrightarrow}$, 
$\stackrel{\sf p}{\longrightarrow}$, and 
$\stackrel{\sf a.s.}{\longrightarrow}$ to 
denote convergence in distribution, convergence in probability, 
and almost sure convergence, respectively.  
For any two real sequences $[a_n]_n$ and $[b_n]_n$, we write $a_n=O(b_n)$ if there exists $C>0$ such that $|a_n|\le C|b_n|$ for all $n$ large enough, and $a_n=o(b_n)$ if for any $c>0$, $|a_n|\le c|b_n|$ holds for all $n$ large enough. 
For a sequence of random variables $[X_n]_n$ and a real sequence $[a_n]_n$, we write $X_n=O_{\P}(a_n)$ if for any $\epsilon>0$ there exists $C>0$ such that $\P(|X_n|\ge C|a_n|)<\epsilon$ for all $n$ large enough, and $X_n=o_{\P}(a_n)$ if for any $c>0$, $\lim_{n\to\infty}\P(|X_n|\ge c|a_n|)=0$.

\subsection{Proofs for Section~\ref{sec:prelim}}

\subsubsection{Proof of Proposition~\ref{prop:vacon}}

\begin{proof}[Proof of Proposition~\ref{prop:vacon}]

We will use the shorthand notation
\[
\xi_n^{(b)} \equiv \xi_n\Big(\big[(\mX_i,Y_i^{(b)},\mZ_i)\big]_{i=1}^{n}\Big),
~~~~
\xi_n^\circ \equiv \xi_n\Big(\big[(\mX_i,Y_i,\mZ_i)\big]_{i=1}^{n}\Big),
~~~~\text{and}~~~~
\Xi_n^{(b)}:=\ind\Big(\xi_n^{(b)}\ge \xi_n^\circ\Big),
\]
and we have
\[
p_{\rm CRT}=(1+B)^{-1}+(1+B)^{-1}\sum_{b=1}^{B}\Xi_n^{(b)}. 
\yestag\label{eq:pprime}
\]

Claim \ref{prop:vacon-1} is a  corollary of \citet[Lemma~F.1]{MR3798878supp}. 
Notice that
\[
p_{\rm CRT}\ge (1+B)^{-1}+(1+B)^{-1}\sum_{b=1}^{B}\ind^{\rm rtb}(\xi_n^{(b)},\xi_n^\circ;U_b)=:p_{\rm CRT}^{\rm rtb},
\]
where $U_1,\dots,U_B$ are independent Bernoulli random variables of equal probabilities to be $0$ or $1$, and
\[
\ind^{\rm rtb}(x,y;u):=\begin{cases}\ind(x\ge y), & \text{if } x\ne y,\\
u, & \text{if } x=y.\end{cases}
\]
Under the null hypothesis and conditionally on 
$\big[(\mX_i,\mZ_i)\big]_{i=1}^{n}$, we have
\[
\xi_n^{(1)},\dots,\xi_n^{(B)},\xi_n^\circ
\]
are independent and identically distributed, 
and accordingly $p_{\rm CRT}^{\rm rtb}$ is discretely uniformly distributed over
\[
\Big\{\frac{1}{1+B},\frac{2}{1+B},\dots,\frac{1+B}{1+B}\Big\}.
\]
As a consequence,
\begin{align*}
\P_{H_0}\Big(\sT^{\Q,\xi_n}_\alpha=1\Biggiven \big[(\mX_i,\mZ_i)\big]_{i=1}^{n}\Big)
&=\P_{H_0}\Big(p_{\rm CRT}\le\alpha\Biggiven \big[(\mX_i,\mZ_i)\big]_{i=1}^{n}\Big)\\
&\le\P_{H_0}\Big(p_{\rm CRT}^{\rm rtb}\le\alpha\Biggiven \big[(\mX_i,\mZ_i)\big]_{i=1}^{n}\Big)
=\frac{\lfloor \alpha(1+B)\rfloor}{1+B}
\le\alpha. 
\end{align*}
Since this inequality holds conditionally, it also holds unconditionally.

It remains to prove Claim \ref{prop:vacon-2}, the consistency of $\sT^{\Q,\xi_n}_\alpha$. 
Since by Propositions~\ref{prop:AC1} and \ref{prop:AC2}, under the fixed alternative $H_1$, 
\[
\xi_n^{(b)}\stackrel{\sf a.s.}{\longrightarrow} 0
~~~~\text{and}~~~~
\xi_n^\circ\stackrel{\sf a.s.}{\longrightarrow} \xi_{H_1}
\]
where $0<\xi_{H_1}\le 1$, we obtain 
\[
\Xi_n^{(b)}=\ind\Big\{\xi_n^{(b)}\ge 
          \xi_n^\circ\Big\}
\stackrel{\sf a.s.}{\longrightarrow}0.
\yestag\label{eq:eachb}
\]
Recall that $B=B_n$ tends to infinity as $n\to\infty$. 
Since random variables $\Xi_n^{(1)},\Xi_n^{(2)},\dots,\Xi_n^{(B_n)}$ are exchangeable, 
applying Lemma~1.1 in \citet{MR790492} yields 
\[
B_n^{-1}\sum_{b=1}^{B_n}\Xi_n^{(b)}
=\E\Big(\Xi_n^{(1)}\Biggiven \cG_n\Big)~~~~\text{a.s.},
\yestag\label{eq:patterson}
\]
where $\cG_n$ is the $\sigma$-field generated as
\[
\cG_n:=\sigma\Big(\sum_{b=1}^{B_n}\Xi_n^{(b)},
                  \sum_{b=1}^{B_{n+1}}\Xi_{n+1}^{(b)},
                  \dots\Big).
\]
Notice that 
 (i) $[\cG_n]_{n=1}^{\infty}$ is a decreasing sequence of $\sigma$-fields
with $\cG_n\to\cG_{\infty}$ where $\cG_{\infty}:=\bigcap_{n=1}^{\infty}\cG_n$,
(ii) $0\le\Xi_n^{(1)}\le1$,
and (iii) $\Xi_n^{(1)}\stackrel{\sf a.s.}{\longrightarrow} 0$.  
Using Lemma~2(c) in \citet{MR548906}, 
we obtain
\[
\E\Big(\Xi_n^{(1)}\Biggiven \cG_n\Big)
\stackrel{\sf a.s.}{\longrightarrow} 
\E\Big(0\Biggiven \cG_{\infty}\Big)=0.
\yestag\label{eq:isaac}
\]
Combining \eqref{eq:patterson} and \eqref{eq:isaac}, we deduce
\[
B_n^{-1}\sum_{b=1}^{B_n}\Xi_n^{(b)}
\stackrel{\sf a.s.}{\longrightarrow} 0,
\]
and moreover, in \eqref{eq:pprime} that
\[
p_{\rm CRT}\stackrel{\sf a.s.}{\longrightarrow} 0,
\]
Therefore, 
\[
\lim_{n\to\infty} \P_{H_1,\Q}(\sT^{\Q,\xi_n}_\alpha=1)
=\lim_{n\to\infty} \P_{H_1,\Q}(p_{\rm CRT}\le \alpha)
=1,
\]
and the proof is complete.
\end{proof}

\subsection{Proofs for Section~\ref{sec:asynormal}}

\subsubsection{Proof of Theorem~\ref{thm:byprod}}

\begin{proof}[Proof of Theorem~\ref{thm:byprod}]

Claim~\ref{thm:byprod1} can be proved in view of the proof of Theorem~\ref{thm:powerless}\ref{thm:pl1}. 

We next give a proof of Claim~\ref{thm:byprod2}. 
When $Y$ and $\mZ$ are both absolutely continuous, we have
\[
\xi^\#_n
=\frac{n\sum_{i=1}^{n}\min(R_i,R_{M(i)}) -\sum_{i=1}^{n}i^2}
{\sum_{i=1}^{n}i(n-i)}
=\frac{n^{-2}\sum_{i=1}^{n}\min(R_i,R_{M(i)}) -(1+n^{-1})(2+n^{-1})/6}
{(1-n^{-2})/6}.
\]
Moreover, in view of Equation \eqref{eq:hajek1} in the Proof of Lemma~\ref{lem:tech4}, we have under the null, 
\[\sqrt{n}\xi^\#_n-\sqrt{n}\xi^\dagger_n=o_{\P}(1),\]
where
\[
\xi^\dagger_n
:=\frac{n^{-1}\sum_{i=1}^{n}\min\{F_{Y}(Y_i),F_{Y}(Y_{M(i)})\} -
\{(n(n-1)\}^{-1}\sum_{i\ne j}\min\{F_{Y}(Y_i),F_{Y}(Y_j)\}}
{1/6}.
\]
In view of the proof of Theorem~\ref{thm:powerless}\ref{thm:pl1}, to establish the central limit theorem for $\sqrt{n}\xi^\dagger_n$ and thus $\sqrt{n}\xi^\#_n$, it suffices to determine the limit of
$
\Var(\sqrt{n}\xi^\dagger_n).
$
Set $K_{\wedge}(y_1,y_2):=\min\{F_{Y}(y_1),F_{Y}(y_2)\}$, 
\[
\gamma_{1}:=\E\Big[\Big\{6K_{\wedge}(Y,Y')-2\Big\}^2\Big]=2~~~~\text{and}~~~~
\gamma_{2}:=\E\Big[\Big\{6K_{\wedge}(Y,Y')-2\Big\}\Big\{6K_{\wedge}(Y,Y'')-2\Big\}\Big]=4/5,
\]
where $Y'$ and $Y''$ are independent copies of $Y$. 
Then
\begin{align*}
\Var(\sqrt{n}\xi^\dagger_n)
&\to\E\Big(n^{-1}\sum_{i=1}^{n}\gamma_{1}\Big)
+\E\Big(n^{-1}
\sum_{\substack{(i,j)~\text{distinct} \\ i\to j, j\to i\in \cE(\cG_n)}}
\gamma_{1}\Big)
+\E\Big(n^{-1}\sum_{\substack{(i,j,k)~\text{distinct} \\ i\to k, j\to k\in \cE(\cG_n)\\
\text{or}~i\to j, j\to k\in \cE(\cG_n)\\
\text{or}~i\to k, j\to i\in \cE(\cG_n)}}
\gamma_{2}\Big)-4\gamma_{2}\\
&\to\gamma_{1}+\gamma_{1}\kq_{q}+\gamma_{2}\Big(\ko_{q}+2-2\kq_{q}\Big)-4\gamma_{2}
 =\frac25+\frac25\kq_{q}+\frac45\ko_{q}.
\end{align*}
This completes the proof. 
\end{proof}

\subsubsection{Proof of Lemma~\ref{lem:tech3}}

\begin{proof}[Proof of Lemma~\ref{lem:tech3}]
Using Lemmas~\ref{lem:MD} and \ref{lem:henze} yields
\[
\E\Big(n^{-1}\sum_{\substack{(i,j,k)~\text{distinct} \\ i\to k, j\to k\in \cE(\cG_n)}}1\Big)
 =\E\Big\{n^{-1}\sum_{k=1}^{n}d_k^{-}(d_k^{-}-1)\Big\}
 =\E\{d_1^{-}(d_1^{-}-1)\}
 =\Var(d_1^{-})
 \to\kd_{d;2},
\]
where $d^{-}_k := \sum_{i:i\to k\in \cE(\cG_n)} 1$ is the in-degree of vertex $(\mX_k,\mZ_k)$ in $\cG_n$. 
\end{proof}


%
%
%
%
%
%
%
%
%
%
%
%
%
%
%
%
%
%
%
%
%
%
%
%
%
%

\subsection{Proofs for Section~\ref{sec:subopt2}}

\subsubsection{Proof of Proposition~\ref{prop:ACHolder}}

\begin{proof}[Proof of Proposition~\ref{prop:ACHolder}]
We first prove $\xi_n -\xi =O_{\P}\big({(\log n)^{p+q+1}}\big/{n^{s/(p+q)}}\big)$. 
Following the Proof of Theorem~4.1 in \citet[Sec.~14]{MR4352523}, we first generalize Lemma~14.1 in \citet[Sec.~14]{MR4352523}, i.e., show that there is some $C_3$ depending only on $C_1$, $C_2$, $p$ and $s$ such that
\[
\E\{\min(\lVert\mX_1-\mX_{N(1)}\rVert_{}^s,1)\}
\le \begin{cases} C_3 n^{-1}(\log n)^2 & \text{if}~p = 1~\text{and}~s=1, \\ 
                  C_3 n^{-s/p}\log n   & \text{otherwise}. \end{cases}
\yestag\label{eq:boundps}
\]
In view of the Proof of Lemma~14.1 in \citet[Sec.~14]{MR4352523}, we get
\[
\P(\lVert\mX_1-\mX_{N(1)}\rVert_{}\ge \epsilon)
\le (1-\delta)^{n-1}+C_pC_2^p\epsilon^{-p}\delta 
\]
and, by taking $\delta=n^{-1}\log n$, find that
\[
\P(\lVert\mX_1-\mX_{N(1)}\rVert_{}\ge \epsilon)
\le \frac{C_4\log n}{n\epsilon^{p}}
\]
for some $C_4$ depending only on $C_1$, $C_2$ and $p$. Thus,
\begin{align*}
&\E\{\min(\lVert\mX_1-\mX_{N(1)}\rVert_{}^s,1)\}\\
&=\int_{0}^{n^{-s/p}}\P(\lVert\mX_1-\mX_{N(1)}\rVert_{}^s\ge \epsilon)\d \epsilon
 +\int_{n^{-s/p}}^{1}\P(\lVert\mX_1-\mX_{N(1)}\rVert_{}^s\ge \epsilon)\d \epsilon\\
&\le \int_{0}^{n^{-s/p}} 1\d \epsilon
    +\int_{n^{-s/p}}^{1} \frac{C_4\log n}{n\epsilon^{p/s}}\d \epsilon
 =n^{-s/p}+\frac{C_4\log n}{n}\int_{n^{-s/p}}^{1} \epsilon^{-p/s}\d \epsilon.
\end{align*}
This completes the proof of \eqref{eq:boundps}.
Next in view of the Proof of Lemma~14.2 and Theorem~4.1 in \citet[Sec.~14]{MR4352523}, the desired result follows. 
\end{proof}

\subsubsection{Proof of Lemma~\ref{lem:condmore}}

\begin{proof}[Proof of Lemma~\ref{lem:condmore}]
We first prove Claim~\ref{lem:condmore1}. 
The conditional independence hypothesis $Y \indep Z \given \mX$ 
implies that $q_{Y \given \mX,Z}(y \given \mx,z)=q_{Y \given \mX}(y \given \mx)$. The rest is obvious. 
Then we prove Claim~\ref{lem:condmore2}. 
Since
\[
\Big\lvert q_{Y, Z \given \mX}(y,z\given \mx) - 
q_{Y, Z \given \mX}(y',z'\given \mx)\Big\rvert
\le L\Big(\lvert y-y'\rvert^s+\lvert z-z'\rvert^s\Big),
\]
it holds that
\begin{align*}
\Big\lvert q_{Z \given \mX}(z\given \mx) - 
q_{Z \given \mX}(z'\given \mx)\Big\rvert
&= \Big\lvert \int_{[0,1]}q_{Y, Z \given \mX}(y,z\given \mx) \d y - 
\int_{[0,1]}q_{Y, Z \given \mX}(y,z'\given \mx) \d y \Big\rvert\\
&\le \int_{[0,1]}  \Big\lvert q_{Y, Z \given \mX}(y,z\given \mx) - 
q_{Y, Z \given \mX}(y,z'\given \mx)\Big\rvert \d y
 \le L \lvert z-z'\rvert^s. 
\end{align*}
Next, since
\[
\Big\lvert
q_{Y,Z \given \mX}(y,z \given \mx) - 
q_{Y,Z \given \mX}(y,z \given \mx')
\Big\rvert
\le L\lVert \mx - \mx'\rVert_{}^s,
\]
we have
\begin{align*}
\Big\lvert q_{Z \given \mX}(z\given \mx) - 
q_{Z \given \mX}(z\given \mx')\Big\rvert
&= \Big\lvert \int_{[0,1]}q_{Y, Z \given \mX}(y,z\given \mx) \d y - 
\int_{[0,1]}q_{Y, Z \given \mX}(y,z\given \mx')\Big\rvert \d y \\
&\le \int_{[0,1]}  \Big\lvert q_{Y, Z \given \mX}(y,z\given \mx) - 
q_{Y, Z \given \mX}(y,z\given \mx')\Big\rvert \d y
 \le L\lVert \mx - \mx'\rVert_{}^s.
\end{align*}
Also,
$
L^{-1} \le q_{Y, Z \given \mX}(y,z\given \mx) \le L
$
implies 
$
L^{-1} \le q_{Z \given \mX}(z \given \mx) \le L.
$
Then for all $\mx,\mx'\in[0, 1]^p$ and $y,z,z'\in[0, 1]$, we have
\begin{align*}
&\lvert q_{Y \given \mX, Z}(y\given \mx, z) 
 - q_{Y \given \mX, Z}(y\given \mx', z')\rvert\\
&=\Big\lvert \frac{q_{Y, Z \given \mX}(y,z\given \mx)}{q_{Z \given \mX}(z \given \mx)}
 - \frac{q_{Y, Z \given \mX}(y,z'\given \mx')}{q_{Z \given \mX}(z' \given \mx')}\Big\rvert\\
&\le \frac{q_{Y, Z \given \mX}(y,z\given \mx)\, \big\lvert q_{Z \given \mX}(z \given \mx) - q_{Z \given \mX}(z' \given \mx')\big\rvert
+
q_{Z \given \mX}(z \given \mx)\, \big\lvert q_{Y, Z \given \mX}(y,z\given \mx) - q_{Y, Z \given \mX}(y,z'\given \mx')\big\rvert}
{q_{Z \given \mX}(z \given \mx)q_{Z \given \mX}(z' \given \mx')}\\
&\le L^3 \Big( \big\lvert q_{Z \given \mX}(z \given \mx) - q_{Z \given \mX}(z' \given \mx')\big\rvert
+\big\lvert q_{Y, Z \given \mX}(y,z\given \mx) - q_{Y, Z \given \mX}(y,z'\given \mx')\big\rvert \Big) \\
&\le L^3 \Big( \big\lvert q_{Z \given \mX}(z \given \mx) - q_{Z \given \mX}(z' \given \mx)\big\rvert + \big\lvert q_{Z \given \mX}(z' \given \mx) - q_{Z \given \mX}(z' \given \mx')\big\rvert \\
&\qquad+\big\lvert q_{Y, Z \given \mX}(y,z\given \mx) - q_{Y, Z \given \mX}(y,z'\given \mx)\big\rvert + \big\lvert q_{Y, Z \given \mX}(y,z'\given \mx)- q_{Y, Z \given \mX}(y,z'\given \mx')\big\rvert \Big) \\
&\le 2L^4(\lvert z-z'\rvert^s + \lVert \mx - \mx'\rVert_{}^s),
\end{align*}
which concludes the proof.
\end{proof}

\subsubsection{Proof of Corollary~\ref{crl:fruit}}

\begin{proof}[Proof of Corollary~\ref{crl:fruit}]

We will use the shorthand notation
\[
\xi_n^{(b)} \equiv \xi_n\Big(\big[(\mX_i,Y_i^{(b)},\mZ_i)\big]_{i=1}^{n}\Big),
~~~~
\xi_n^\circ \equiv \xi_n\Big(\big[(\mX_i,Y_i,\mZ_i)\big]_{i=1}^{n}\Big),
~~~~\text{and}~~~~
\Xi_n^{(b)}:=\ind\Big(\xi_n^{(b)}\ge \xi_n^\circ\Big).
\]
We have
\[
p_{\rm CRT}=(1+B)^{-1}+(1+B)^{-1}\sum_{b=1}^{B}\Xi_n^{(b)}. 
\yestag\label{eq:pprimefruit}
\]
Notice that $\P_{\mX_i,Y_i^{(b)},\mZ_i}\in\P_{0}(L,s)$
and $\P_{\mX_i,Y_i,\mZ_i}\in\cP_{1}(L,s)$ for $i\in\zahl{n}$. 
Moreover, Lemma~\ref{lem:condmore} 
implies that all distributions in $\cP_{[0,1]^{p+2},\infty}(L,s)$ and $\cQ_{[0,1]^{p+2},\infty}(L,s)$ are such that the assumptions
in Proposition~\ref{prop:ACHolder} are satisfied. 
Since $\xi(\P_{1,n})\gtrsim n^{-s/(p+1)+\delta}$, 
we have
\[
n^{s/(p+1)-\delta/2}\xi_n^{(b)}
\stackrel{\sf p}{\longrightarrow} 0~~~~\text{and}~~~~
n^{s/(p+1)-\delta/2}\xi_n^\circ
\stackrel{\sf p}{\longrightarrow} \infty.
\]
Accordingly, we obtain 
\[
\Xi_n^{(b)}=\ind\Big\{n^{s/(p+1)-\delta/2}\xi_n^{(b)}\ge 
                      n^{s/(p+1)-\delta/2}\xi_n^\circ\Big\}
\stackrel{\sf p}{\longrightarrow}0.
\yestag\label{eq:eachb_local}
\]


Recall that $B=B_n$ tends to infinity as $n\to\infty$. 
Since $\Xi_n^{(1)}\stackrel{\sf p}{\longrightarrow}0$, 
using Theorem~5.7 in \citet[Chap.~3]{MR3701383} yields
that every subsequence $\{n'\}$ contains a further subsequence $\{n''\}$ 
for which $\Xi_{n''}^{(1)}\stackrel{\sf a.s.}{\longrightarrow}0$. 
In view of the Proof of Proposition~\ref{prop:vacon},
we obtain
\[
B_{n''}^{-1}\sum_{b=1}^{B_{n''}}\Xi_{n''}^{(b)}
\stackrel{\sf a.s.}{\longrightarrow} 0,
\]
and then using Theorem~5.7 in \citet[Chap.~3]{MR3701383} once again gives
\[
B_{n}^{-1}\sum_{b=1}^{B_{n}}\Xi_{n}^{(b)}
\stackrel{\sf p}{\longrightarrow} 0.
\]
Moreover, we deduce from \eqref{eq:pprimefruit} that
\[
p_{\rm CRT}\stackrel{\sf p}{\longrightarrow} 0.
\]
Therefore, 
\[
\lim_{n\to\infty} \P_{H_{1,n}}(\sT^{\Q,\xi_n}_\alpha=1)
=\lim_{n\to\infty} \P_{H_{1,n}}(p_{\rm CRT}\le \alpha)
=1,
\]
and we have completed the proof.
\end{proof}

\subsubsection{Proof of Theorem \ref{thm:neybov}}

We revisit the example considered in the Proof of Theorem 4.2 in \citet[Sec.~B]{MR4319245}. 

\begin{example}\label{ex:neybov}
Under the null hypothesis $H_0$ we specify the distribution $\P_0$ with density 
\[
q_{X,Y,Z}(x,y,z) = 1
\yestag\label{eq:nullmodel}
\]
for all $(x,y,z) \in [0,1]^{3}$. 
Under the local alternative hypothesis $H_{1,n}$ we specify the distribution $\P_{1,n}$ with density
\begin{align*}
q_{X,Y,Z}(x,y,z) = 1 + \gamma_{\rho,m'}(y,z)\eta_{\rho,m}(x),
\yestag\label{eq:localmodel}
\end{align*}
where  
\[
\eta_{\rho,m}(x) = \rho \sum_{k \in \zahl{m}} \nu_{k} h_{k, m}(x)
~~~~\text{and}~~~~
\gamma_{\rho,m'}(y,z) =  \rho^2 \sum_{i \in \zahl{m'}}\sum_{j \in \zahl{m'}} \delta_{ij} h_{i, m'}(y)h_{j, m'}(z).
\]
In the definition of these functions $\rho > 0$ is a constant, $m$ and $m'$ are positive integers, 
$\nu_{k},k\in\zahl{m}$ and $\delta_{ij},i,j\in\zahl{m'}$ are i.i.d.~Rademacher random variables (taking values of $1$ and $-1$ with
probability $1/2$ each), 
and 
\[
h_{k, m}(x) = \sqrt{m}\times h(m x - k + 1)~~~~\text{for}~x \in [(k-1)/m, k/m]
\] 
(and is zero elsewhere), 
where $h$ is an infinitely differentiable function supported on $[0,1]$ 
such that 
\[
\int h(x) \d x = 0~~~~\text{and}~~~~\int h^2(x) \d x = 1.
\yestag\label{eq:normalization}
\] 
Let $c_1=\int |h(x)| \d x$ and $c_{\infty}=\max(\lVert h\rVert_{\infty},\lVert h'\rVert_{\infty})$. In addition, we assume that
\[
h(x)=-h(1-x)~~~~\text{and}~~~~
\int_{0}^{1}\int_{0}^{x} h(u) \d u\d x = 0.
\yestag\label{eq:specific}
\]
In order for the joint density \eqref{eq:localmodel} to be bonafide, it suffices to have that $\rho^3 \sqrt{m}^{}\sqrt{(m')^2} c_{\infty}^{3} \le 1$. Furthermore, we require that
\[
\rho^3 \sqrt{m}^{}\sqrt{(m')^2} c_{\infty}^{3} \to 0,~~~~\text{as}~n\to\infty.
\yestag\label{eq:perturbation}
\]
\end{example}

\vspace{5mm}

It is simple to check the following lemma for Example~\ref{ex:neybov}.

\begin{lemma}\label{lem:tech7}
The marginal of $(X,Y)$ is uniformly distributed on $[0, 1]^2$ for both $\P_{X,Y,Z}=\P_0$ and $\P_{X,Y,Z}=\P_{1,n}$. 
\end{lemma}

Moreover, we are able to compute the population Chatterjee's correlation for $\P_{X,Y,Z}=\P_{1,n}$. 

\begin{lemma}\label{lem:tech8}
Under the local alternative, the population Azadkia--Chatterjee
correlation
is
\[
\xi(\P_{1,n})
=6\rho^6 m\int_{0}^{1} \{H(u)\}^2 \d u,
\yestag\label{eq:truevalue}
\]
where $H(x):=\int_{0}^{x} h(u) \d u$.
\end{lemma}

To prepare the proof of Claim~\ref{thm:neybov1}, consider 
\[
\widehat\xi_n
:=\frac{n^{-2}\sum_{i=1}^{n}\{\min(R_i,R_{M(i)})-\min(R_i,R_{N(i)})\}}
{\int\E[\Var\{\ind(Y\ge t)\given \mX\}]\d\P_{Y}(t)},
\yestag\label{eq:hatestimate}
\]
for which we have the following lemma. 
\begin{lemma}\label{lem:tech11}
Let $\widehat\xi_{n}(\P_{1,n})$ and $\widehat\xi_{n}(\P_0)$ denote the estimates given by \eqref{eq:hatestimate} under the local alternative hypothesis 
and null hypothesis, respectively. 
Then
\begin{align*}
&\Big\lvert \E\{\widehat\xi_{n}(\P_{1,n})\} 
          - \E\{\widehat\xi_{n}(\P_0)\} \Big\rvert \lesssim \rho^6 m,\\
\text{and}~~~~
&\Big\lvert \E[n\{\widehat\xi_{n}(\P_{1,n})\}^2] 
          - \E[n\{\widehat\xi_{n}(\P_0)\}^2] \Big\rvert \lesssim n\rho^6 m.
\yestag\label{eq:meanvar2}
\end{align*}
\end{lemma}

We also have the following lemma for $\widehat\xi_n$:
\begin{lemma}\label{lem:tech12}
Under the local alternative hypothesis, assuming $\rho^2 m \lesssim n^{-1}$, it holds that
\[
\sqrt{n}\xi_n - \sqrt{n}\widehat\xi_n = o_{\P}(1).
\yestag\label{eq:slu}
\]
\end{lemma}

In the sequel, we set 
\[
m = m' = \lceil K_1 n^{2/(4s+3)}\rceil, ~~~~ \rho^3 = K_2 n^{-(2s+3)/(4s+3)},
\yestag\label{eq:mrho-1}
\]
where $K_1$ and $K_2$ are positive absolute constants satisfying
\[
K_2 K_1^{3/2+s} = L/(8c_{\infty}^3)
~~~~\text{and}~~~~
K_2 K_1^{3/2}   = 12\Zeta/c_1^3. 
\yestag\label{eq:mrho-2}
\]
We obtain the following result. 
\begin{lemma}\label{lem:tech9}
Assumption~\eqref{eq:perturbation} holds and $\P_0\in \cP_{0}(L,s),\P_{1,n}\in \cP_1(\Zeta n^{-2s/(4s+3)};L,s)$.
\end{lemma}

\smallskip
Now it is ready to prove Theorem~\ref{thm:neybov}. 


\begin{proof}[Proof of Theorem~\ref{thm:neybov}\ref{thm:neybov1}]

In view of the proof of Theorem~\ref{thm:powerless}\ref{thm:pl1}, we
obtain from \eqref{eq:meanvar2} that
\[
n\E[\{\widehat\xi_{n}(\P_{1,n})\}^2]
\to\frac45
+\frac25\Big\{\kq_{2}+\kq_{1}\Big\}
+\frac45\Big\{\ko_{2}+\ko_{1}\Big\}=:\sigma^2.
\yestag\label{eq:secondmoment}
\]
Using \eqref{eq:yexu1} and \eqref{eq:yexu2} once again, we obtain
\begin{align*}
\limsup_{n\to\infty} \P_{H_{1,n},\Q}\Big(\sT^{\Q,\xi_n}_\alpha=1\Big)
&=\limsup_{n\to\infty} \P_{H_{1,n},\Q}\Big(\sqrt{n}\xi_{n}(\P_{1,n}) > \sqrt{n}\xi_n^{[1+B-\lfloor\alpha(1+B)\rfloor]}\Big)\\
&=\limsup_{n\to\infty} \P_{H_{1,n},\Q}\Big(\sqrt{n}\xi_{n}(\P_{1,n}) > \Phi_{\sigma}^{-1}(1-\alpha)\Big).
\end{align*}
Combining \eqref{eq:slu} and \eqref{eq:secondmoment} yields
\begin{align*}
\limsup_{n\to\infty} \P_{H_{1,n},\Q}\Big(\sqrt{n}\xi_{n}(\P_{1,n}) > \Phi_{\sigma}^{-1}(1-\alpha)\Big)
&=\limsup_{n\to\infty} \P_{H_{1,n},\Q}\Big(\sqrt{n}\widehat\xi_{n}(\P_{1,n}) > \Phi_{\sigma}^{-1}(1-\alpha)\Big)\\
&\le \lim_{n\to\infty} \frac{\E[n\{\widehat\xi_{n}(\P_{1,n})\}^2]}{\{\Phi_{\sigma}^{-1}(1-\alpha)\}^2}
 =\frac{1}{\{\Phi^{-1}(1-\alpha)\}^2}=:\beta_{\alpha}.
\yestag\label{eq:mafan}
\end{align*}
It is easy to check $\beta_{\alpha} < 1$ for any $\alpha < 0.1$. 
This completes the proof for $\xi_n$. 
\end{proof}


\begin{proof}[Proof of Theorem~\ref{thm:neybov}\ref{thm:neybov2}]
We first prove
\[
\lim_{\Zeta\to 0}\lim_{n\to\infty}\TV(H_{1,n}(\Zeta), H_0)=0.
\]
We use the relation \citep[Equation~(2.27)]{MR2724359}
\[
\TV\Big(H_{1,n}(\Zeta), H_0\Big)
\le\Big\{\chi^2\Big(H_{1,n}(\Zeta), H_0\Big)\Big\}^{1/2}
\]
where the $\chi^2(\Q, \P):=\int(\d\Q/\d\P-1)^2\d\P$ denotes the chi-square distance between $\Q$ and $\P$. In view of the Proof of Theorem 4.2 in \citet[Appendix~B]{MR4319245}, we have that
\[
\chi^2\Big(H_{1,n}(\Zeta), H_0\Big)\le C_0(n \rho^6)^2 m (m')^2, ~~~~\text{for}~(n \rho^6)^2 m (m')^2\le \frac12,
\]
where $C_0$ is some absolute constant. Using \eqref{eq:mrho-1} and \eqref{eq:mrho-2}, we have
\[
(n \rho^6)^2 m (m')^2 \to \frac{\{12\Zeta/c_1^3\}^{4+3/s}}{\{L/(8c_{\infty}^3)\}^{3/s}},
\]
where the right-hand side tends to $0$ as $\Zeta\to0$. 
This concludes the proof. 

\smallskip Next, we prove the existence of $\sT^{\rm bin}$. 
The testing strategy is exactly the same as described in Section 5.3 of \citet{MR4319245} (see also Section 5.2). The only difference lies in the binning size. 
In detail, let
\[d=\ell_1=\ell_2=\lceil  n^{2/(4s+3)}\rceil.\]
We partition the support $[0,1]$ into bins $\{\cC_j\}_{j=1}^{d}$, where $\cC_j:=[(j-1)/d,j/d)$. 
Using a Poisson random variable $N$ with mean $n/2$, we will accept
the null hypothesis if $N > n$, and if $N\le n$ we draw without replacement a random sample $\cS$ of size $N$ from $\zahl{n}$. 
The set of discretized observations 
is defined as $\{(X_i, Y_i', Z_i')\}_{i\in \cS}$, 
where $Y_i':= j$ iff $Y_i\in \cC_j$ 
  and $Z_i':= j$ iff $Z_i\in \cC_j$ for $j\in\zahl{d}$. 
Let $\cD_j:=\{(Y_i', Z_i')\}_{X_i\in\cC_j, i\in \cS}$, 
and let $\sigma_j$ be the cardinality of $\cD_j$. 
For brevity, suppose that $\cD_j$ can be reindexed as $\{(Y_k', Z_k')\}_{k\in\zahl{\sigma_j}}$.
For $\sigma_j\ge 4$, assume that $\sigma_j = 4 + 4t_j$ for some $t_j\in \N$. 
Define $t_{1,j} := \min(t_j,\ell_1)$ and $t_{2,j} := \min(t_j,\ell_2)$. 
Next we split $\cD_j$ into three datasets of sizes $t_{1,j}$, $t_{2,j}$, and $2t_j + 4$ as below: 
$\cD_{j,\,Y'}=\{Y_k'\}_{k=1}^{t_{1,j}}$, 
$\cD_{j,\,Z'}=\{Z_k'\}_{k=t_{1,j}+1}^{t_{1,j}+t_{2,j}}$,
and $\cD_{j,\,Y',Z'}=\{(Y_k',Z_k')\}_{k=2t_j+1}^{\sigma_j}$,
and we are able to compute $U_j:=U_{W}(\cD_j)$ defined in
Equation~(5.5) of \citet{MR4319245} for each $\cD_j$ with at
least four observations.  The statistic $U_j$ is a weighted U-statistic that has
similarities with a Pearson $\chi^2$-statistic for testing
independence of $Y'_j$ and $Z'_j$ based on $\cD_j$.
We can now compute the test statistic defined in Equation~(5.6) of \citet{MR4319245}:
\[T=\sum_{j\in \zahl{d}}\ind(\sigma_j\ge4)\sigma_j\omega_j U_j,\]
where $\omega_j=\sqrt{\min(\sigma_j,\ell_1)\min(\sigma_j,\ell_2)}$.  Finally, we define the test $\sT^{\rm bin}$ as
\[\sT^{\rm bin}=\ind(N\le n) \ind(T\ge \zeta\sqrt{ d})\]
where $\zeta$ is a sufficiently large absolute constant.

The proof of our claim about $\sT^{\rm bin}$ proceeds in parallel to the proof
of Theorems~5.5 and 5.6 in \citet{MR4319245}.  As we detail
now, the following main differences arise.
We start from the proof of their Theorem~C.5 (which implies Theorem 5.5).
Lemmas C.7--C.9 therein still hold. 
Suppose that
\[
q\in\cP_{1}(L,s)~~~~\text{and}~~~
\inf_{q^{0} \in \cP_{0}} \lVert q - q^{0}\rVert_1\ge \epsilon:=\Zeta n^{-2s/(4s+3)}.
\]
For the expression in (C.13) in \citet{MR4319245}, we have that
\[
\sum_{j\in\zahl{d}}\frac{\epsilon_j}{\sqrt{\ell_1\ell_2}}\alpha_j
\ge \frac{n}{\sqrt{\ell_1\ell_2}}\Big(\frac{\epsilon}{2}-\frac{3}{2}\frac{L}{d^s}\Big)
=:  \frac{n}{\sqrt{\ell_1\ell_2}}\eta.
\]
We will assume that $\Zeta\ge 3L$ such that
\[\epsilon\ge {3L}/{d^s}~~~~\text{and}~~~~\eta\ge \epsilon/4.\] 
It may be readily checked that the following conditions hold:
\begin{align}
\frac{n \eta^2}{\sqrt{\ell_1\ell_2}} 
&\ge C_1\Zeta^2\sqrt{ d},\label{eq:epsilon:condition:one}\\
\max\Big(\frac{n^{3/2}\eta^2}{ \ell_1\sqrt{\ell_2}\frac{ \sqrt{n\ell_1}}{\ell_2}},
 \frac{n^{3/2}\eta^2}{ \ell_1\sqrt{\ell_2}\sqrt{d}}\Big) 
&\ge C_2\Zeta^2\sqrt{ d},\\
\frac{n^2 \eta^2}{36 \ell_1\ell_2d} 
&\ge C_3\Zeta^2\sqrt{ d},\\
\text{and}~~~~
\frac{n^4 \eta^4 }{16\ell_1 \ell_2  d^3} 
&\ge C_4\Zeta^2\sqrt{ d},
\label{eq:epsilon:condition:four}
\end{align}
where $C_1,C_2,C_3,C_4$ are absolute constants. 

Denote $\sigma = \{\sigma_j\}_{j \in \zahl{d}}$ and 
$R = \{\{\cD_{j,Y'},\cD_{j,Z'}\}\}_{j \in \zahl{d}}$. 
We have the following results. 

\begin{lemma}\label{lem:alter_variance}
Suppose $\P_{X,Y,Z}\in\cP_1(\epsilon;L,s)$, where $\epsilon \geq
3{L}/{d^s}$ and  Conditions
\eqref{eq:epsilon:condition:one}--\eqref{eq:epsilon:condition:four}
hold. Then with probability at least $1-\beta/2$ over $\sigma, R$ we
have for some absolute constants $C_5,C_6$ depending on $L,s$ that
\[\E[T | \sigma, R] \ge C_5\Zeta^2\sqrt{ d}\]
and 
\[
\Var[T | \sigma, R] \leq C_6(d + (\sqrt{d} + 1)\E[T | \sigma, R] + \E[T | \sigma, R]^{3/2}).
\]
\end{lemma}

\begin{lemma}\label{lem:null_variance}
Suppose $\P_{X,Y,Z}\in\cP_0(L,s)$, and
\[\ell_1 \geq \ell_2
~~~~\text{and}~~~
d \ell_1 \lesssim n.
\]
Then with probability at least $1-\alpha/2$ we have for some absolute
constants $C_7,C_8$ depending on $L,s$ that
\[\E[T | \sigma, R] \leq \frac{C_7 n}{d^{2s}\sqrt{\ell_1\ell_2}}\] 
and 
\[
\Var[T | \sigma, R] \leq C_8(d + (\sqrt{d} + 1)\E[T | \sigma, R] + \E[T | \sigma, R]^{3/2}).
\]
\end{lemma}

The proofs of Lemmas~\ref{lem:alter_variance} and \ref{lem:null_variance} closely follow that of Lemmas~C.10 and C.11 of \citet{MR4319245}. Using Lemmas~\ref{lem:alter_variance} and \ref{lem:null_variance} yields the following results, which are the revised version of Lemmas~C.13 and C.12, respectively. 

\begin{lemma}\label{lem:alternative:hypothesis}
Suppose $\P_{X,Y,Z}\in\cP_1(\epsilon;L,s)$, where $\epsilon \geq  3{L}/{d^s}$ and  Conditions \eqref{eq:epsilon:condition:one}--\eqref{eq:epsilon:condition:four} hold.  Then for a small enough absolute constant $C_9$ depending on $L,s,\beta$ we have that
\[
\P\Big(T \leq C_9\Zeta^2\sqrt{ d}\Big) \leq \beta. 
\]
\end{lemma}

\begin{lemma}\label{lem:null:hypothesis} 
If $\P_{X,Y,Z} \in \cP_0(L,s)$ and 
\[
\frac{n}{d^{2s}\sqrt{\ell_1\ell_2}} \asymp \sqrt{d},
\]
then for a sufficiently large absolute constant $C_{10}$ depending on $L,s,\alpha$, we have
\[
\P\Big(T \geq  \frac{C_{10} n}{d^{2s}\sqrt{\ell_1\ell_2}}\Big) \leq \alpha.
\]
\end{lemma}

Finally, we choose $\zeta=C_{10}$ such that
\[C_9\Zeta^2\sqrt{ d}\ge \zeta\sqrt{ d} \asymp \frac{C_{10} n}{d^{2s}\sqrt{\ell_1\ell_2}}\]
for all sufficiently large $\Zeta$. The rest of the proof is analogous
to the steps in the proof of Theorems~5.6 and C.15 in
\citet{MR4319245}, and hence omitted. 
\end{proof}

\subsection{Proofs for Section~\ref{sec:main-proof}}

\subsubsection{Proof of Lemma~\ref{lem:tech4}}

\begin{proof}[Proof of Lemma~\ref{lem:tech4}]
Using Lemma~D.1, a H\'ajek representation theorem, in \citet{deb2020kernel}, we have
\begin{align*}
&n^{-2}\sum_{i=1}^{n}\min(R_i,R_{M(i)}) -
\{n^2(n-1)\}^{-1}\sum_{i\ne j}\min(R_i,R_{j})\\
&=n^{-1}\sum_{i=1}^{n}\min\{F_{Y}(Y_i),F_{Y}(Y_{M(i)})\} -
\{n(n-1)\}^{-1}\sum_{i\ne j}\min\{F_{Y}(Y_i),F_{Y}(Y_j)\}+o_{\P}(1),
\yestag\label{eq:hajek1}
\end{align*}
and, similarly,
\begin{align*}
&n^{-2}\sum_{i=1}^{n}\min(R_i,R_{N(i)}) - 
\{n^2(n-1)\}^{-1}\sum_{i\ne j}\min(R_i,R_{j})\\
&=n^{-1}\sum_{i=1}^{n}\min\{F_{Y}(Y_i),F_{Y}(Y_{N(i)})\} -
\{n(n-1)\}^{-1}\sum_{i\ne j}\min\{F_{Y}(Y_i),F_{Y}(Y_j)\}+o_{\P}(1).
\yestag\label{eq:hajek2}
\end{align*}
Combining \eqref{eq:hajek1} and \eqref{eq:hajek2} yields the desired result. 
\end{proof}

\subsubsection{Proof of Lemma~\ref{lem:tech1}}

\begin{proof}[Proof of Lemma~\ref{lem:tech1}]
Let us construct an undirected graph $\cG^{\rm Dep}_n$ depending on $\cG_n$ as follows: For any $i\ne j$, we connect vertices $i$ and $j$ in $\cG^{\rm Dep}_n$ if and only if there is a path of length $1$ or $2$ joining $(\mX_i,\mZ_i)$ and $(\mX_j,\mZ_j)$ in $\cG_n$
or in $\cG^{\mX}_n$.
As illustrated in the Proof of Theorem~4.1 in \citet[Appendix~C.6]{deb2020kernel}, $\cG^{\rm Dep}_n$ is a dependency graph with maximum degree $\lesssim (\kC_{p+q}+\kC_{p})^2$, where $\kC_{p}\le \kC_{p+q}$ by the definition of $\kC_{p}$ as given in Lemma~\ref{lem:MD}. 
Applying the Berry--Esseen theorem for dependency graphs \citep[Theorem~2.7]{MR2073183} to $\cG^{\rm Dep}_n$ yields the desired result \eqref{eq:BE-no}. 
\end{proof}

\subsubsection{Proof of Lemma~\ref{lem:tech5+2}}

\begin{proof}[Proof of Lemma~\ref{lem:tech5+2}]
We first prove 
\[
\sum_{i=1}^{n}\E(V_i^2\given \cF_n)
\stackrel{\sf a.s.}{\longrightarrow}
\sum_{i=1}^{n}\E(V_i^2).
\yestag\label{eq:ES1-1}
\]
We only prove
\[
n^{-1}\sum_{\substack{(i,j)~\text{distinct} \\ i\to j\in\cE(\cG_n)\cap \cE(\cG^{\mX}_n)}}1
\stackrel{\sf a.s.}{\longrightarrow}
\E\Big\{n^{-1}\sum_{\substack{(i,j)~\text{distinct} \\ i\to j\in\cE(\cG_n)\cap \cE(\cG^{\mX}_n)}}1
\Big\},
\yestag\label{eq:ES1eg-no}
\]
and the other summands in \eqref{eq:Gamma1-no} can be handled similarly. 
Define 
\[
U_n\equiv U_n\Big(\big[(\mX_i,\mZ_i)\big]_{i=1}^{n}\Big):=
n^{-1}\sum_{\substack{(i,j)~\text{distinct} \\ i\to j\in\cE(\cG_n)\cap \cE(\cG^{\mX}_n)}}1.
\]
Let $(\bm{\tX}_1,\bm{\tZ}_1),\ldots,(\bm{\tX}_n,\bm{\tZ}_n)$ be independent copies 
of $(\mX_1,\mZ_1),\ldots,(\mX_n,\mZ_n)$. 
Set $\mX_i^{(j)}:=\mX_i$ if $i\neq j$ and $\mX_i^{(j)}:=\bm{\tX}_i$ if $i=j$, 
$\mZ_i^{(j)}:=\mZ_i$ if $i\neq j$ and $\mZ_i^{(j)}:=\bm{\tZ}_i$ if $i=j$, 
and 
\[
U_n^{(j)}:= U_n\Big(\big[(\mX_i^{(j)},\mZ_i^{(j)})\big]_{i=1}^{n}\Big).
\]
In view of Proof of Proposition~3.2(ii) in
\citet[Appendix~C.3]{deb2020kernel}, there exists a constant $C_{p+q;p}$ depending only on $\kC_{p+q}$ and $\kC_{p}$ such that
\[
\Big\lvert U_n-U_n^{(j)}\Big\rvert\le \frac{C_{p+q;p}}{n}.
\]
Using a generalized Efron--Stein inequality \citep[Theorem~2]{MR2123200} with $q = 4$ and Jensen's inequality, we have
\begin{align*}
\sum_{n=1}^{\infty}\E \Big[\Big\lvert U_n-\E(U_n)\Big\rvert^4\Big]
&\le \kappa_4^4\sum_{n=1}^{\infty} \E\Big[\Big\lvert\E\Big\{\sum_{j=1}^{n}\Big(U_n-U_n^{(j)}\Big)^2\Biggiven\cF_n\Big\}\Big\rvert^2\Big]\\
&\le \kappa_4^4\sum_{n=1}^{\infty} \E\Big[\Big\lvert\sum_{j=1}^{n}\Big(U_n-U_n^{(j)}\Big)^2\Big\rvert^2\Big]
 \le \kappa_4^4\sum_{n=1}^{\infty} \frac{C^{4}_{p+q;p}}{n^2}
 <\infty. 
\end{align*}
Combining Markov's inequality and the Borel--Cantelli lemma yields \eqref{eq:ES1eg-no}.

Next we prove 
\[
\sum_{i\ne j}\E(V_iV_j\given \cF_n)
\stackrel{\sf a.s.}{\longrightarrow}
\sum_{i\ne j}\E(V_iV_j).
\yestag\label{eq:ES1-2}
\] 
We only prove
\[
n^{-1}\sum_{\substack{(i,j)~\text{distinct} \\ j\to i\in \cE(\cG_n)}}
\gamma_{4;a,b}^*(\mX_i,\mZ_i)
\stackrel{\sf a.s.}{\longrightarrow}
\E\Big\{n^{-1}\sum_{\substack{(i,j)~\text{distinct} \\ j\to i\in \cE(\cG_n)}}
\gamma_{4;a,b}^*(\mX_i,\mZ_i)\Big\},
\yestag\label{eq:ES1eg}
\]
and the other summands in \eqref{eq:Gamma2-no} can be handled similarly.
Define 
\[
T_n\equiv T_n\Big(\big[(\mX_i,\mZ_i)\big]_{i=1}^{n}\Big):=
n^{-1}\sum_{\substack{(i,j)~\text{distinct} \\ j\to i\in \cE(\cG_n)}}
\gamma_{4;a,b}^*(\mX_i,\mZ_i).
\]
In view of Proof of Proposition~3.2(ii) in
\citet[Appendix~C.3]{deb2020kernel}, there exists a constant $C_{p+q}$ depending only on $\kC_{p+q}$ such that
\[
\lvert T_n-T_n^{(j)}\rvert\le \frac{C_{p+q}}{2n}\Big\{\max_{1\le i\le n}\Big\lvert\gamma_{4;a,b}^*(\mX_i,\mZ_i)\Big\rvert+\max_{1\le i\le n}\Big\lvert\gamma_{4;a,b}^*(\mX_i^{(j)},\mZ_i^{(j)})\Big\rvert\Big\}.
\]
Using a generalized Efron--Stein inequality \citep[Theorem~2]{MR2123200} with $q = 4$ and Jensen's inequality, we have
\begin{align*}
 \sum_{n=1}^{\infty}\E \Big[\Big\lvert T_n-\E(T_n)\Big\rvert^4\Big]
&\le\kappa_4^4\sum_{n=1}^{\infty} \E\Big[\Big\lvert\E\Big\{\sum_{j=1}^{n}\Big(T_n-T_n^{(j)}\Big)^2\Biggiven\cF_n\Big\}\Big\rvert^2\Big]
 \le\kappa_4^4\sum_{n=1}^{\infty} \E\Big[\Big\lvert\sum_{j=1}^{n}\Big(T_n-T_n^{(j)}\Big)^2\Big\rvert^2\Big]\\
&\le\kappa_4^4\sum_{n=1}^{\infty} \frac{C^{4}_{p+q}}{n^2}\E\Big\{\max_{1\le i\le n}\Big\lvert\gamma_{4;a,b}^*(\mX_i,\mZ_i)\Big\rvert^4\Big\}
 <\infty,
\end{align*}
where the last step applies bounds on the expectation of the maximum of random variables (see \citet{gilstein1981bounds} or \citet[Theorem~3]{MR813239})
to Assumption~\ref{asp:only}\ref{asp:5}. Combining Markov's inequality and the Borel--Cantelli lemma yields \eqref{eq:ES1eg}. 
Putting \eqref{eq:ES1-1} and \eqref{eq:ES1-2} together yields \eqref{eq:ES1-no}. 
\end{proof}

\subsubsection{Proof of Lemma~\ref{lem:tech6}}

\begin{proof}[Proof of Lemma~\ref{lem:tech6}]
We will follow the ideas of Proof of Theorem~2 in \citet{MR937563} and that of Theorem~1.4 in \citet{MR914597}.

{\bf Claim \eqref{eq:kr}.}
It suffices to show that
\[
\limsup_{n\to\infty}\E\Big(n^{-1}
\sum_{\substack{(i,j)~\text{distinct} \\ i\to j\in\cE(\cG_n)\cap \cE(\cG^{\mX}_n)}}1\Big)=0.
\]
We have $i\to j\in\cE(\cG_n)\cap \cE(\cG^{\mX}_n)$ if and only if the set
$S^*(\mW_i,\mW_j)$
contains no point in $[\mW_k]_{k=1}^n$ other than $\mW_i,\mW_j$, where
\[
S^*(\mw_1,\mw_2):=
\Big\{
\mw:\lVert \mw - \mw_1\rVert < \lVert \mw_2 - \mw_1\rVert
\Big\}
\;\cup\;
\Big\{
(\mx,\mz):\lVert \mx - \mx_1\rVert < \lVert \mx_2 - \mx_1\rVert
\Big\},
\]
and $\mx_i$ is the sub-vector consisting of the first $p$ elements of $\mw_i$, and 
$\mz_i$ is the sub-vector consisting of the last $q$ elements of $\mw_i$. 

For every $\epsilon,\delta>0$, we can partition $\R^{p+q}$ into a set
$G_{\epsilon,\delta}$ and its complement $H_{\epsilon,\delta}$, where $G_{\epsilon,\delta}$ is the collection of all $\mw_1$ for which $\lVert\mw_1\rVert\le1/\delta$ and
\[
\Big\lvert f(\mw_2)-f(\mw_1)\Big\lvert\le \epsilon f(\mw_1)
\]
for all $\lVert\mw_2-\mw_1\rVert\le\delta$. 
By the continuity of density $f$, it is possible to find $\delta > 0$ depending upon $\epsilon$, such that 
\[
\mu(H_{\epsilon,\delta})<\epsilon
~~~~\text{and}~~~~
G_{\epsilon,\delta}=\{\mw_1:\lVert\mw_1\rVert\le1/\delta\}.
\yestag\label{eq:BadGood}
\]
We pick $\delta$ in this manner, and write $G_{\epsilon}$ and $H_{\epsilon}$ hereafter if no confusion arises.

The data points are partitioned into two sets, according to membership
in $G_{\epsilon}$, or its complement. The number of edges in
$\cE(\cG_n)\cap \cE(\cG^{\mX}_n)$ can be written as $N_{G_\epsilon}+
N_{H_\epsilon}$, where $N_{G_\epsilon}$ refers to the edges in which
both vertices are in $G_{\epsilon}$, and $N_{H_\epsilon}$  refers to
the other edges. The expected number of $N_{H_\epsilon}$ is bounded by
$2\kC_{p+q}$ times the $\text{expected number of vertices in }H_{\epsilon}$ (because every vertex has degree no larger than $2\kC_{p+q}$). Then 
\[\E(n^{-1}N_{H_\epsilon})\le 2\kC_{p+q}\mu(H_{\epsilon}) <2\kC_{p+q}\epsilon.\] 

Recall that $B(\mw_1,r)$ denotes the ball of radius $r$ centered at $\mw_1$, 
and $\lambda(\cdot)$ denotes the Lebesgue measure. 
We let
$\mu(\cdot)$ denote the probability measure of a set, i.e., the integral of $f$ over the set, 
and $V_d$ is the volume of the unit ball in $\R^d$. 
Define
\begin{align*}
S(\mw_1,\mw_2;\Theta)&:=
\Big\{
(\mx,\mz):\lVert \mx - \mx_1\rVert_{} < \lVert \mx_2 - \mx_1\rVert_{}
,~
\lVert \mz - \mz_1\rVert_{} < \Theta\lVert \mw_2 - \mw_1\rVert_{}
\Big\}
\end{align*}
for $\Theta>0$. It is clear that $$B(\mw_1,\lVert\mw_2-\mw_1\rVert_{}),S(\mw_1,\mw_2;\Theta)\;\subseteq\; S^*(\mw_1,\mw_2).$$ 

For any fixed and arbitrarily small $\epsilon>0$, fix a $\delta>0$ such that \eqref{eq:BadGood} holds, and consider $N_{G_\epsilon}$ (the edges in which both vertices are in $G_{\epsilon}$). 
In what follows, we write 
$r_{12}:=\lVert\mw_2-\mw_1\rVert_{}$, 
$r_{12}^{\mx}:=\lVert\mx_2-\mx_1\rVert_{}$, and 
$r_{12}^{\mz}:=\lVert\mz_2-\mz_1\rVert_{}$,
where $\mw_i=(\mx_i,\mz_i)$ for simplicity. 
We observe that
\begin{align*}
\E(n^{-1}N_{G_\epsilon})
&\le (n-1)\iint_{\mw_1,\mw_2\in G_{\epsilon}}\exp[-(n-2)\mu\{S^*(\mw_1,\mw_2)\}]f(\mw_1)f(\mw_2)\d\mw_2\d\mw_1\\
&\le (n-1)\iint_{\mw_1,\mw_2\in G_{\epsilon}: r_{12}\le\delta_n}\exp[-(n-2)\mu\{S^*(\mw_1,\mw_2)\}]f(\mw_1)f(\mw_2)\d\mw_2\d\mw_1\\
&\qquad +(n-1)\iint_{\mw_1,\mw_2\in G_{\epsilon}: r_{12}>\delta_n}\exp[-(n-2)\mu\{S^*(\mw_1,\mw_2)\}]f(\mw_1)f(\mw_2)\d\mw_2\d\mw_1\\
&\le (n-1)\iint_{\substack{\mw_1,\mw_2\in G_{\epsilon}: r_{12}\le\delta_n,\\ r_{12}^{\mx}\le\theta_n r_{12}}}\exp[-(n-2)\mu\{S^*(\mw_1,\mw_2)\}]f(\mw_1)f(\mw_2)\d\mw_2\d\mw_1\\
&\qquad +(n-1)\iint_{\substack{\mw_1,\mw_2\in G_{\epsilon}: r_{12}\le\delta_n,\\ r_{12}^{\mx}>\theta_n r_{12}}}\exp[-(n-2)\mu\{S^*(\mw_1,\mw_2)\}]f(\mw_1)f(\mw_2)\d\mw_2\d\mw_1\\
&\qquad +(n-1)\iint_{\mw_1,\mw_2\in G_{\epsilon}: r_{12}>\delta_n}\exp[-(n-2)\mu\{S^*(\mw_1,\mw_2)\}]f(\mw_1)f(\mw_2)\d\mw_2\d\mw_1\\
&=: I+I\!I+I\!I\!I,~~~~\text{say,}
\yestag\label{eq:san}
\end{align*}
where $\delta_n,\theta_n$ will be specified later. 

As long as $(1+\theta_n)\delta_n<\delta$, 
\begin{align*}
I
&\le (n-1)\iint_{\substack{\mw_1,\mw_2\in G_{\epsilon}: r_{12}\le\delta_n,\\ r_{12}^{\mx}\le\theta_n r_{12}}}\exp[-(n-2)\mu\{B(\mw_1, r_{12})\}]f(\mw_1)f(\mw_2)\d\mw_2\d\mw_1\\
&\le (n-1)\iint_{\substack{\mw_1,\mw_2\in G_{\epsilon}: r_{12}\le\delta_n,\\ r_{12}^{\mx}\le\theta_n r_{12}}}\exp\{-(n-2)(1-\epsilon)f(\mw_1)V_{p+q} r_{12}^{p+q}\}f(\mw_1)f(\mw_2)\d\mw_2\d\mw_1\\
&\le \Big[(n-1)\iint_{\substack{\mw_1,\mw_2\in G_{\epsilon}: r_{12}\le\delta_n,\\ r_{12}^{\mx}\le\theta_n r_{12}}}\exp\{-2(n-2)(1-\epsilon)f(\mw_1)V_{p+q} r_{12}^{p+q}\}f(\mw_1)f(\mw_2)\d\mw_2\d\mw_1\Big]^{1/2}\\
&\qquad\Big[(n-1)\iint_{\substack{\mw_1,\mw_2\in G_{\epsilon}: r_{12}\le\delta_n,\\ r_{12}^{\mx}\le\theta_n r_{12}}}1 f(\mw_1)f(\mw_2)\d\mw_2\d\mw_1\Big]^{1/2}\\
&\le \Big[(n-1)\iint_{\substack{\mw_1,\mw_2\in G_{\epsilon}: r_{12}\le\delta_n}}\exp\{-2(n-2)(1-\epsilon)f(\mw_1)V_{p+q} r_{12}^{p+q}\}f(\mw_1)f(\mw_2)\d\mw_2\d\mw_1\Big]^{1/2}\\
&\qquad\Big[(n-1)\int_{\mw_1\in G_\epsilon} \mu(\{\mw_2=(\mx_2,\mz_2): r_{12}\le\delta_n,~ r_{12}^{\mx}\le\theta_n\delta_n\})f(\mw_1)\d\mw_1\Big]^{1/2}\\
&\le \Big[(n-1)\iint_{\substack{\mw_1,\mw_2\in G_{\epsilon}: r_{12}\le\delta}}\exp\{-2(n-2)(1-\epsilon)f(\mw_1)V_{p+q} r_{12}^{p+q}\}f(\mw_1)f(\mw_2)\d\mw_2\d\mw_1\Big]^{1/2}\\
&\qquad\Big[(n-1)\int_{\mw_1\in G_\epsilon} \mu(\{\mw_2=(\mx_2,\mz_2): r_{12}^{\mz}\le\delta_n,~ r_{12}^{\mx}\le\theta_n\delta_n\})f(\mw_1)\d\mw_1\Big]^{1/2}\\
&\le \Big[o(1)+\frac{1+\epsilon}{1-\epsilon}\frac{(n-1)}{2(n-2)}\Big]^{1/2}
\Big[(n-1)(1+\epsilon)\max_{\mw_1\in G_\epsilon}f(\mw_1)V_q(\delta_n)^q V_p(\theta_n\delta_n)^p \Big]^{1/2},
\end{align*}
where in the last step, the first term is covered by the Proof of
Theorem~2 in \citet{MR937563}, and the latter term is handled by the fact $r_{12}\le r_{12}^{\mx}+r_{12}^{\mz}=(1+\theta_n)\delta_n<\delta$. 

For the second summand in \eqref{eq:san}, if $0<\Theta_n<\delta/\delta_n-1$ and $\theta_n^{p}\Theta_n^{q}=\Omega$ for some constant $\Omega>0$, then
\begin{align*}
I\!I
&\le (n-1)\iint_{\substack{\mw_1,\mw_2\in G_{\epsilon}: r_{12}\le\delta_n,\\ r_{12}^{\mx}>\theta_n r_{12}}}\exp[-(n-2)\mu\{S(\mw_1,\mw_2;\Theta_n)\}]f(\mw_1)f(\mw_2)\d\mw_2\d\mw_1\\
&\le (n-1)\iint_{\substack{\mw_1,\mw_2\in G_{\epsilon}: r_{12}\le\delta_n,\\ r_{12}^{\mx}>\theta_n r_{12}}}\exp\{-(n-2)(1-\epsilon)f(\mw_1)\\
&\mkern300mu V_p(\theta_n r_{12})^pV_q(\Theta_n r_{12})^q\}f(\mw_1)f(\mw_2)\d\mw_2\d\mw_1\\
&= (n-1)\iint_{\substack{\mw_1,\mw_2\in G_{\epsilon}: r_{12}\le\delta_n,\\ r_{12}^{\mx}>\theta_n r_{12}}}\exp\{-(n-2)(1-\epsilon)f(\mw_1)\Omega V_pV_q r_{12}^{p+q}\}f(\mw_1)f(\mw_2)\d\mw_2\d\mw_1\\
&\le (n-1)\iint_{\substack{\mw_1,\mw_2\in G_{\epsilon}: r_{12}\le\delta_n}}\exp\{-(n-2)(1-\epsilon)f(\mw_1)\Omega V_pV_q r_{12}^{p+q}\}f(\mw_1)f(\mw_2)\d\mw_2\d\mw_1\\
&\le (n-1)\iint_{\substack{\mw_1,\mw_2\in G_{\epsilon}: r_{12}\le\delta}}\exp\{-(n-2)(1-\epsilon)f(\mw_1)\Omega V_pV_q r_{12}^{p+q}\}f(\mw_1)f(\mw_2)\d\mw_2\d\mw_1\\
&\le o(1)+\frac{1+\epsilon}{1-\epsilon}\frac{(n-1)}{(n-2)}\frac{V_{p+q}}{\Omega V_pV_q},
\end{align*}
where the last step is  by the Proof of Theorem~2 in \citet{MR937563}.

Lastly, if $n(\delta_n)^{p+q}\ge a n^{b}$ for some constants $a,b>0$ and all sufficiently large $n$, then for all $n$ large enough,
\begin{align*}
I\!I\!I
&\le(n-1)\iint_{\mw_1,\mw_2\in G_{\epsilon}: r_{12}>\delta_n}\exp[-(n-2)\mu\{B(\mw_1, r_{12})\}]f(\mw_1)f(\mw_2)\d\mw_2\d\mw_1\\
&\le(n-1)\iint_{\mw_1,\mw_2\in G_{\epsilon}: r_{12}>\delta_n}\exp\{-(n-2)(1-\epsilon)f(\mw_1)V_{p+q}(\delta_n)^{p+q}\}f(\mw_1)f(\mw_2)\d\mw_2\d\mw_1\\
&\le(n-1)\int_{\mw_1\in G_{\epsilon}}\exp\{-(n-2)(1-\epsilon)f(\mw_1)V_{p+q}(\delta_n)^{p+q}\}f(\mw_1)\d\mw_1\\
&\le o(1).
\end{align*}
Here the proof of the last step is similar to that of Theorem~2 in \citet{MR937563}.

Taking
$\delta_n$, $\theta_n$, and $\Theta_n$ 
such that
\[
\delta_n^{p+3q/2}=n^{-1},
~~~~
\theta_n^{p}=\epsilon^{-1}\delta_n^{q},
~~~~\text{and}~~~~
\Theta_n=\delta/(2\delta_n),
\]
we deduce the result.

{\bf Claim \eqref{eq:ks}.}
It suffices to show that
\[
\limsup_{n\to\infty}\E\Big(n^{-1}
\sum_{\substack{(i,j)~\text{distinct} \\ i\to j\in \cE(\cG_n), j\to i\in \cE(\cG^{\mX}_n)}}1\Big)=0.
\]
We have $i\to j\in\cE(\cG_n)$ and $j\to i\in\cE(\cG^{\mX}_n)$ if and only if the set
$S^{**}(\mW_i,\mW_j)$
contains no point in $[\mW_k]_{k=1}^n$ other than $\mW_i,\mW_j$, where
\[
S^{**}(\mw_1,\mw_2):=
\Big\{
\mw:\lVert \mw - \mw_1\rVert_{} < \lVert \mw_2 - \mw_1\rVert_{}
\Big\}
\bigcup
\Big\{
(\mx,\mz):\lVert \mx - \mx_2\rVert_{} < \lVert \mx_2 - \mx_1\rVert_{}
\Big\},
\]
and $\mx_i$ is the sub-vector consisting of the first $p$ elements of $\mw_i$,  and
$\mz_i$ is the sub-vector consisting of the last $q$ elements of $\mw_i$. 
The rest of the proof will be in analogy to that of Claim I.

{\bf Claim \eqref{eq:kt}.}
We have
\begin{align*}
&\E\Big(n^{-1}
\sum_{\substack{(i,j,k)~\text{distinct} \\ i\to k\in \cE(\cG_n), j\to k\in \cE(\cG^{\mX}_n)}}1\Big)\\
&=\E\Big(n^{-1}
\sum_{\substack{(i,j,k)~ i\ne k, j\ne k \\ i\to k\in \cE(\cG_n), j\to k\in \cE(\cG^{\mX}_n)}}1\Big)
-\E\Big(n^{-1}
\sum_{\substack{(i,k)~\text{distinct} \\ i\to k\in \cE(\cG_n)\cap \cE(\cG^{\mX}_n)}}1\Big)\\
&=\E\Big(n^{-1}\sum_{j=1}^{n}d_j^{-}d_j^{\mX-}\Big)
-\E\Big(n^{-1}
\sum_{\substack{(i,k)~\text{distinct} \\ i\to k\in \cE(\cG_n)\cap \cE(\cG^{\mX}_n)}}1\Big)
\end{align*}
where $d^{-}_j := \sum_{i:i\to j\in \cE(\cG_n)} 1$ is the in-degree of
vertex $(\mX_j,\mZ_j)$ in $\cG_n$, and
$d^{\mX-}_j$ is the in-degree of vertex $\mX_j$ in $\cG^{\mX}_n$. 
It suffices to show that
\[
\lim_{n\to\infty}\P(d^{-}_j=k,d^{\mX-}_j=\ell)=\kp_{p+q,k}\kp_{p,\ell}
\]
for all fixed $j$, where $\kp_{d;k}$ is as defined in Lemma~\ref{lem:henze}. 
For the sake of notational simplicity, let $\mW_0=(\mX_0,\mZ_0)$
be an additional random variable with density $f$, independent of
$\mW_1,\dots,\mW_n$.  Let $A_j$ denote the event that $\mW_0$ is the nearest neighbor of $\mW_j$, 
and let $B_j$ denote the event that $\mX_0$ the nearest neighbor of
$\mX_j$, $1\le j\le n$.  Put
\[
d^{-}_0=\sum_{j=1}^{n}\ind(A_j)~~~~\text{and}~~~~
d^{\mX-}_0=\sum_{j=1}^{n}\ind(B_j).
\]
Jordan's formula, generalized to two systems of events, yields that
\[
\P(d^{-}_0=k,d^{\mX-}_0=\ell)
=\sum_{u=0}^{\kC_{p+q}-k}\sum_{v=0}^{\kC_p-\ell}(-1)^{u+v}\binom{u+k}{k}\binom{v+\ell}{\ell}\beta_{u+k,v+\ell},
\]
for $0\le k\le\kC_{p+q}$, $0\le\ell\le\kC_p$, 
where
\begin{align*}
&\beta_{0,0}=1,~~~
\beta_{r,0}=\sum_{1\le i_1<\cdots<i_{r}\le n}\P\Big(A_{i_1}\cap\cdots\cap A_{i_r}\Big),~~~
\beta_{0,s}=\sum_{1\le j_1<\cdots<j_{s}\le n}\P\Big(B_{j_1}\cap\cdots\cap B_{j_s}\Big),\\
&\beta_{r,s}=\sum_{1\le i_1<\cdots<i_{r}\le n}\sum_{1\le j_1<\cdots<j_{s}\le n}\P\Big(A_{i_1}\cap\cdots\cap A_{i_r}\cap B_{j_1}\cap\cdots\cap B_{j_s}\Big),
\end{align*}
for $r,s\ge1$. 
We have $\beta_{r,0}\to \kd_{p+q,r}/r!$ and $\beta_{0,s}\to \kd_{p,s}/s!$ by (3.1) in \citet{MR914597}. 
Using symmetry 
gives
\[
\beta_{r,s}=\sum_{t=0}^{\min(r,s)}\binom{n}{t}\binom{n-t}{r-t}\binom{n-r}{s-t}\P\Big(A_{1}\cap\cdots\cap A_{r}\cap B_{1}\cap\cdots\cap B_{t}\cap B_{r+1}\cdots\cap B_{r+s-t}\Big).
\]
It remains to prove
\[
n^{r+s}\P\Big(A_{1}\cap\cdots\cap A_{r}\cap B_{r+1}\cdots\cap B_{r+s}\Big)\to \kd_{p+q,r}\kd_{p,s}
\]
and
\[
n^{r+s-t}\P\Big(A_{1}\cap\cdots\cap A_{r}\cap B_{1}\cap\cdots\cap B_{t}\cap B_{r+1}\cdots\cap B_{r+s-t}\Big)\to 0
\]
for all $1\le t\le \min(r,s)$. Straightforward arguments paralleling
those of (3.9) and (3.12) in \citet{MR914597} yield these desired results. 
Finally,
\[
\E\Big(n^{-1}\sum_{j=1}^{n}d_j^{-}d_j^{\mX-}\Big)
=\E[d_1^{-}d_1^{\mX-}]
\to\sum_{k=0}^{\kC_{p+q}}\sum_{\ell=0}^{\kC_p}k\ell \kp_{p+q,k}\kp_{p,\ell}
=\sum_{k=0}^{\kC_{p+q}}k\kp_{d,k}\sum_{\ell=0}^{\kC_p}\ell\kp_{p,\ell}
=1.
\]
Combining with Claim \eqref{eq:kr} concludes the proof. 
\end{proof}

\subsection{Proofs for the supplement}

\subsubsection{Proof of Lemma~\ref{lem:tech7}}

\begin{proof}[Proof of Lemma~\ref{lem:tech7}]
It is clear the marginal of $(X,Y)$ is uniformly distributed on $[0, 1]^2$ for $\P_{X,Y,Z}=\P_0$. When $\P_{X,Y,Z}=\P_{1,n}$, we have for any fixed sequence $\nu:=[\nu_k]_{k \in \zahl{m}}$ and $\delta=[\delta_{ij}]_{i,j \in \zahl{m'}}$, where $\nu_k,\delta_{ij}\in\{-1,+1\}$, 
\begin{align*}
  q_{X, Y}(x, y) 
&= \int_{0}^{1} q_{X,Y,Z}(x,y,z) \d z 
 = 1 +  \eta_{\rho,m}(x)\int_{0}^{1} \gamma_{\rho,m'}(y,z) \d z \\
&= 1 +  \eta_{\rho,m}(x)\times\rho^2 \sum_{i \in \zahl{m'}}\sum_{j \in \zahl{m'}} \delta_{ij} h_{i, m'}(y)\Big\{\int_{0}^{1} h_{j, m'}(z)\d z\Big\} 
 = 1.
\yestag\label{eq:jointxy}
\end{align*}
Taking expectation over all Rademacher sequences $\nu$ and $\delta$ completes the proof.
\end{proof}

\subsubsection{Proof of Lemma~\ref{lem:tech8}}

\begin{proof}[Proof of Lemma~\ref{lem:tech8}]
Recall that
\[
\xi=1-\frac{\int\E[\Var\{\ind(Y\ge y)\given Z,X\}]\d \P_{Y}(y)}
           {\int\E[\Var\{\ind(Y\ge y)\given   X\}]\d \P_{Y}(y)}.
\yestag\label{eq:Xizhi0}
\]
We first compute $\xi$ for any fixed sequence $\nu:=[\nu_k]_{k \in \zahl{m}}$ and $\delta=[\delta_{ij}]_{i,j \in \zahl{m'}}$, where $\nu_k,\delta_{ij}\in\{-1,+1\}$. 
Here, we have $q_{Y\given X}(y\given x)=1$ by \eqref{eq:jointxy}. 
We also have $q_{X,Z}(x,z)=1$ similarly as for \eqref{eq:jointxy}, and thus
\[
q_{Y\given Z,X}(y\given z,x) 
 = 1 + \gamma_{\rho,m'}(y,z)\eta_{\rho,m}(x).
\]
Accordingly, we obtain
\begin{align*}
&\int\E[\Var\{\ind(Y\ge y)\given X\}]\d \P_{Y}(y)\\
&=\int\E[\Var\{\ind(Y\le y)\given X\}]\d \P_{Y}(y)\\
&=\int\E[\P(Y\le y\given X) - \{\P(Y\le y\given X)\}^2]\d u_{Y}(y)\\
&=\int_{0}^{1}\int_{0}^{1}\Big[\int_{0}^{y}q_{Y\given X}(t\given x)\d t - \Big\{\int_{0}^{y}q_{Y\given X}(t\given x)\d t\Big\}^2\Big]\d x\d y\\
&=\int_{0}^{1}\int_{0}^{1}\Big(y - y^2\Big)\d x\d y=\frac12-\frac13=\frac16.\yestag\label{eq:Xizhi1}
\end{align*}
We also obtain that
\begin{align*}
&\int\E[\Var\{\ind(Y\ge y)\given Z,X\}]\d u_{Y}(y)\\
&=\int\E[\Var\{\ind(Y\le y)\given Z,X\}]\d u_{Y}(y)\\
&=\int\E[\P(Y\le y\given Z,X) - \{\P(Y\le y\given Z,X)\}^2]\d u_{Y}(y)\\
&=\int_{0}^{1}\int_{0}^{1}\int_{0}^{1}\Big[\int_{0}^{y}q_{Y\given (Z,X)}(t\given z,x)\d t - \Big\{\int_{0}^{y}q_{Y\given (Z,X)}(t\given z,x)\d t\Big\}^2\Big]\d z\d x\d y,
\end{align*}
where
\begin{align*}
&\int_{0}^{1}\int_{0}^{1}\int_{0}^{1}\Big[\int_{0}^{y}q_{Y\given Z,X}(t\given z,x)\d t\Big]\d z\d x\d y\\
&=\int_{0}^{1}\int_{0}^{1}\int_{0}^{1}\Big[y+\int_{0}^{y}\gamma_{\rho,m'}(t,z)\eta_{\rho,m}(x) \d t\Big]\d z\d x\d y\\
&=\int_{0}^{1}\int_{0}^{1}\int_{0}^{1}y \d z\d x\d y
 =\frac12,
\end{align*}
and
\begin{align*}
&\int_{0}^{1}\int_{0}^{1}\int_{0}^{1}\Big[\Big\{\int_{0}^{y}q_{Y\given Z,X}(t\given z,x)\d t\Big\}^2\Big]\d z\d x\d y\\
&=\int_{0}^{1}\int_{0}^{1}\int_{0}^{1}\Big[\Big\{y+\int_{0}^{y}\gamma_{\rho,m'}(t,z)\eta_{\rho,m}(x) \d t\Big\}^2\Big]\d z\d x\d y\\
&=\int_{0}^{1}\int_{0}^{1}\int_{0}^{1}\Big[y^2+2y\Big\{\int_{0}^{y}\gamma_{\rho,m'}(t,z)\eta_{\rho,m}(x) \d t\Big\}+\Big\{\int_{0}^{y}\gamma_{\rho,m'}(t,z)\eta_{\rho,m}(x) \d t\Big\}^2\Big]\d z\d x\d y\\
&=\frac13+\int_{0}^{1}\int_{0}^{1}\int_{0}^{1}\Big[\Big\{\int_{0}^{y}\gamma_{\rho,m'}(t,z)\eta_{\rho,m}(x) \d t\Big\}^2\Big]\d z\d x\d y.
\end{align*}
Hence,
\begin{align*}
&\int\E[\Var\{\ind(Y\ge y)\given Z,X\}]\d \P_{Y}(y)\\
&=\frac16-\int_{0}^{1}\int_{0}^{1}\int_{0}^{1}\Big[\Big\{\int_{0}^{y}\gamma_{\rho,m'}(t,z)\eta_{\rho,m}(x) \d t\Big\}^2\Big]\d z\d x\d y\\
&=\frac16-\int_{0}^{1}\int_{0}^{1}\int_{0}^{1}\Big[\Big\{\eta_{\rho,m}(x) \int_{0}^{y}\gamma_{\rho,m'}(t,z) \d t\Big\}^2\Big]\d z\d x\d y\\
&=\frac16-\int_{0}^{1}\int_{0}^{1}\int_{0}^{1}\Big[\Big\{\eta_{\rho,m}(x) \int_{0}^{y}\rho^2 \sum_{i \in \zahl{m'}}\sum_{j \in \zahl{m'}} \delta_{ij} h_{i, m'}(t)h_{j, m'}(z) \d t\Big\}^2\Big]\d z\d x\d y\\
&=\frac16-\int_{0}^{1}\int_{0}^{1}\int_{0}^{1}\Big[\Big\{\eta_{\rho,m}(x) \times \rho^2 \sum_{i \in \zahl{m'}}\sum_{j \in \zahl{m'}} \delta_{ij} H_{i, m'}(y)h_{j, m'}(z) \Big\}^2\Big]\d z\d x\d y\\
&=\frac16-\int_{0}^{1}\int_{0}^{1}\int_{0}^{1}\Big[\Big\{\rho \sum_{k \in \zahl{m}} \nu_{k} h_{k, m}(x)\times \rho^2 \sum_{i \in \zahl{m'}}\sum_{j \in \zahl{m'}} \delta_{ij} H_{i, m'}(y)h_{j, m'}(z) \Big\}^2\Big]\d z\d x\d y,
\end{align*}
where
\begin{align*}
&H_{k, m}(x) 
 := \int_{0}^{x} h_{k, m}(t) \d t\\
& = \int_{0}^{x} \sqrt{m}\times h(m t - k + 1) \d t\\
& = \int_{0}^{mx} m^{-1/2}\times h(u - k + 1) \d u\\
& = m^{-1/2}\int_{-k+1}^{mx-k+1} h(u) \d u\\
& = \begin{cases} 
m^{-1/2}\int_{0}^{mx-k+1} h(u) \d u & ~~~\text{if } x\in[(k-1)/m, k/m],\\
0                                   & ~~~\text{otherwise}
\end{cases}\\
& = \begin{cases} 
H_{k,m}(x)                          & ~~~\text{if } x\in[(k-1)/m, k/m],\\
0                                   & ~~~\text{otherwise},
\end{cases}
\end{align*}
where 
\[
H_{k, m}(x):=m^{-1/2} H(mx-k+1)~~~\text{ and }~~~H(x):=\int_{0}^{x} h(u) \d u.
\]

It is easy to verify that $h_{k,m}(x)h_{k',m}(x)=0$ and $H_{k,m}(x)H_{k',m}(x)=0$ for $k\ne k'$. Thus we deduce
\begin{align*}
&\Big\{\rho \sum_{k \in \zahl{m}} \nu_{k} h_{k, m}(x)\times \rho^2 \sum_{i \in \zahl{m'}}\sum_{j \in \zahl{m'}} \delta_{ij} H_{i, m'}(y)h_{j, m'}(z) \Big\}^2 \\
& = \rho^6 \sum_{k \in \zahl{m}} \sum_{i \in \zahl{m'}}\sum_{j \in \zahl{m'}} \{h_{k, m}(x)\}^2 \{H_{i, m'}(y)\}^2 \{h_{j, m'}(z)\}^2,
\end{align*}
and accordingly
\begin{align*}
&\int\E[\Var\{\ind(Y\ge y)\given Z,X\}]\d \P_{Y}(y)\\
&=\frac16-\int_{0}^{1}\int_{0}^{1}\int_{0}^{1}\Big[\rho^6 \sum_{k \in \zahl{m}} \sum_{i \in \zahl{m'}}\sum_{j \in \zahl{m'}} \{h_{k, m}(x)\}^2 \{H_{i, m'}(y)\}^2 \{h_{j, m'}(z)\}^2\Big]\d z\d x\d y\\
&=\frac16-\rho^6 \sum_{k \in \zahl{m}} \sum_{i \in \zahl{m'}}\sum_{j \in \zahl{m'}} \Big[\int_{0}^{1}\{h_{k, m}(x)\d x\}^2\Big] \Big[\int_{0}^{1}\{H_{i, m'}(y)\}^2\d y\Big] \Big[\int_{0}^{1}\{h_{j, m'}(z)\}^2\d z\Big],
\end{align*}
where
\begin{align*}
\int_{0}^{1} \{h_{k, m}(x)\}^2 \d x 
&= \int_{0}^{1} m\times \{h(m x - k + 1)\}^2 \d x
 = \int_{0}^{m} \{h(u - k + 1)\}^2 \d u\\
&= \int_{- k + 1}^{m - k + 1} \{h(u)\}^2 \d u
 = \int_{0}^{1} \{h(u)\}^2 \d u
 = 1.
\end{align*}
Similarly,
\begin{align*}
\int_{0}^{1} \{H_{k, m}(x)\}^2 \d x 
&= \int_{0}^{1} m^{-1}\{H(m x - k + 1)\}^2 \d x
 = \int_{0}^{m} m^{-2}\{H(u - k + 1)\}^2 \d u\\
&= \int_{- k + 1}^{m - k + 1} m^{-2}\{H(u)\}^2 \d u
 = m^{-2}\int_{0}^{1} \{H(u)\}^2 \d u.
\end{align*}
Therefore,
\begin{align*}
&\int\E[\Var\{\ind(Y\ge y)\given Z,X\}]\d \P_{Y}(y)\\
&=\frac16-\rho^6 \sum_{k \in \zahl{m}} \sum_{i \in \zahl{m'}}\sum_{j \in \zahl{m'}} \Big[\int_{0}^{1}\{h_{k, m}(x)\d x\}^2\Big] \Big[\int_{0}^{1}\{H_{i, m'}(y)\}^2\d y\Big] \Big[\int_{0}^{1}\{h_{j, m'}(z)\}^2\d z\Big]\\
&=\frac16-\rho^6 m m' m' (m')^{-2}\int_{0}^{1} \{H(u)\}^2 \d u\\
&=\frac16-\rho^6 m\int_{0}^{1} \{H(u)\}^2 \d u.
\yestag\label{eq:Xizhi2}
\end{align*}
Plugging \eqref{eq:Xizhi1} and \eqref{eq:Xizhi2} into \eqref{eq:Xizhi0} yields that
\[
\xi(\P_{1,n})
=6\rho^6 m\int_{0}^{1} \{H(u)\}^2 \d u.
\]
Taking expectation over all Rademacher sequences $\nu:=[\nu_k]_{k \in \zahl{m}}$ and $\delta=[\delta_{ij}]_{i,j \in \zahl{m'}}$ completes the proof. 
\end{proof}

\subsubsection{Proof of Lemma~\ref{lem:tech11}}

\begin{proof}[Proof of Lemma~\ref{lem:tech11}]
Notice that
\[
\min(R_i,R_{M(i)})=1+\sum_{k\ne i, M(i)}\ind\Big\{Y_k\le \min(Y_i,Y_{M(i)})\Big\}.
\]
We write
\[
\widehat\xi_n
=\frac{n^{-2}\sum_{i=1}^{n}\{\min(R_i,R_{M(i)})-1\}
      -n^{-2}\sum_{i=1}^{n}\{\min(R_i,R_{N(i)})-1\}}
      {1/6}
=:\frac{\widehat\gamma_n-\widehat\gamma^{X}_n}{1/6}~~~~\text{say.}
\]
Since the marginal of $(X,Y)$, is uniformly distributed on $[0,1]^2$ under both null hypothesis $\P_0$ and local alternative hypothesis $\P_{1,n}$, it holds that
\[
\E\{\widehat\gamma^{X}_n(\P_{1,n})\}=\E\{\widehat\gamma^{X}_n(\P_0)\}
~~~~\text{and}~~~~
\E[\{\widehat\gamma^{X}_n(\P_{1,n})\}^2]=\E[\{\widehat\gamma^{X}_n(\P_0)\}^2].
\]
In order to prove \eqref{eq:meanvar2}, we aim to prove
\begin{align*}
&\Big\lvert \E\{\widehat\gamma_{n}(\P_{1,n})\} 
          - \E\{\widehat\gamma_{n}(\P_0)\} \Big\rvert \lesssim \rho^6 m,
\yestag\label{eq:exiexp-aim}\\
&\Big\lvert \E[n\{\widehat\gamma_{n}(\P_{1,n})\}^2] 
          - \E[n\{\widehat\gamma_{n}(\P_0)\}^2] \Big\rvert \lesssim n\rho^6 m,
\yestag\label{eq:exivar-aim}
\end{align*}
and
\begin{align*}
&\Big\lvert \E[n\{\widehat\gamma_{n}(\P_{1,n})\}\{\widehat\gamma^{X}_{n}(\P_{1,n})\}] 
          - \E[n\{\widehat\gamma_{n}(\P_0)\}\{\widehat\gamma^{X}_{n}(\P_0)\}] \Big\rvert \lesssim n\rho^6 m.
\yestag\label{eq:exicov-aim}
\end{align*}

We start from computing $\E\{\widehat\gamma_{n}(\P_{1,n})\}$ and 
                       $\E[\{\widehat\gamma_{n}(\P_{1,n})\}^2]$.
Recall our notation $\mW=(X,Z)$, $\mW_i=(X_i,Z_i)$, $\mw=(x,z)$, $\mw_i=(x_i,z_i)$, 
$r_{ij}:=\lVert\mw_j-\mw_i\rVert$, and that
$B(\mw_1,r)$ denotes the ball of radius $r$ centered at $\mw_1$, 
and $\lambda(\cdot)$ denotes the Lebesgue measure. 
We obtain
\begin{align*}
&\E\{\widehat\gamma_{n}(\P_{1,n})\}\\
&=n^{-2}\sum_{i=1}^{n}\E\Big[\sum_{k\ne i, M(i)}\ind\Big\{Y_k\le \min(Y_i,Y_{M(i)})\Big\}\Big]\\
&=n^{-1}\E\Big[\sum_{k\ne 1, M(1)}\ind\Big\{Y_k\le \min(Y_1,Y_{M(1)})\Big\}\Big]\\
&=n^{-1}(n-1)\E\Big[\sum_{k\ne 1, 2}\ind\Big\{Y_k\le \min(Y_1,Y_{2})\Big\}\ind\Big\{M(1)=2\Big\}\Big]\\
&=n^{-1}(n-1)(n-2)\E\Big[\ind\Big\{Y_3\le \min(Y_1,Y_{2})\Big\}\ind\Big\{M(1)=2\Big\}\Big]\\
&=n^{-1}(n-1)(n-2)\E\bigg(\E\Big[\ind\Big\{Y_3\le \min(Y_1,Y_{2})\Big\}\ind\Big\{M(1)=2\Big\}\Biggiven Y_1,\mW_1,Y_2,\mW_2,Y_3,\mW_3\Big]\bigg)\\
&=n^{-1}(n-1)(n-2)\E_{\nu,\delta}\bigg(\int \ind\{y_3\le \min(y_1,y_{2})\}
  q_{Y,\mW}(y_1,\mw_1)q_{Y,\mW}(y_2,\mw_2)q_{Y,\mW}(y_3,\mw_3)\\
&\qquad\times\ind(r_{13}<r_{12})
       \times[\lambda\{[0,1]^2\backslash B_{}(\mw_1,r_{12})\}]^{n-3}
\d y_3\d \mw_3\d y_2\d \mw_2\d y_1\d \mw_1\bigg)\\
&=n^{-1}(n-1)(n-2)\int \ind\{y_3\le \min(y_1,y_{2})\}
  \E_{\nu,\delta}\Big\{q_{Y,\mW}(y_1,\mw_1)q_{Y,\mW}(y_2,\mw_2)q_{Y,\mW}(y_3,\mw_3)\Big\}\\
&\qquad\times\ind(r_{13}<r_{12})
       \times[\lambda\{[0,1]^2\backslash B_{}(\mw_1,r_{12})\}]^{n-3}
\d y_3\d \mw_3\d y_2\d \mw_2\d y_1\d \mw_1,
\yestag\label{eq:exi1}
\end{align*}
where the expectation $\E_{\nu,\delta}$ is taken over all Rademacher sequences $\nu:=[\nu_k]_{k \in \zahl{m}}$ and $\delta=[\delta_{ij}]_{i,j \in \zahl{m'}}$.  
Recall that
\[
q_{X,Y,Z}(x,y,z) = 1 + \rho^3 \sum_{k \in \zahl{m}}\sum_{i \in \zahl{m'}}\sum_{j \in \zahl{m'}} \nu_{k} \delta_{ij} h_{k, m}(x)h_{i, m'}(y)h_{j, m'}(z), 
\]
and accordingly
\begin{align*}
&\E_{\nu,\delta}\Big\{q_{Y,\mW}(y_1,\mw_1)q_{Y,\mW}(y_2,\mw_2)q_{Y,\mW}(y_3,\mw_3)\Big\}\\
&=1+\rho^6 \sum_{k \in \zahl{m}} \sum_{i \in \zahl{m'}} \sum_{j \in \zahl{m'}} \{h_{k, m}(x_1)h_{k, m}(x_2)\} \{h_{i, m'}(y_1)h_{i, m'}(y_2)\} \{h_{j, m'}(z_1)h_{j, m'}(z_2)\} \\
&\qquad+\rho^6 \sum_{k \in \zahl{m}} \sum_{i \in \zahl{m'}} \sum_{j \in \zahl{m'}} \{h_{k, m}(x_1)h_{k, m}(x_3)\} \{h_{i, m'}(y_1)h_{i, m'}(y_3)\} \{h_{j, m'}(z_1)h_{j, m'}(z_3)\} \\
&\qquad+\rho^6 \sum_{k \in \zahl{m}} \sum_{i \in \zahl{m'}} \sum_{j \in \zahl{m'}} \{h_{k, m}(x_2)h_{k, m}(x_3)\} \{h_{i, m'}(y_2)h_{i, m'}(y_3)\} \{h_{j, m'}(z_2)h_{j, m'}(z_3)\}.
\end{align*}
Notice that
\begin{align*}
n^{-1}(n-1)(n-2)&\int \ind\{y_3\le \min(y_1,y_{2})\}\times\ind(r_{13}<r_{12})\\
&\qquad\times[\lambda\{[0,1]^2\backslash B_{}(\mw_1,r_{12})\}]^{n-3}
\d y_3\d \mw_3\d y_2\d \mw_2\d y_1\d \mw_1=\E\{\widehat\gamma_{n}(\P_0)\}=\frac{n-2}{3n},\\
(n-1)&\int \ind(r_{13}<r_{12}) \times[\lambda\{[0,1]^2\backslash B(\mw_1,r_{12})\}]^{n-3}
\d \mw_3\d \mw_2\d \mw_1=1,\\
(n-2)&\int \sum_{k \in \zahl{m }} \{h_{k, m }(x_1)h_{k, m }(x_2)\} 
           \sum_{j \in \zahl{m'}} \{h_{j, m'}(z_1)h_{j, m'}(z_2)\} \\
&\qquad \times\ind(r_{13}<r_{12}) \times[\lambda\{[0,1]^2\backslash B_{}(\mw_1,r_{12})\}]^{n-3} \d \mw_3\d \mw_2\d \mw_1 \lesssim mm',\\
(n-2)&\int \sum_{k \in \zahl{m }} \{h_{k, m }(x_1)h_{k, m }(x_3)\} 
           \sum_{j \in \zahl{m'}} \{h_{j, m'}(z_1)h_{j, m'}(z_3)\} \\
&\qquad \times\ind(r_{13}<r_{12}) \times[\lambda\{[0,1]^2\backslash B_{}(\mw_1,r_{12})\}]^{n-3} \d \mw_3\d \mw_2\d \mw_1 \lesssim mm',\\
(n-2)&\int \sum_{k \in \zahl{m }} \{h_{k, m }(x_2)h_{k, m }(x_3)\} 
           \sum_{j \in \zahl{m'}} \{h_{j, m'}(z_2)h_{j, m'}(z_3)\} \\
&\qquad \times\ind(r_{13}<r_{12}) \times[\lambda\{[0,1]^2\backslash B_{}(\mw_1,r_{12})\}]^{n-3} \d \mw_3\d \mw_2\d \mw_1 \lesssim mm'.
\yestag\label{eq:compu-tools1}
\end{align*}
 In addition, by \eqref{eq:specific},
\begin{align*}
&\int \ind\{y_3\le \min(y_1,y_{2})\} \sum_{i \in \zahl{m'}} h_{i, m'}(y_1)h_{i, m'}(y_2) \d y_3\d y_2\d y_1 \lesssim \frac{1}{m'},\\
&\int \ind\{y_3\le \min(y_1,y_{2})\} \sum_{i \in \zahl{m'}} h_{i, m'}(y_1)h_{i, m'}(y_3) \d y_3\d y_2\d y_1 \lesssim \frac{1}{m'},\\
&\int \ind\{y_3\le \min(y_1,y_{2})\} \sum_{i \in \zahl{m'}} h_{i, m'}(y_2)h_{i, m'}(y_3) \d y_3\d y_2\d y_1 \lesssim \frac{1}{m'}.
\yestag\label{eq:compu-tools2}
\end{align*}
Putting theses together yields
the desired result \eqref{eq:exiexp-aim}.

Now we turn to $\E[n\{\widehat\gamma_{n}(\P_{1,n})\}^2]$. We have
\begin{align*}
&\E[n\{\widehat\gamma_{n}(\P_{1,n})\}^2]\\
&=n^{-2}\E\bigg(\Big[\sum_{k\ne 1, M(1)}\ind\Big\{Y_k\le \min(Y_1,Y_{M(1)})\Big\}\Big]^2\bigg)+n^{-2}(n-1)\\
&\qquad\times\E\bigg(
\Big[\sum_{k\ne 1, M(1)}\ind\Big\{Y_k\le \min(Y_1,Y_{M(1)})\Big\}\Big]\times
\Big[\sum_{\ell\ne 2, M(2)}\ind\Big\{Y_\ell\le \min(Y_2,Y_{M(2)})\Big\}\Big]\bigg).
\yestag\label{eq:exivar-1}
\end{align*}
The first term in \eqref{eq:exivar-1} can be expanded as
\begin{align*}
&\E\bigg(\Big[\sum_{k\ne 1, M(1)}\ind\Big\{Y_k\le \min(Y_1,Y_{M(1)})\Big\}\Big]^2\bigg)\\
&=(n-1)\E\bigg(\Big[\sum_{k\ne 1, 2}\ind\Big\{Y_k\le \min(Y_1,Y_{2})\Big\}\Big]^2\ind\Big\{M(1)=2\Big\}\bigg)\\
&=(n-1)\E\bigg(\Big[\sum_{k\ne 1, 2}\sum_{\ell\ne 1, 2}\ind\Big\{\max(Y_k,Y_{\ell})\le \min(Y_1,Y_{2})\Big\}\Big]\ind\Big\{M(1)=2\Big\}\bigg)\\
&=(n-1)(n-2)\E\Big[\ind\Big\{Y_3\le \min(Y_1,Y_{2})\Big\}\ind\Big\{M(1)=2\Big\}\Big]\\
&\quad+(n-1)(n-2)(n-3)\E\Big[\ind\Big\{\max(Y_3, Y_4)\le \min(Y_1,Y_{2})\Big\}\ind\Big\{M(1)=2\Big\}\Big]\\
&=(n-1)(n-2)\int \ind\{y_3\le \min(y_1,y_{2})\}
  \E_{\nu,\delta}\Big\{q_{Y,\mW}(y_1,\mw_1)q_{Y,\mW}(y_2,\mw_2)q_{Y,\mW}(y_3,\mw_3)\Big\}\\
&\qquad\times\ind(r_{12}<r_{13})
       \times[\lambda\{[0,1]^2\backslash B_{}(\mw_1,r_{12})\}]^{n-3}
\d y_3\d \mw_3\d y_2\d \mw_2\d y_1\d \mw_1\\
&\quad+(n-1)(n-2)(n-3)\int \ind\{\max(y_3,y_4)\le \min(y_1,y_{2})\}
  \E_{\nu,\delta}\Big\{\prod_{k=1}^{4}q_{Y,\mW}(y_k,\mw_k)\Big\}\\
&\qquad\times\ind\{r_{12}<\min(r_{13},r_{14})\}
       \times[\lambda\{[0,1]^2\backslash B_{}(\mw_1,r_{12})\}]^{n-4}
\d y_4\d \mw_4\d y_3\d \mw_3\d y_2\d \mw_2\d y_1\d \mw_1.
\end{align*}
The second term in \eqref{eq:exivar-1} can be similarly expanded.
Similar computations as for \eqref{eq:compu-tools1} and \eqref{eq:compu-tools2}, 
but more involved, yield \eqref{eq:exivar-aim}.

We can similarly prove \eqref{eq:exicov-aim} and thus skip the details for brevity. 
This completes the proof. 
\end{proof}

\subsubsection{Proof of Lemma~\ref{lem:tech12}}

\begin{proof}[Proof of Lemma~\ref{lem:tech12}]
We have
\[
\sqrt{n}\xi_n - \sqrt{n}\widehat\xi_n
:=\sqrt{n}\widehat\xi_n
\Big(\frac{\int\E[\Var\{\ind(Y\ge t)\given \mX\}]\d\P_{Y}(t)}
{n^{-2}\sum_{i=1}^{n}\{R_i-\min(R_i,R_{N(i)})\}}-1\Big).
\yestag\label{eq:fenjie}
\]
For the first term in the right-hand side of \eqref{eq:fenjie}, \eqref{eq:secondmoment}  yields $\sqrt{n}\widehat\xi_n=O_{\P}(1)$. 
For the latter term, recall that \citet[Theorem~9.1]{MR4352523} show
\[
\frac{1}{n^2}\sum_{i=1}^{n}\{R_i-\min(R_i,R_{N(i)})\}
\stackrel{\sf a.s.}{\longrightarrow}
\int\E[\Var\{\ind(Y\ge t)\given \mX\}]\d\P_{Y}(t)
=\frac{1}{6}.
\]
Accordingly, the latter term in \eqref{eq:fenjie} is $o_{\P}(1)$. This concludes the proof of \eqref{eq:slu}.
\end{proof}

%
%
%
%
%
%
%
%
%
%
%
%
%
%

\subsubsection{Proof of Lemma~\ref{lem:tech9}}

\begin{proof}[Proof of Lemma~\ref{lem:tech9}]

It is clear that \eqref{eq:perturbation} holds and $\P_0\in \cP_{0}(L,s)$. 
Next, we show that $\P_{1,n}\in \cP_1(\Zeta n^{-2s/(4s+3)};L,s)$. 
To show $\P_{1,n}\in\cP_{1}(L,s)$, 
in view of the Proof of Theorem 4.2 in \citet[Appendix~B]{MR4319245}, 
it suffices to notice that
\begin{align*}
\Big\lvert
q_{Y,Z \given \mX}(y,z \given \mx) - 
q_{Y,Z \given \mX}(y,z \given \mx')
\Big\rvert
&\le \Big(\rho^3 \sqrt{m}^{}\sqrt{(m')^2} c_{\infty}^{3} \times 4 m^s \Big)
\Big(\lvert x - x'\rvert^s\Big),\\
\Big\lvert
q_{Y,Z \given \mX}(y,z   \given \mx) - 
q_{Y,Z \given \mX}(y',z' \given \mx)
\Big\rvert
&\le \Big(\rho^3 \sqrt{m}^{}\sqrt{(m')^2} c_{\infty}^{3} \times 4 (m')^s \Big)
\Big(\lvert y-y'\rvert^s+\lvert z-z'\rvert^s\Big),
\end{align*}
and $\rho^3 \sqrt{m}^{}\sqrt{(m')^2} c_{\infty}^{3} \to 0$. 
To show $\inf_{q^{0} \in \cP_{0}} \lVert q_{1,n} - q^{0}\rVert_1\ge\Zeta n^{-2s/(4s+3)}$, where $q_{1,n}$ denotes the corresponding density of $\P_{1,n}$, 
using Lemma B.4 in \citet[Appendix~B]{MR4319245} yields
\[
\inf_{q^{0} \in \cP_{0}}\lVert q_{1,n} - q^{0}\rVert_1 
\geq \frac{\lVert q_{1,n} - 1\rVert_1}{6}
=\frac{\rho^3\sqrt{m}^{} \sqrt{(m')^2} c_1^{3}}{6}
\ge \Zeta n^{-2s/(4s+3)}.
\]
This completes the proof. 
\end{proof}

\end{appendix}

{\small
\bibliographystyle{apalike}
\bibliography{AMS}
}

\end{document}